\documentclass[letter,11pt]{amsart}
\usepackage[english]{babel}
\usepackage{colonequals}
\usepackage{stmaryrd}
\usepackage{leftidx}
\usepackage{calc}
\usepackage{dsfont}
\usepackage{arydshln}

\usepackage[utf8]{inputenc}
\usepackage[english]{babel}
\usepackage{amsthm}
\usepackage{amssymb}
\usepackage{amsmath}
\usepackage[pdfauthor={Zheng Liu, Giovanni Rosso},pdftitle={Non-cuspidal Hida theory for Siegel modular forms and trivial zeros of $p$-adic $L$-functions},linktocpage,colorlinks=true,linkcolor=blue,citecolor={[rgb]{1.0,0.3,0.0}},pagebackref,breaklinks]{hyperref}

\usepackage{bm}

\usepackage{amsthm}
\usepackage{amsfonts}
\usepackage{amssymb}
\usepackage{amsmath}
\usepackage{amsthm}
\usepackage{mathrsfs}
\usepackage{hyperref}
\usepackage{graphicx}
\usepackage[all]{xy}
\usepackage{tikz-cd}
\usepackage{upgreek}

\usepackage{enumerate}
\usepackage{blkarray}
\usepackage[shortlabels]{enumitem}

\newcommand{\set}[1]{\left\lbrace #1 \right\rbrace}
\newcommand{\mb}[1]{\mathbb{#1}}
\newcommand{\mc}[1]{\mathcal{#1}}
\newcommand{\mr}[1]{\mathrm{#1}}
\newcommand{\mf}[1]{\mathfrak{#1}}
\newcommand{\ms}[1]{\mathscr{#1}}

\newcommand{\pV}{\mathpzc{V}}

\newcommand\numberthis{\addtocounter{equation}{1}\tag{\theequation}}

\newcommand{\ut}{\underline{t}}
\newcommand{\ua}{\underline{a}}
\newcommand{\ub}{\underline{b}}
\newcommand{\uc}{\underline{c}}
\newcommand{\utau}{\protect\underline{\uptau}}
\newcommand{\uep}{\protect\underline{\epsilon}}
\newcommand{\unu}{\protect\underline{\nu}}
\newcommand{\utheta}{\underline{\theta}}
\newcommand{\ol}{\overline}
\newcommand{\wh}{\widehat}
\newcommand{\mvw}{\vartheta}
\newcommand{\qexp}{q\text{-}\mathrm{exp}}
\newcommand{\wt}{\widetilde}
\newcommand{\be}{\mathbf{e}}
\newcommand{\imp}{P\text{-}\mathrm{imp}}
\newcommand{\Pord}{P\text{-}\mathrm{ord}}
\newcommand{\sC}{\mathscr{C}}
\newcommand{\sL}{\mathscr{L}}
\newcommand{\bbeta}{\bm{\beta}}
\newcommand{\bz}{\bm{z}}
\newcommand{\pW}{\mathpzc{W}}
\newcommand{\sgn}{\mathrm{sgn}}
\newcommand{\triv}{\mathrm{triv}}

\newcommand{\lra}{\longrightarrow}
\newcommand{\ra}{\rightarrow}
\newcommand{\hra}{\hookrightarrow}
\newcommand{\lhra}{\ensuremath{\lhook\joinrel\longrightarrow}}

\newcommand{\btimes}{
  \mathop{
    \vphantom{\bigoplus} 
    \mathchoice
      {\vcenter{\hbox{\resizebox{\widthof{$\displaystyle\bigoplus$}}{!}{$\boxtimes$}}}}
      {\vcenter{\hbox{\resizebox{\widthof{$\bigoplus$}}{!}{$\boxtimes$}}}}
      {\vcenter{\hbox{\resizebox{\widthof{$\scriptstyle\oplus$}}{!}{$\boxtimes$}}}}
      {\vcenter{\hbox{\resizebox{\widthof{$\scriptscriptstyle\oplus$}}{!}{$\boxtimes$}}}}
  }\displaylimits 
}

\DeclareMathAlphabet{\mathpzc}{OT1}{pzc}{m}{it}
\newtheoremstyle{note}{11pt}{11pt}{}{}{\bfseries}{.}{.5em}{}
\newtheorem*{teono}{Theorem}
\numberwithin{equation}{subsection}

\newtheorem{thm}{Theorem}[subsection]
\newtheorem{theo}[thm]{Theorem}
\newtheorem{prop}[thm]{Proposition}

\newtheorem{lemma}[thm]{Lemma}

\newtheorem{conj}[thm]{Conjecture}

\theoremstyle{remark}
\newtheorem{rem}[thm]{Remark}
\newtheorem*{remno}{Remark}

\newcommand{\diag}{\mathrm{diag}}
\newcommand{\ord}{\mathrm{ord}}

\newcommand{\uepsilon}{\protect\underline{\epsilon}}

\newcommand{\bA}{\mathbb{A}}

\newcommand{\bC}{\mathbb{C}}

\newcommand{\bG}{\mathbb{G}}
\newcommand{\bH}{\mathbb{H}}
\newcommand{\bI}{\mathbb{I}}

\newcommand{\bQ}{\mathbb{Q}}
\newcommand{\bR}{\mathbb{R}}

\newcommand{\bT}{\mathbb{T}}
\newcommand{\bU}{\mathbb{U}}
\newcommand{\bV}{\mathbb{V}}
\newcommand{\bW}{\mathbb{W}}

\newcommand{\bZ}{\mathbb{Z}}

\newcommand{\cA}{\mathcal{A}}
\newcommand{\cB}{\mathcal{B}}
\newcommand{\cC}{\mathcal{C}}
\newcommand{\cD}{\mathcal{D}}
\newcommand{\cE}{\mathcal{E}}
\newcommand{\cF}{\mathcal{F}}
\newcommand{\cG}{\mathcal{G}}

\newcommand{\cI}{\mathcal{I}}

\newcommand{\cL}{\mathcal{L}}
\newcommand{\cM}{\mathcal{M}}

\newcommand{\cO}{\mathcal{O}}
\newcommand{\cP}{\mathcal{P}}

\newcommand{\cS}{\mathcal{S}}

\newcommand{\cV}{\mathcal{V}}

\newcommand{\fc}{\mathfrak{c}}

\newcommand{\fe}{\mathfrak{e}}

\newcommand{\ffi}{\mathfrak{i}}

\newcommand{\fm}{\mathfrak{m}}

\newcommand{\fp}{\mathfrak{p}}

\newcommand{\fs}{\mathfrak{s}}

\newcommand{\fw}{\mathfrak{w}}

\newcommand{\fC}{\mathfrak{C}}

\newcommand{\sT}{\ms{T}}
\newcommand{\mM}{\mathpzc{M}}

\newcommand{\pG}{\mathpzc{G}}

\newcommand{\cont}{\mr{cont}}

\DeclareMathOperator{\GL}{GL}
\DeclareMathOperator{\SL}{SL}

\DeclareMathOperator{\Sp}{Sp}
\DeclareMathOperator{\U}{U}

\DeclareMathOperator{\Hom}{Hom}
\DeclareMathOperator{\Sym}{Sym}

\newenvironment{psm}
{\left(\begin{smallmatrix}}
{\end{smallmatrix}\right)}

\makeindex
\begin{document}

\title[Non-cuspidal Hida theory and trivial zeros]{\texorpdfstring{Non-cuspidal Hida theory for Siegel modular forms and trivial zeros of \lowercase{$p$}-adic $L$-functions}{Non-cuspidal Hida theory for Siegel modular forms and trivial zeros of p-adic L-functions}}

\author{Zheng Liu}
\email{\href{mailto:zliu@math.ucsb.edu}{zliu@math.ucsb.edu}}
\urladdr{\url{https://math.ucsb.edu/~zliu}}
\address{University of California, Santa Barbara, CA, United States}

\author{Giovanni Rosso}
\email{\href{mailto:giovanni.rosso@concordia.ca}{giovanni.rosso@concordia.ca}}
\urladdr{\url{https://sites.google.com/site/gvnros/}}
\address{Concordia University, Departments of Mathematics and Statistics,
Montreal, Qu\'ebec, Canada}


\begin{abstract}
We study the derivative of the standard $p$-adic $L$-function associated with a $P$-ordinary Siegel modular form (for $P$ a parabolic subgroup of $\GL(n)$) when it presents a semi-stable trivial zero. This implies part of Greenberg's conjecture on the order and leading coefficient of $p$-adic $L$-functions at such trivial zero. We use the method of Greenberg--Stevens. For the construction of the {\it improved} $p$-adic $L$-function we develop Hida theory for non-cuspidal Siegel modular forms.
\end{abstract}
\maketitle
\tableofcontents

\section*{Introduction}

In the seminal paper \cite{MTT} the three authors consider an elliptic curve $E$ and a prime $p$ such that $E$ has split multiplicative reduction at $p$ (for example $E=X_0(11)$ and $p=11$). In  this case the $p$-adic $L$-function $\cL_p(s,E)$ presents a trivial zero at $s=1$ because of the modified Euler factor at $p$. If the complex $L$-value $L(1,E)$ is not vanishing, they conjecture that the first derivative of the $p$-adic $L$-function at $s=1$ is the algebraic part of the complex $L$-value, up to an error factor of the form $\log_p(q_E)/\mr{ord}_p(q_E)$,  which they call the $\ell$-invariant. Here $q_E$ is the Tate period of $E$.

This conjecture has been proved in \cite{SSS} using Hida theory and a two-variable $p$-adic $L$-function together with a one-variable improved $p$-adic $L$-function. At the same time, Greenberg generalized the conjecture of Mazur--Tate--Taitelbaum to the class of $p$-adic Galois representations $V$ that satisfy the so-called Panchishkin condition. Assuming $L(0,V)\neq 0$, his conjecture, roughly speaking, predicts that the multiplicity of the trivial zero of $\cL_p(s,V)$ at $s=0$ equals the order of vanishing of $\cL_p(s,V)$ at $s=0$, and gives an exact formula for the leading coefficient of the $p$-adic $L$-function. This precise formula involves a factor $\ell(V)$, called the $\ell$-invariant of $V$, which is defined in purely Galois theoretic terms and coincides with $\log_p(q_E)/\mr{ord}_p(q_E)$ when $V$ is the Tate module of an elliptic curve $E$. This conjecture has been recently generalized to all semi-stable Galois representations (which satisfy some technical conditions) \cite{BenLinv}. For the precise statement, see Conjecture \ref{conj:greenberg}. \\

Let $n$ be an integer and let $P$ be the parabolic of $\mr{GL}(n)$ associated with the partition $n=n_1+\ldots+n_d$, \textit{i.e.}
\begin{align*}
P&=\left\{\left.\begin{pmatrix}a_1&*&*\\&\ddots&*\\&&a_d\end{pmatrix}\in\mr{GL}(n)\,\right|\,a_i\in\mr{GL}(n_i),\,1\leq i\leq d\right\}.
\end{align*}
The main objective of the paper is to study Conjecture \ref{conj:greenberg} when $V$ is the standard Galois representation associated to an irreducible cuspidal automorphic representation $\pi$ of $\mr{Sp}(2n,\bA)$ which is $P$-ordinary, {\it i.e.}  the archimedean component $\pi_\infty$ is isomorphic to a holomorphic discrete series $\cD_{\ut}$ of weight $\ut=(\underbrace{t^P_1,\dots, t^P_1}_{n_1},\underbrace{t^P_2,\dots, t^P_2}_{n_2},\dots,\underbrace{t^P_d,\dots, t^P_d}_{n_d})$, and the action of certain $\bU^P_p$-operators (which are Hecke operators at $p$ whose normalization depends on $\ut$) on $\pi$ admits a non-zero eigenvector with eigenvalues being $p$-adic units. 

Denote by $L(s,\pi \times \xi)$ the standard $L$-function for $\pi$ twisted by a finite order Dirichlet character $\xi$. It is defined as an infinite Euler product. The local $L$-factor for a place $v$ where both $\pi$ and $\xi$ are unramified is given as
\begin{equation*}
   L_v(s,\pi_v\times \xi_v)=(1-\xi_v(q_v)q_v^{-s})^{-1}\prod_{i=1}^{n}(1-\xi_v(q_v)\alpha_{v,i}q_v^{-s})^{-1}(1-\xi_v(q_v)\alpha_{v,i}^{-1}q_v^{-s})^{-1},
\end{equation*}  
where $\alpha_{v,i}^{\pm 1}$, $1\leq i\leq n$, are the Satake parameters of $\pi_v$ and $q_v$ is the cardinality of the residue field.  
The Deligne critical points for $L(s,\pi \times \xi)$ are the integers $s_0$ such that 
\begin{equation*}
   1\leq s_0\leq t_d^P-n,\,(-1)^{s_0+n}=\xi(-1),\text{ or } n+1-t_d^P\leq s_0\leq 0, \,(-1)^{s_0+n+1}=\xi(-1).
\end{equation*}
The algebracity of these critical $L$-values divided by certain Petersson norm period has been shown in \cite{Ha81,Sh00,BS}. In \cite{LiuSLF}, the first author constructed an $n+1$-variable $p$-adic $L$-function interpolating the critical values   to the right of the center of the partial standard $L$-function with $\pi$ varying in a Hida family (ordinary for $P=B$ the Borel of $\mr{GL}(n)$).\\

In this paper we generalize the results of \cite{LiuSLF} and construct a $d+1$-variable $p$-adic $L$-function for $P$-ordinary Hida families (where $P$ is general), interpolating critical values to the left of the center of the partial standard $L$-function. Moreover, we improve the formulae on the interpolation properties of the constructed $p$-adic $L$-functions by using results in \cite{LiuAZI} and functional equation for local doubling zeta integrals. The modified Euler factors at the place $p$ and the archimedean place appearing in the interpolation formulae match exactly with the prediction in \cite{CoaMot}. The $L$-functions appearing in the interpolation formula are not  the motivic $L$-functions but some imprimitive version, as the Euler factors at the bad primes are missing.  Still, our `imprimitive'  standard $p$-adic $L$-functions for $P$-ordinary Siegel modular forms of general vector weights verify the `imprimitive' version of the conjecture of Coates--Perrin-Riou.

Let $T_P=P/SP$ be the maximal quotient torus of $P$. We say that $\utau^P\in\mr{Hom}_{\mr{cont}}(T_P(\mb Z_p),\overline{\mb Q}_p^{\times})$ is arithmetic if it is a product of an algebraic character corresponding to integers $(t_1^P,\ldots,t_d^P)$ and a finite order character $(\epsilon_1^P,\ldots,\epsilon_d^P)$; we say that it is admissible if moreover $t_1^P \geq  t_2^P \ldots \geq t^P_d\geq n+1$. We fix a sufficiently large $p$-adic field $F$ as coefficient field and denote by $\mc O_F$ its valuation ring. 

Suppose $p\geq 3$. Hida theory for $P$-ordinary cuspidal Siegel modular forms has been developed in \cite{PilHida} generalizing the case $P=B$ in \cite{HPEL}. Let $\sC_P$ be a geometrically irreducible component of the spectrum of  $\bT^{0,N}_{\Pord}$, the Hecke algebra acting on $P$-ordinary Hida families of cuspidal Siegel modular forms of tame principal level $N$, and let $F_{\sC_P}$ be its function field. We denote by $\bI_{\sC_P} $ the integral closure of $\Lambda_P :=\mc O_F \llbracket T_P(\mb Z_p)^{\circ} \rrbracket$ in $F_{\sC_P}$, where $T_P(\bZ_p)^\circ$ is the maximal $p$-profinite subgroup of $T_P(\bZ_p)$. Let $\Sym(n,\bZ)^{*}$ be the subset of $\Sym(n,\bQ)$ consisting of elements $\alpha$ such that $\mr{Tr} (\alpha a) \in  \bZ$ for all symmetric matrices $a \in \Sym(n,\bZ)$. Let  $\Sym(n,\bZ)^{*}_{\geq 0}$ be the subset of semi-positive definite matrices and $\Sym(n,\bZ)^{*}_{>0}$ the subset of positive definite matrices.

We prove the following theorem: 
\begin{teono}[Theorem \ref{thm:twopadicLfun}]
Let $\sC_P$ be as above. For a Dirichlet character $\phi$ with conductor dividing $N$ and $\phi^2\neq 1$, a pair $(\beta_1,\beta_2)\in N^{-1}\Sym(n,\bZ)^{*\oplus 2}_{>0}$, and $j\in\bZ/(p-1)$ such that $\phi\omega^j(-1)=1$, there is a $p$-adic $L$-function $\cL_{\sC_P,\phi\omega^j,\beta_1,\beta_2} \in  \bI_{\sC_P}[[S]] \otimes_{\bI_{\sC_P} } F_{\sC_P}$ with the following interpolation property. 

Let $x:\bI_{\sC_P}\ra F'$ be an $F'$-point of $\sC_P$ (with $F'$ being a finite extension of $F$). Suppose that the weight  map $\Lambda_P \ra \bT^{0,N}_{\Pord}$  is \'etale at $x$  and maps $x$ to an admissible point $\utau^P\in\Hom_{\cont}\left(T_P(\bZ_p),F^{\prime\times}\right)$. For an integer $n+1\leq k\leq t^P_d$ and a finite order character $\chi^\circ:\bZ^\times_p\ra \ol{\bQ}^\times$ trivial on $(\bZ/p)^\times$, if $x$ is classical we have
\begin{align*}
 &\cL_{\sC_P,\phi\omega^j,\beta_1,\beta_2}(\chi^\circ(1+p)(1+p)^k-1,x) =C_{x,\beta_1,\beta_2,\phi,N}\\
   &\hspace{2em}\times E_p(n+1-k,\pi_x\times\phi\chi^\circ\omega^{j-k})\,E_\infty(n+1-k,\pi_x\times\phi\chi^\circ\omega^{j-k}) L^{Np\infty}(n+1-k,\pi_x\times\phi\chi^\circ\omega^{j-k}),
\end{align*}
and $0$ otherwise. 
Here the factor $E_p(n+1-k,\pi_x\times\phi\chi^\circ\omega^{j-k})$ (resp. $E_\infty(n+1-k,\pi_x\times\phi\chi^\circ\omega^{j-k})$ is the modified Euler factor at $p$ (resp. $\infty$) as predicted by Coates--Perrin-Riou \cite{CoaMot}. (See \eqref{eq:Ep}\eqref{eq:Euler-infty} for explicit formulae for these two factors in our case.) The constant $C_{x,\beta_1,\beta_2,\phi,N}$ is defined in Theorem~\ref{thm:twopadicLfun}, and is independent of the cyclotomic variable $k$, $\chi^\circ$
\end{teono}
The construction of this $p$-adic $L$-function uses the doubling method \cite{GaPull,PSR}.
\begin{remno}
For the whole paper, we assume that $\phi^2 \neq 1$. This hypothesis is absolutely not necessary, but when $\phi^2=1$ the $p$-adic $L$-function could have a possible pole in the cyclotomic variable (outside the range of interpolation), which comes from the pole of the Kubota--Leopoldt $p$-adic function appearing in the Fourier coefficients of the Siegel Eisenstein series on $\mr{Sp}(4n)$. When $n=1$ this pole cancels out if and only if $\sC_P$ has no CM \cite[Proposition 5.2]{H6}. In general we expect a cyclotomic pole if and only if the standard representation associated with $\sC_P$ is reducible and the trivial representation appears as a sub-quotient of it.
\end{remno}

When $n_d=1$ and $\epsilon^P_d$ is trivial, the factor $1-\phi_p(p)^{-1}\alpha_{n,x}^{-1}p^{s-1}$ appears in $E_p(s,\pi_x\times\phi)$, where $\alpha_{n,x}$ is an algebraic number related with the $\bU_p^P$-eigenvalues (see \S\ref{sec:Ep} for the precise definition). Supposing that $x$ is classical, then $\alpha_{n,x}$ corresponds to the Frobenius eigenvalue of $p$-adic valuation $-(t^P_d-n)$ in the Weil representation associated to $\pi_{x,p}$. If $\phi(-1)=(-1)^{n+1}$, $\phi_p(p)=1$, and $\alpha_{n,x_0}=p^{-1}$ for a classical point $x_0\in \sC_P(F)$, then the factor $1-\phi_p(p)^{-1}\alpha_{n,x_0}^{-1}p^{k-n}$ vanishes if $k=n+1$, and a trivial zero occurs at the point $((1+p)^{n+1}-1,x_0)$ for $\cL_{\sC_P,\phi\omega^{n+1},\beta_1,\beta_2}$. Denote by $\rho_{x_0}:G_{\bQ}\ra\GL(2n+1,\ol{\bQ}_p)$ the Galois representation attached to $x_0$ \cite{Arthur,CH}. We shall call $((1+p)^{n+1}-1,x_0)$ a semi-stable trivial zero for $\cL_{\sC_P,\phi\omega^{n+1},\beta_1,\beta_2}$ if furthermore $\mr{Fil}^0\left.\rho_{x_0}\right|_{G_{\bQ_p}}/\mr{Fil}^2\left.\rho_{x_0}\right|_{G_{\bQ_p}}$ is a two dimensional indecomposable $G_{\bQ_p}$-representation (see \S \ref{sec:trivzero} for the definition of this filtration). When $n_d=1$ and $\pi_{x_0}$ is $P$-ordinary with $\alpha_{n,x_0}=p^{-1}$, the condition on $\left.\rho_{x_0}\right|_{G_{\bQ_p}}$ is expected to be always satisfied (see Remark \ref{rem:mono}). It is for this special type of trivial zeros that we can use the Greenberg--Stevens method to study the derivative of the $p$-adic $L$-function.

The step of expressing the $\ell$-invariant in terms of the derivative of $\bU^P_p$-eigenvalues in the Greenberg--Stevens method for the so-called trivial zero of type $M$ (as named in \cite{TTT}) has already been done \cite{RosLinv}. The other step in the method, which  relates the derivative with respect to the cyclotomic variable of the $p$-adic $L$-function to the derivative with respect to the weight variable of the $\bU^P_p$-eigenvalues, applies in the following situation. Suppose that there is a $d+1$-variable $p$-adic $L$-function $\cL(S,T_1,\dots,T_d)$ with $S$ as the cyclotomic variable, and it has a trivial zero at $(s_0,t_1,\dots,t_d)$. If there exists a $d$-tuple integer $(a_1,\dots,a_d)\neq 0$, and integers $a_0\neq a'_0$, such that $\cL(S,T_1,\dots,T_d)$ vanishes along the line $(s_0,t_1,\dots,t_d)+S(a_0,a_1,\dots,a_d)$ and can be improved (in the sense of saving the factor that causes the trivial zero in the interpolation result) along the line $(s_0,t_1,\dots,t_d)+S(a'_0,a_1,\dots,a_d)$, then the strategy applies.

In our above mentioned case of the semi-stable zero, the assumption on $\left.\rho_{x_0}\right|_{G_{\bQ_p}}$ implies that the trivial zero is of type $M$. The $p$-adic $L$-function $\cL_{\sC_P,\phi\omega^{n+1},\beta_1,\beta_2}$ vanishes along the hyperplane $S=(1+p)^{n+1}-1$ (because of the missing factor $1-\phi_p(p)p^{-s}$ for $\pi$ with $\pi_p$ unramified). Meanwhile,  when $k=t^P_d$ the factor $1-\phi_p(p)^{-1}\alpha_n^{-1}p^{n-k}$ is a $p$-adic analytic function as $\alpha_n p^{t_d^P-n}$ can be expressed in terms of $\mb U^P_p$ eigenvalues. Hence there is the possibility to improve the $p$-adic $L$-function along the hyperplane $S=(1+p)^{t^P_d}-1$. The lines $((1+p)^{n+1}-1,x_0)+S(0,0,\dots,0,1)$ and $((1+p)^{n+1}-1,x_0)+S(1,0,\dots,0,1)$ satisfy the conditions in the previous paragraph.

Now in order to carry out the Greenberg--Stevens method, we need to construct the improved $p$-adic $L$-function. Indeed, by a different choice of the local sections at $p$ for the Siegel Eisenstein series on $\mr{Sp}(4n)$ (compare the tables in \S\ref{sec:table}), we obtain a new Eisenstein series such that applying to it the pullback formula from the doubling method produces the complex $L$-function without the factor $1-\phi_p(p)^{-1}\alpha_n^{-1}p^{s-1}$.  

However, a new difficulty arises. One useful fact about the sections selected for constructing the $p$-adic $L$-function in Theorem \ref{thm:twopadicLfun} is that the restrictions to $\mr{Sp}(2n,\bA)\times \mr{Sp}(2n,\bA)$ of the corresponding Siegel Eisenstein series are cuspidal, so Hida theory for cuspidal Siegel modular forms can be applied to finish the construction. However, the new Eisenstein series for the improved $p$-adic $L$-function does not restrict to ($p$-adic) cuspidal forms on $\mr{Sp}(2n,\bA)\times \mr{Sp}(2n,\bA)$. Therefore, Hida theory for non-cuspidal Siegel modular forms needs to be developed in order to construct the improved $p$-adic $L$-function. Such a theory has been developed for Siegel modular forms with $ P=\mr{GL}(n)$ \cite{PilHida}, and for $\mathrm{U}(2,2)$ \cite{SU} which is later generalized to $\mathrm{U}(n,1)$ \cite{Hsieh}.
In the first section we develop Hida theory for $p$-adic Siegel modular forms vanishing  along the strata of the toroidal compactification associated with cusp labels of rank strictly bigger than $r$, for an integer $ r \leq n_d$. 

Let $\omega_{\ut}$ be the coherent sheaf whose global sections are weight $t$ holomorphic Siegel modular forms. Compared to \cite{SU,Hsieh}, we bypass introducing the subsheaf $\omega^{\flat}_{\ut}$ inside $\omega_{\ut}$ and proving the base change property for its global sections. Instead, our approach is based on a careful analysis of the quotient  $\left. V^{SP,r}_{m,l}\middle/ V^{SP,r-1}_{m,l}\right.$, where $V^{SP,r}_{m,l}$ (resp. $V^{SP,r-1}_{m,l}$) denotes the space of functions on the $l$-th layer of the Igusa tower modulo $p^m$ which vanish along the strata of the toroidal compactification associated with cusp labels of rank strictly bigger than $r$ (resp. $r-1$). We denote by $\pV^{SP,r}$ the direct limit over $l$ and $m$ of $V^{SP,r}_{m,l}$. This allows us to define a useful subspace $V^{SP,r,\flat}_{m,l}\subset V^{SP,r}_{m,l}$, and after taking the limit, to establish the exact sequence 
\begin{equation}\label{eq:sesintro}
   0\lra\pV^{SP,r-1}\lra\pV^{SP,r,\flat}\lra \bigoplus_{\substack{V\in\fC_{\bV}/\Gamma\\\mr{rk}\,V=r}} \bZ_p[[T_P(\bZ_p)]]\otimes_{\bZ_p[[T_{P_{n-r}}(\bZ_p)]]} \pV^{SP_{n-r},0}_V\lra 0,
\end{equation}
from which one can establish Hida theory for $\pV^{SP,r}$ by using  cuspidal Hida theory and induction on $r$. 

The idea of using exact sequences involving non-cupidal Siegel modular forms and Siegel modular forms of lower genus also appears in \cite{WeiVek,BR}. As they work in characteristic $0$, for an irreducible algebraic representation $W$ of $\mr{GL}(n)$, and a congruence subgroup $L\subset \GL(n)$ consisting of elements of the form $\begin{pmatrix}I_{n_r}&\ast\\0&\ast\end{pmatrix}$, one has
\begin{equation}\label{eq:inv}
\begin{aligned}
W(R)^L&=W(\bQ)^L\otimes R, & \text{ for  a }\bQ\text{-algebra }R.
\end{aligned}
\end{equation}  
However, \eqref{eq:inv} fails if $\bQ$ is replaced by $\bZ_p$. The failure of equation \eqref{eq:inv}  causes the difficulty for directly generalizing the  Hida theory for cuspidal Siegel modular forms to non-cuspidal Siegel modular forms. The sheaf $\omega^{\flat}_{\ut}$ in \cite{SU, Hsieh} is about remedying the failure of \eqref{eq:inv} when $\bQ$ is replaced by $\bZ_p$. This issue is bypassed in our approach, as we study the space $\pV^{SP,r}$ via the terms on the two ends of the exact sequence \eqref{eq:sesintro}.

Our results are summarized as follows:
\begin{teono}[Theorem \ref{thm:main}]
For the given parabolic subgroup $P\subset\GL(n)$  and an integer $1\leq r\leq n_d$, the following holds:
\begin{enumerate}[(i)] 
\item An ordinary projector $e_P=e_P^2$ can be defined on $\pV^{SP,r}$, and the Pontryagin dual of its ordinary part 
\begin{equation*}
\pV^{r,*}_{\Pord}=\Hom_{\bZ_p}\left(e_P\pV^{SP,r},\bQ_p/\bZ_p\right)
\end{equation*}
(which is naturally an $\cO_F[[T_P(\bZ_p)]]$-module) is finite free over $\Lambda_P=\cO_F[[T_P(\bZ_p)^\circ]]$.
\item Define
\begin{equation*}
   \cM^{r}_{\Pord}=\Hom_{\Lambda_P}\left(\pV^{r,*}_{\Pord},\Lambda_P\right).
\end{equation*}
Given a dominant arithmetic weight $\utau^P\in\Hom_{\cont}(T_P(\bZ_p),\ol{\bQ}_p^\times)$ with dominant algebraic part $\ut^P\in X(T_P)^+$ and finite order part $\uepsilon^P\in\Hom_{\cont}(T_P(\bZ_p),\ol{\bQ}^\times)$, let $\cP_{\utau^P}$ be the corresponding prime ideal of $\cO_F[[T_P(\bZ_p)]]$. Then
\begin{equation*}
   \left.\cM^r_{\Pord}\otimes_{\cO_F[[T_P(\bZ_p)]]}\cO_F[[T_P(\bZ_p)]]\middle/\cP_{\utau^P}\right.\stackrel{\sim}{\lra}\varprojlim_m\varinjlim_l e_P V^{SP,r}_{m,l}[\utau^P],
\end{equation*}
and (see \eqref{eq:SP}, \eqref{eq:GammaSP} and \eqref{eq:temb} for the definition of the congruence subgroup $\Gamma_{SP}(p^l)\subset \Sp(2n,\bZ)$ and weight $\iota(\ut^P)\in X(T)$ associated to $\ut^P\in X(T_P)$)
\begin{equation*}
   \varinjlim_l e_PM^r_{\imath(\ut^P)}\left(\Gamma\cap\Gamma_{SP}(p^l),\uepsilon^P;F\right)\lhra \left(\left.\cM^r_{\Pord}\otimes_{\cO_F[[T_P(\bZ_p)]]}\cO_F[[T_P(\bZ_p)]]\middle/\cP_{\utau^P}\right.\right)[1/p].
\end{equation*}
Here the maps are equivariant under the action of the unramified Hecke algebra away from $Np$ and the $\bU^P_p$-operators.
\item There is the following so-called fundamental exact sequence (in the study of Klingen Eisenstein congruence), 
\begin{equation*}
  0\lra \cM^{r-1}_{\Pord}\lra \cM^{r}_{\Pord}\lra \bigoplus_{\substack{V\in\fC_{\bV}/\Gamma\\\mathrm{rk}\,V=r}}\cM^0_{V,P_{n-r}\text{-}\,\ord}\otimes_{\cO_F[[T_{P_{n-r}}(\bZ_p)]]}\cO_F[[T_P(\bZ_p)]]\lra 0,
\end{equation*}
and $\cM^0_{V,P_{n-r}\text{-}\,\ord}$ is the $\cO_F[[T_{P_{n-r}}(\bZ_p)]]$-module of families of $p$-adic ordinary Siegel modular forms of degree $n-r$ over $Y_{V,\ord}$ for the parabolic $P_{n-r}\subset\GL(n-r)$ defined by the partition $n-r=n_1+\dots+n_{d-1}+(n_d-r)$.
\end{enumerate}
\end{teono}

The whole Section \S\ref{sec:NCH} is devoted to the proof of this theorem, see in particular \S \ref{sec:HF}. 
The method can be generalized to other PEL type Shimura varieties, both with (using \cite{HPEL}) and without (using \cite{EischenMantovan,BraRosmu}) ordinary locus. 

With the new choice of sections at $p$ for the Siegel Eisenstein series and the Hida theory for non-cuspidal Siegel modular forms  we can then construct the improved $p$-adic $L$-function:
\begin{teono}[Theorem \ref{thm:twopadicLfun}]
With the same notation as above, assume that the parity of $\sC_P$ is compatible with $\phi$ ({\it i.e.} $\phi(-1)=\uptau^P_d(-1)$ for a $\utau^P$ in the image of the projection of $\sC_P$ to the weight space). There is a $p$-adic $L$-function $\cL^{\imp}_{\sC_P,\phi,\beta_1,\beta_2}(x) \in F_{\sC_P}$ which satisfies the following interpolation property. If $x$ is \'etale, its projection $\utau^P$ in the weight space is admissible, and $x$ is classical, then
\begin{align*}
   \cL^{\imp}_{\sC_P,\phi,\beta_1,\beta_2}(x)=\,&C_{x,\beta_1,\beta_2,\phi,N}\\
   &\hspace{-2em}\times E^{P\textnormal{-imp}}_p(n+1-t^P_d,\pi_x\times\phi\epsilon^P_d)\, E_\infty(n+1-t^P_d,\pi_x\times\phi\epsilon^P_d)\, L^{Np\infty}(n+1-t^P_d,\pi_x\times\phi\epsilon^P_d),
\end{align*}
where  $E^{P\textnormal{-imp}}_p(n+1-t^P_d,\pi_x\times\phi\epsilon^P_d)$ is the improved modified Euler factor at $p$ defined as in \S\ref{sec:Ep}.
\end{teono}

The Greenberg--Stevens method \cite{SSS} allows us to prove the following theorem on semi-stable trivial zeroes:
\begin{teono}[Theorem \ref{thm:mainzero}]
Let $x_0$ be an $F$-point of $\sC_P$ where the weight projection map $\Lambda_P\ra \bT^{1,N}_{\Pord}$ is \'{e}tale and maps $x_0$ to an admissible and classical point $\utau^P_0$ in the weight space. Suppose that the $p$-adic $L$-function $\cL_{\sC_P,\phi\omega^{n+1},\beta_1,\beta_2}\in  \bI_{\sC_P}[[S]] \otimes_{\bI_{\sC_P} } F_{\sC_P}$ has a semi-stable trivial zero at $((1+p)^{n+1}-1,x_0)$ and the local-global compatibility is satisfied by the $p$-adic Galois representation $\rho_{x_0}$. (This in particular implies that $\epsilon^P_d$ is trivial.) Then we have 
\begin{align*}
&\left.\frac{\textup{d}\cL_{\sC_P,\phi\omega^{n+1},\beta_1,\beta_2}(S,x_0)}{\textup{d} S}\right|_{S=(1+p)^{n+1}-1}\\
   =& \, -\ell(\rho_{x_0})\cdot C_{x_0,\beta_1,\beta_2,\phi,N}\cdot E^{P\textnormal{-imp}}_p(0,\pi_{x_0}\times\phi)\, E_\infty(0,\pi_{x_0}\times\phi)\, L^{Np\infty}(0,\pi_{x_0}\times\phi),
\end{align*}
where $ \ell(\rho_{x_0})$ is the $\ell$-invariant as defined by Greenberg.
\end{teono}

This result almost implies the conjecture of Greenberg, up to the non-vanishing of the $\ell$-invariant and of the imprimitive $L$-function. The non-vanishing of the $\ell$-invariant is a very hard problem and it is known only in the case of \cite{MTT}, thanks to a deep result in transcendental number theory stating that $q_E$ is trascendental \cite{Saint}. Note that for $n=2$ we know the non-vanishing of $ \ell(\rho_{x_0})$ whenever $\pi_{x_0}=\mr{Sym^3}(\pi_{f_E})$, where $f_E$ is the weight two modular form associated with an elliptic curve with semi-stable reduction at $p$. The imprimitive $L$-function could vanish because of the vanishing of some of the Euler factors at a prime $\ell$ dividing $N$. One may deal with such vanishing by selecting better sections at $\ell|N$.

\subsection*{Acknowledgments} The authors thank Kai-Wen Lan,  Vincent Pilloni, Eric Urban for many useful discussions, and the referees for their careful reading of the paper. Part of this work has been done while GR was a Herchel Smith fellow at Cambridge University and supernumerary fellow at Pembroke College, and during many visits at Columbia University; he would like to warmly thank these institutions. This material is based upon work supported by the National Science Foundation under Grant No. DMS-1638352.

\subsubsection*{Notation} 
For the whole length of the paper we fix an odd prime $p$ as well as isomorphism between $\ol{\bQ}_p$ and $\bC$. Also, we fix a positive integer $N\geq 3$ prime to $p$ and an integer $n \geq 1$ together with a partition $n=n_1+\dots +n_d$ with $n_1,\dots,n_d\geq 1$. 

We denote by $\mb V $ a free $\mb{Z}$-module of rank $2n$ with a standard basis $e_1, \ldots, e_n,f_1,\cdots,f_n$ equipped with a symplectic pairing given by $\begin{pmatrix}
0 &  I_n \\
-  I_n  & 0 
\end{pmatrix}$ with respect to the standard basis. Then $e_1,\cdots,e_n$ span a maximal isotropic subspace $V_n$ inside $\mb{V}$. Set $V_n^*=\mb{V}/V^\bot_n$. One can canonically identify $V^*_n$ with the maximal isotropic subspace of $\mb{V}$ spanned by $f_1,\cdots, f_n$ and there is the polarization $\mb{V}=V_n\oplus V^*_n$.

Let $G=\mr{Sp}(2n)$ be the algebraic group acting on $\mb{V}$ preserving the symplectic pairing. In matrix form it is
\begin{equation*}
   \left\{g\in\mr{GL}(2n):\ltrans g\begin{pmatrix} 0 & I_n\\-I_n & 0\end{pmatrix}g=\begin{pmatrix} 0 & I_n\\-I_n & 0\end{pmatrix}\right\}.
\end{equation*} 
Let $Q_G$ be the standard Siegel parabolic subgroup of $G$ preserving $V_n$, whose unipotent subgroup we denote by $U_{Q_G}$. We identify the Levi subgroup of $Q_G$ with $\mr{GL}(n)$ via the map 
\begin{align}\label{eq:QG}
  Q_G&\lra\mr{GL}(n)\\
  \begin{pmatrix}a&b\\0&\ltrans a^{-1}\end{pmatrix}&\longmapsto a.\nonumber
\end{align}
Denote by $B$ the standard Borel subgroup of $\mr{GL}(n)$ consisting of upper triangular matrices, and by $U_B$, $T$ its unipotent radical and maximal torus respectively. We fix the isomorphism of $\mb{G}_m^n$ with $T$ which sends $(a_1,\dots,a_n)$ to $\diag(a_1,\dots,a_n)$. The inverse image under \eqref{eq:QG} of $B$ constitutes the standard Borel subgroup $B_G$ of $G$ with unipotent radical $N_G$ and maximal torus $T_G$. The tori $T$ and $T_G$ are identified via the map \eqref{eq:QG}.

We put ourselves in the setting of \cite{PilHida}, i.e. considering Siegel modular forms ordinary with respect to a general parabolic subgroup of $\mr{GL}(n)$ containing $B$ associated to our fixed partition $n=n_1+\dots+n_d$ (the ordinarity considered in \cite{HPEL} is the ordinarity with respect to $B$). Set $N_i=\sum_{j=1}^i n_i$, $1\leq i\leq d$. Define 
\begin{align}
   P&=\left\{\left.\begin{pmatrix}a_1&*&*\\&\ddots&*\\&&a_d\end{pmatrix}\in\mr{GL}(n)\,\right|\,a_i\in\mr{GL}(n_i),\,1\leq i\leq d\right\},\label{eq:P}\\
   SP&=\left\{\left.\begin{pmatrix}a_1&*&*\\&\ddots&*\\&&a_d\end{pmatrix}\in\mr{SL}(n)\,\right|\,a_i\in\mr{SL}(n_i),\,1\leq i\leq d\right\}, \label{eq:SP}
\end{align}
and $U_P$ to be the unipotent radical of $P$. When the partition is taken as $n=1+1+\cdots+1$, the group $P$, (resp. both $SP$ and $U_P$) is just $B$ (resp. $U_B$). Let $T_P=P/SP$ and we fix the following isomorphism
\begin{align}\label{eq:TP}
   T_P=P/SP&\stackrel{\sim}{\lra} \mb{G}^d_m\\
   \begin{pmatrix}a_1&&\\&\ddots&\\&&a_d\end{pmatrix}u&\longmapsto \left(\mr{det}(a_1),\dots,\mr{det}(a_d)\right).\nonumber
\end{align}
Note here that $T_P$ is not the maximal torus in $P$. Denote by $X(T_P)$ the group of characters of $T_P$ (which are also naturally viewed as characters of $P$). We identify it with $\mb Z^d$ by associating to $\underline{t}^P:=(t^P_1,\dots,t^P_d)$ the character sending $\diag(a_1,\dots,a_d)$ to $\prod_{i=1}^d\mr{det}(a_i)^{t^P_i}$. When working with $B$, we shall drop the superscript from the notation for the characters when there is unlikely confusion. The map \eqref{eq:TP} restricts to a map $T\ra T_P$, which induces an embedding
\begin{equation}\label{eq:temb}
\begin{aligned}
 \imath:X(T_P) &\lra X(T)\\
  \underline{t}^P:=(t^P_1,\dots,t^P_d) &\longmapsto (\underbrace{t^P_1,\dots, t^P_1}_{n_1},\underbrace{t^P_2,\dots, t^P_2}_{n_2},\dots,\underbrace{t^P_d,\dots, t^P_d}_{n_d}).
\end{aligned}
\end{equation}
Denote by $X(T)^+$ the subset of $X(T)$ of dominant weights with respect to $B$ and set $X(T_P)^+=X(T_P)\cap X(T)^+$. Then $\ut^P\in X(T_P)$ belongs to $X(T_P)^+$ if and only if $t^P_1\geq t^P_2\geq\cdots\geq t^P_d$. Fix a finite extension $F$ of $\mb{Q}_p$ (assumed to be sufficiently large in the context) and denote by $\mc{O}_F$ its ring of integers. The weight space in Hida theory for $P$-ordinary Siegel modular forms over $F$  is $\mr{Spec}\left(\mc{O}_F[[T_P(\mb{Z}_p)]]\right)$. For an arithmetic point of the weight space $\utau^P\in\mr{Spec}\left(\mc{O}_F[[T_P(\mb{Z}_p)]]\right)(\overline{\mb{Q}}_p)$, i.e. a character in $\mr{Hom}_{\mr{cont}}\left(T_P(\mb{Z}_p),\overline{\mb{Q}}^\times_p\right)$ that is the product of an algebraic and a finite order character, we write its algebraic part (resp. finite order part) as $\utau^P_{\mr{alg}}=\ut^P=(t^P_1,\dots,t^P_d)$ (resp. $\utau^P_{\mr{f}}=\uep^P=(\epsilon^P_1,\dots,\epsilon^P_d)$).

We fix the standard additive character $\be_{\mb{A}}=\bigotimes_v\be_v:\mb{Q}\backslash\mb{A}\ra\mb{C}^\times$ with local component $\be_v$ defined as $\be_v(x)=\begin{cases}e^{-2\pi i\{x\}_v},& v\neq\infty\\e^{2\pi ix},&v=\infty\end{cases}$ where $\{x\}_v$ is the fractional part of $x$.

\section{Non-cuspidal Hida theory}\label{sec:NCH}

In this section we develop Hida theory for non-cuspidal Siegel modular forms, or more precisely, for the $P$-ordinary Siegel modular forms vanishing along the strata with cusp labels of rank $>r$ for some $0\leq r\leq n_d$. Later, the family of Siegel Eisenstein series on $\Sp(4n)$ we shall use to construct the improved $p$-adic $L$-function, unlike the ones for the usual $p$-adic $L$-function, does not restrict to cuspidal forms on $\Sp(2n)\times\Sp(2n)$ (see the discussion at the end of \S\ref{sec:placep}). The Hida theory developed here will be applied to them. Also, we expect such a theory to be of independent interest and to find applications elsewhere, for instance in the study of Eisenstein congruences.

The main difficulty in directly generalizing  Hida theory for cuspidal forms on PEL Shimura varieties to non-cuspidal forms is that for an algebraic representation $W$ of $\GL(n)_{/\bZ}$ of finite rank, an algebra $R$ and a subgroup $L\subset \GL(n,\bZ)$ of the form
\begin{equation*}
   L=\left\{\begin{pmatrix}I_{n-r}&\ast\\0&\ast\end{pmatrix}\in\GL(n,\bZ)\right\},\quad 1\leq r\leq n,
\end{equation*}
the module $W(R/p^m)^L$ is not necessarily equal to $W(R)^L\otimes R/p^m$.

In \cite[\S6]{SU} and \cite[\S4]{Hsieh}, a subsheaf of $\omega^\flat_{\ut}\subset\omega_{\ut}$ is introduced to remedy this failure of base change property. The sheaf $\omega^\flat_{\ut}$ is not free and differs from $\omega_{\ut}$ along the boundary of the toroidal compactification. The base change property for global sections of $\omega^\flat_{\ut}$ is shown \textit{loc. cit}. With this base change property, by mimicking Hida's method \cite{HPEL}, Hida theory for certain non-cuspidal forms on $\U(2,2)$ and $\U(n,1)$ is established \textit{loc. cit}. 

Here we take a different approach. Let $\pV^{SP,r}$ be the space of $p$-adic forms for the parabolic $P$ vanishing along strata indexed by cusp labels of rank $> r$ with $p$-power torsion coefficients. Instead of studying the space $\pV^{SP,r}$ via the classical Siegel modular forms embedded in it through \eqref{eq:embed2} (for which a base change property for the space of certain non-cuspidal classical Siegel modular forms is required), we make a careful analysis of the Igusa tower over the boundary and define a nice subspace $\pV^{SP,r,\flat}$ inside $\pV^{SP,r}$. The exact sequences in Proposition \ref{prop:exact} plus Proposition \ref{prop:UpEquiv} allow us to deduce desired properties for the space $\pV^{SP,r,\flat}$ from those for the space of cuspidal $p$-adic Siegel modular forms. Meanwhile, Proposition \ref{prop:key} shows that the desired properties for $\pV^{SP,r,\flat}$ imply the existence of a nice ordinary projection on $\pV^{SP,r}$. Then we obtain Hida theory for non-cuspidal Siegel modular forms as summarized in Theorem \ref{thm:main}.

Exact sequences for automorphic bundles and $p$-adic analytic deformation of automorphic bundles, similar to those in Proposition \ref{prop:exact}, are used in \cite{WeiVek} and \cite{BrasRos}, where things are simpler than our case here because everything is in characteristic zero and the issue of base change does not appear.

\subsection{Compactifications of Siegel varieties}\label{sec:TorG}

We start by briefly recalling some facts on the toroidal and minimal compactifications of Siegel varieties of principal level. We mainly follow the notation in \cite{PilHida}. Fix an integer $N\geq 3$ coprime to $p$. Let $Y$ be the degree $n$ Siegel variety of principal level $N$ defined over $\mb Z[1/N,\zeta_N]$. All the objects we consider in the following are endowed with principal level $N$ structure and we shall omit $N$ to lighten the notation. 

Recall that $\mb V=\mb Z^{2n}$ with standard basis $e_1,\cdots,e_n,f_1,\cdots,f_n$ and the symplectic pairing given by $\begin{pmatrix}0&  I_n\\- I_n & 0 
\end{pmatrix}$.  For $1\leq r\leq n$, let $V_r$ be the submodule of $\mb V$ spanned by $e_1,\cdots,e_r$ (we sometimes call it the standard submodule of rank $r$ in $\bV$), and we put $V_0=\{0\}$. The group $\mr{Sp}(\mb V)\cong\mr{Sp}(2n,\mb Z)$ acts on $\mb V$ preserving the symplectic pairing. Denote by $\Gamma$ the kernel of the projection $\mr{Sp}(2n,\mb Z)\ra\mr{Sp}(2n,\mb Z/N\mb Z)$.

Denote by $\mf{C}_{\mb V}$ the set of cotorsion free isotropic $\mb Z$-submodules of $\mb V$. The group $\mr{Sp}(2n,\mb Z)$ acts naturally on $\mf{C}_{\mb V}$. The quotient $\mf{C}_{\mb V}/\Gamma$ is called the set of cusp labels of level $\Gamma$ (or of principal level $N$). For a free $\mb Z$-module $X$ of finite rank, we write $C(X)$ to denote the cone of positive semi-definite symmetric bilinear forms on $X\otimes\mb{R}$ with rational radicals, and by $C(X)^{\circ}$ its interior. A surjective morphism $X\ra X'$ of free $\mb Z$-modules induces an inclusion $C(X')\hra C(X)$.  Define $\mc{C}_{\mb V}$ as the quotient of the disjoint union $\coprod\limits_{V\in\mf{C}_{\mb V}}C(\mb V/V^{\bot})$ by the equivalence relations induced by the inclusions $C(\mb V/V^\bot)\hra C(\mb V/V^{\prime\bot})$ for $V\subset V'$, $V,V'\in\mf{C}_\mb V$. The group $\mr{Sp}(2n,\mb Z)$ acts on $\mc{C}_\mb V$.

A $\mr{GL}(n,\mb Z)$-admissible smooth rational polyhedral cone decomposition $\Sigma$ of $C(\mb Z^n)$ (\cite[Definition 2.2]{FC}) gives rise to a rational polyhedral cone decomposition $\Sigma_{\mc{C}_{\mb V}}$ of $\mc{C}_{\mb V}$. Corresponding to it is a toroidal compactification $X^\Sigma$ of $Y$ endowed with an action of $\mr{Sp}(2n,\mb Z)$ \cite[\S IV.6]{FC}. 

The toroidal compactification $X^\Sigma$ comes with a stratification indexed by $\Sigma_{\mc{C}_{\mb V}}/\Gamma$, and we denote by $Z_\sigma$ the stratum in $X^\Sigma$ associated with $\sigma\in\Sigma_{\mc{C}_{\mb V}}$. There is a canonical map $\Sigma_{\mc{C}_{\mb V}}\ra\mf{C}_{\mb V}$ sending $\sigma$ to the unique $V_\sigma\in\mf{C}_\mb V$ satisfying $\sigma\subset C(\mb V/V_\sigma^\bot)^\circ$. The locally closed subscheme $Z_V\subset X^\Sigma$ is defined as the union $\coprod\limits_{\sigma\in\Sigma_{\mc{C}_{\mb V}}/\Gamma,V_\sigma=V} Z_\sigma$. For $0\leq r\leq n$, define $\cI^r_{X^\Sigma}$ to be the sheaf of ideals associated to the closed subscheme $\coprod\limits_{V\in\fC_{\bV}/\Gamma,\mr{rk} V>r} Z_V$.

Over $X^\Sigma$, there is the canonical semi-abelian scheme $\mc{G}_{ /X^\Sigma}$ whose restriction to $Y$ is the universal principally polarized abelian scheme of genus $n$ with principal level $N$ structure. The coherent sheaf $\omega$ over $X^\Sigma$ is defined as the sheaf of invariant differentials of $\mc{G}_{ /X^\Sigma}$, which is locally free of rank $n$.

From the toroidal compactification, the minimal compactification is constructed as
\begin{equation*}
   X^\star=\mathrm{Proj}\left(\bigoplus_{k\geq 0}\mr{H}^0(X^\Sigma,\mr{det}^k\omega)\right).
\end{equation*}
The projection $\pi:X^\Sigma\ra X^\star$ is proper with connected fibers. The minimal compactification $X^\star$ is stratified by $\mf{C}_{\mb V}/\Gamma$. The stratum $Y_{V}$ corresponding to $V\in\mf{C}_{\mb V}/\Gamma$ is defined as the image of $Z_V$. As a scheme it is isomorphic to the Siegel variety of degree $(n-\mathrm{rk}\,V)$ and principal level $N$. 

\subsection{\texorpdfstring{The Igusa tower over the ordinary locus and $p$-adic forms}{The Igusa tower over the ordinary locus and p-adic forms}} \label{sec:Igusa}
The invertible sheaf $\mr{det}^k\omega$ descends to an invertible sheaf on $X^\star$, which we still denote by $\mr{det}^k\omega$. For sufficiently large $k$, it is very ample over $X^\star$. Choose $k$ such that the $k$-th power of the Hasse invariant, which is an element in $\mr{H}^0\left(X^\star_{/\mb{F}_p[\zeta_N]},\,\mr{det}^{p-1}\omega\right)$, lifts to a global section $E\in \mr{H}^0\left(X^\star,\mr{det}^{t(p-1)}\omega\right)$ (for $t \geq 0$ an integer big enough). We write its pull-back in $\mr{H}^0\left(X^\Sigma,\mr{det}^{t(p-1)}\omega\right)$ also as $E$.

Let $X^{\star,\mr{ord}}=X^{\star}[1/E]$ and $X^{\Sigma,\mr{ord}}=X^\Sigma[1/E]$, and define $Y^\ord_V$, $Z^{\ord}_\sigma$, $Z^\ord_V$ similarly. The reductions modulo powers of $p$ of these schemes are independent of the choice of $E$ and are called the ordinary loci. Note that $X^{\star,\ord}$ is affine, while $X^{\Sigma,\ord}$ is not (except when $n=1$).

Fix a finite extension $F$ of $\mb{Q}_p$ containing all $N$-th roots of unity. We regard all above schemes as defined over $\mc{O}_F$. For $m\geq 1$, we use a subscript $_m$ to indicate the reduction modulo $p^m$. Over $X^{\Sigma,\mr{ord}}_m$, we consider the full Igusa tower 
$$\ms{T}^\Sigma_{m,l}=\underline{\mr{Isom}}_{X^{\Sigma,\mr{ord}}_m}\left(\mu^n_{p^l},\mc{G}_{/X_m^{\Sigma,\mr{ord}}}[p^l]^\circ\right),$$
for $l\geq 1$. It is an \'{e}tale cover of $X^{\Sigma,\mr{ord}}_m$ with Galois group isomorphic to $\mr{GL}(n,\mb Z/p^l\mb Z)$. 
The group 
$$\Gamma_0(p^l)=\left\{g=\begin{pmatrix}a_g&b_g\\c_g&d_g\end{pmatrix}\in G(\mb Z):c_g\equiv 0\mod p^l\right\}$$
naturally acts on $\ms{T}^\Sigma_{m,l}$ with $g$ acting on the principal level $N$ structure and $a_g$ acting by the Galois action of $\mr{GL}(n,\mb Z/p^l\mb Z)$ on $\ms{T}^\Sigma_{m,l}$ over $X^{\Sigma,\mr{ord}}_m$. 

Define
$$\sT^{\Sigma}_{SP,m,l}=\left.\ms{T}^\Sigma_{m,l}\middle/SP(\mb Z_p/p^l\bZ_p)\right.$$
(see \eqref{eq:P} and \eqref{eq:SP} for the definition of the algebraic subgroups $P$, $SP$ of $\GL(n)$). It parametrizes (in addition to the structure parametrized by $X^{\Sigma,\mr{ord}}_m$) the level structure $(E_i,\varepsilon_i)_{1\leq i\leq d,\,p^l}$, where $\{0\}=E_0\subset E_1\subset\cdots\subset E_d=\mc{G}[p^l]^\circ$ is a $d$-step increasing filtration and $\varepsilon_i$ is an isomorphism $\bigwedge^{\scriptsize {n_i}} \mu_{p^l}^{n_i}\cong \bigwedge^{\scriptsize {n_i}} E_i/E_{i-1}$. There is a natural $T_P(\mb Z_p)$-action on $\sT^{\Sigma}_{SP,m,l}$. 


Write $f_{m,l}:\sT^{\Sigma}_{SP,m,l}\ra X^{\Sigma,\mr{ord}}_m$ for the natural projection. Define
\begin{align*}
   V^{SP,r}_{m,l} & =\mr{H}^0\left(\sT^{\Sigma}_{SP,m,l},\, f^*_{m,l}\mc{I}^r_{X^{\Sigma,\mr{ord}}_m}\right),\\
	 \pV^{SP,r} & =\varinjlim\limits_{m}\varinjlim\limits_l V^{SP,r}_{m,l}.
\end{align*}
The space $\pV^{SP,r}$ is called the space of $p$-adic Siegel modular forms for the parabolic $P$ vanishing along the strata indexed by cusp labels of rank $>r$ with $p$-power torsion coefficients. When $P=B$ we shall drop the $P$ from the notation, and when $r=n$ we shall drop $r$ from the notation.

The $T_P(\mb Z_p)$-action on $\sT^{\Sigma}_{SP,m,l}$ equips $V^{SP,r}_{m,l}$ and $\pV^{SP,r}$ with an $\mc{O}_F\llbracket T_P(\mb Z_p)\rrbracket $-module structure. The space $\pV^0$ is the space of cuspidal $p$-adic Siegel modular forms with $p$-power torsion coefficients, which is considered in Hida theory (for the Borel $B$) for cuspidal Siegel modular forms, while $\pV^{SP,0}$ (for general $P$) is the one in \cite{PilHida}. 

Besides the torsion $\mb Z_p$-module $\pV^{SP,r}$ (which is in fact $p$-divisible by Remark \ref{rmk:Vbc}), we will also consider the $\mb Z_p$-module $\varprojlim\limits_m\varinjlim\limits_l V^{SP,r}_{m,l}$, {\it i.e.} taking the inverse instead of direct limit with respect to $m$ (which is torsion free over $\mb Z_p$ by Remark \ref{rmk:Vbc}). It is the torsion $\mb Z_p$-module $\pV^{SP,r}$ that will be used to construct the $\mc{O}_F\llbracket T_P(\mb Z_p)\rrbracket $-module of Hida families. Meanwhile, the space $\varprojlim\limits_m\varinjlim\limits_l V^{SP,r}_{m,l}$ is more easily seen related to the classical Siegel modular forms.

More precisely, for $\ut^P\in X(T_P)^+$, $\uepsilon^P\in\Hom\left(T_P(\bZ/p^l),\bC^\times\right)$, and $\utau^P\in\Hom_{\cont}\left(T_P(\bZ_p),\ol{\bQ}^\times_p\right)$ which is the product of $\ut^P$ and $\uepsilon^P$, there is a canonical Hecke-equivariant embedding \cite[\S 4.2.1]{PilHida}
\begin{equation}\label{eq:embed2}
   \varinjlim_l \mr M^r_{\imath(\ut^P)}\left(\Gamma\cap\Gamma_{SP}(p^l),\uepsilon^P;F\right)\lhra \left(\varprojlim_m\varinjlim_l V^{SP,r}_{m,l}[\utau^P]\right)[1/p].
\end{equation}
Here $\mr M^r_{\ut}\left(\Gamma\cap \Gamma_{SP}(p^l),\uepsilon^P;F\right)$ denotes the space of classical holomorphic Siegel modular forms of weight $\ut=\imath(\ut^P)$ and level $\Gamma\cap \Gamma_{SP}(p^l)$ with nebentypus $\uepsilon^P$ vanishing along strata with cusp labels of rank $>r$, and the congruence subgroup $\Gamma_{SP}(p^l)$ is defined as
\begin{equation}\label{eq:GammaSP}
   \Gamma_{SP}(p^l)=\left\{g\in\Sp(2n,\bZ):g\,\mr{mod}\, p^l \text{ belongs to } SP(\bZ/p^l)\right\}.
\end{equation}
The vanishing condition here for the classical Siegel modular forms is equivalent to requiring that at all cusps all the Fourier coefficients with indices of corank $> r$ vanish. The space $\varprojlim_m\varinjlim_l V^{SP,r}_{m,l}[\utau^P]$ is the $\utau^P$-eigenspace for the action of $T_P(\bZ_p)$ on $\varprojlim_m\varinjlim_l V^{SP,r}_{m,l}$. The construction of the embedding \eqref{eq:embed2} mainly relies on the Hodge--Tate map
\begin{equation*}
   \mr{Hom}_R(\mc G[p^\infty]^\circ,\mu_{p^\infty})\otimes_{\mb Z_p}R\stackrel{\sim}{\lra}\omega_{\mc G/R}
\end{equation*}
for an ordinary  semi-abelian scheme $\mc G$ over a $\mb Z_p$-algebra $R$. 

\subsection{The main theorem}
Our goal is to establish the following theorem.
\begin{thm}\label{thm:main}
For given $P\subset\GL(n)$ as in \eqref{eq:P} and an integer $1\leq r\leq n_d$, we have the following.
\begin{enumerate}[(i)] 
\item An ordinary projector $e_P=e_P^2$ can be defined on $\pV^{SP,r}$, and the Pontryagin dual of its ordinary part 
\begin{equation*}
\pV^{r,*}_{\Pord}=\Hom_{\bZ_p}\left(e_P\pV^{SP,r},\bQ_p/\bZ_p\right)
\end{equation*}
(which is naturally an $\cO_F\llbracket T_P(\bZ_p)\rrbracket $-module) is finite free over $\Lambda_P=\cO_F\llbracket T_P(\bZ_p)^\circ]]$, where $T_P(\bZ_p)^\circ$ is the maximal $p$-profinite subgroup of $T_P(\bZ_p)$.
\item Define
\begin{equation*}
   \cM^{r}_{\Pord}=\Hom_{\Lambda_P}\left(\pV^{r,*}_{\Pord},\Lambda_P\right).
\end{equation*}
Given an arithmetic weight $\utau^P\in\Hom_{\cont}(T_P(\bZ_p),\ol{\bQ}_p^\times)$ with dominant algebraic part $\ut^P\in X(T_P)^+$ and finite order part $\uepsilon^P\in\Hom_{\cont}(T_P(\bZ_p),\ol{\bQ}^\times)$, let $\cP_{\utau^P}$ be the corresponding prime ideal of $\cO_F\llbracket T_P(\bZ_p)\rrbracket $. Then
\begin{equation*}
   \left.\cM^r_{\Pord}\otimes_{\cO_F\llbracket T_P(\bZ_p)\rrbracket }\cO_F\llbracket T_P(\bZ_p)\rrbracket \middle/\cP_{\utau^P}\right.\stackrel{\sim}{\lra}\varprojlim_m\varinjlim_l e_P V^{SP,r}_{m,l}[\utau^P],
\end{equation*}
which combining with \eqref{eq:embed2} gives
\begin{equation}\label{eq:cembed}
   \left.\varinjlim_l e_P \mr M^r_{\imath(\ut^P)}\left(\Gamma\cap\Gamma_{SP}(p^l),\uepsilon^P;F\right)\lhra \left(\cM^r_{\Pord}\otimes_{\cO_F\llbracket T_P(\bZ_p)\rrbracket }\cO_F\llbracket T_P(\bZ_p)\rrbracket \middle/\cP_{\utau^P}\right)[1/p]\right..
\end{equation}
Here the maps are equivariant under the action of the unramfied Hecke algebra away from $Np$ and the $\bU^P_p$-operators.
\item There is the following so-called fundamental exact sequence (in the study of Klingen Eisenstein congruence), 
\begin{equation*}
  0\lra \cM^{r-1}_{\Pord}\lra \cM^{r}_{\Pord}\lra \bigoplus_{\substack{V\in\fC_{\bV}/\Gamma\\\mathrm{rk}\,V=r}}\cM^0_{V,P_{n-r}\text{-}\,\ord}\otimes_{\cO_F\llbracket T_{P_{n-r}}(\bZ_p)\rrbracket }\cO_F\llbracket T_P(\bZ_p)\rrbracket \lra 0,
\end{equation*}
where 
\begin{equation*}
   P_{n-r}=\left\{\left.\begin{pmatrix}a_1&*&*\\&\ddots&*\\&&a_d\end{pmatrix}\in\mr{GL}(n-r)\,\right|\,a_i\in\mr{GL}(n_i),\,1\leq i\leq d-1,\,a_d\in\GL(n_d-r)\right\},\\
\end{equation*}
and $\cM^0_{V,P_{n-r}\text{-}\,\ord}$ is the $\cO_F\llbracket T_{P_{n-r}}(\bZ_p)\rrbracket $-module of families of $p$-adic cuspidal ordinary Siegel modular forms of degree $n-r$ over $Y_{V,\ord}$ for the parabolic $P_{n-r}$.
\end{enumerate}
\end{thm}

\begin{rem}
For two different $P$, $P'$, there are no inclusion relations between $\cM^r_{\Pord}$ and $\cM^r_{P'\text{-ord}}$, thus in order to study certain Siegel modular forms by using Hida theory, one needs first to specify a parabolic $P$ containing the standard Borel subgroup of $\mr{GL}(n)$ such that the Siegel modular forms under inspection are $SP(\bZ_p)$-invariant and $P$-ordinary. There is not such Hida theory (as to the authors' knowledge) that studies $P$-ordinary Siegel modular forms for all $P$ simultaneously. On the contrary, when studying families of finite slope families, there is no need to specify a parabolic as the theory established for the Borel (e.g. \cite{AIP}) treats all Siegel modular forms of finite slope.
\end{rem}

The remaining part of this section is devoted to proving this theorem. The proof relies on a careful study of the quotient $\pV^{SP,r}/\pV^{SP,r-1}$ and the boundary of the Igusa tower, which leads to the definition of the subspace $V^{SP,r,\flat}_{m,l}\subset V^{SP,r}_{m,l}$. This subspace is characterized by the vanishing along certain connected components of the Igusa tower over $\coprod\limits_{\substack{V\in\fC_{\bV}/\Gamma\\\mathrm{rk}\,V=r}}Z_{V,\ord}$, and plays an important role in our proof of the above theorem.

\subsection{The Mumford construction}\label{sec:Mum}
We quickly recall the Mumford construction which will be used in the description of the fiber of the push-forward of the ideal sheaf $\cI^r_{X^{\Sigma}}$, as well as in the definition of $q$-expansions. 

Given a free $\mb{Z}$-module $X_r$ of rank $r$ with basis $x_1,\cdots,x_r$, set $X^*_r$ to be its dual free $\mb{Z}$-module with dual basis $x^*_1,\cdots,x^*_r$. Let $\sT_{n-r,m,l}$ the Igusa tower, $m,l\geq 1$, over the degree $(n-r)$ Siegel variety of principal level $N$ and $(\mc{A}_{/\sT_{n-r,m,l}},\psi_{N,\sT_{n-r,m,l}},\phi_{p,\sT_{n-r,m,l}})$ be the universal object over it. 

The extensions of $\mc{A}_{/\sT_{n-r,m,l}}$ by the torus $X^*_r\otimes\mb{G}_m$ are parametrized by $\mr{Hom}_{\sT_{n-r,m,l}}(X_r,\mc{A}_{/\sT_{n-r,m,l}})$. Let $\mc{B}_{X^*_r,m,l}$ be an abelian scheme which is isogenous to 
$$\mr{Hom}_{\sT_{n-r,m,l}}(N^{-1}X_r,\mc{A}_{/\sT_{n-r,m,l}})$$
via an isogeny of degree a power of $p$, related to the $p$-level structure of the Igusa tower. Given $\mu\in N^{-1}X_r$, there is tautologically a map $c(\mu):\mc{B}_{X^*_r,m,l}\ra \mc{A}_{/\sT_{n-r,m,l}}$ through evaluation at $\mu$. Denote by $S^2(X_r)$ the symmetric quotient of $X_r\otimes_\mb{Z} X_r$. Let $\mc{P}\ra \mc{A}_{/\sT_{n-r,m,l}}\times_{\sT_{n-r,m,l}}\mc{A}_{/\sT_{n-r,m,l}}$ be the Poincar\'{e} bundle and $\mc{P}^\times\ra \mc{A}_{/\sT_{n-r,m,l}}\times_{\sT_{n-r,m,l}}\mc{A}_{/\sT_{n-r,m,l}}$ be its associated $\mb{G}_m$-torsor.

Pick a basis $[\mu_i\otimes\nu_i]$, $1\leq i\leq r(r-1)/2$, of $N^{-1}S^2(X_r)$ with $\mu_i,\nu_i$ belonging to $N^{-1}X_r$.  Associated to each $[\mu_i\otimes\nu_i]$ there is a map
\begin{equation*}
   c(\mu_i)\times c(\nu_i):\mc{B}_{X^*_r,m,l}\lra \mc{A}_{/\sT_{n-r,m,l}}\times_{\sT_{n-r,m,l}}\mc{A}_{/\sT_{n-r,m,l}},
\end{equation*}
along which one can pull back the Poincar\'{e} bundle and its associated $\bG_m$-torsor. Define
$$\mc{M}_{X^*_r,m,l}=\prod_{i}(c(\mu_i)\times c(\nu_i))^*(\mc{P}^\times)^{\otimes N},$$ 
which is a torsor over $\mc{B}_{X^*_r,m,l}$ for the torus $\mr{Hom}_{\mb{Z}}(N^{-1}S^2(X_r),\mb{G}_m)$. For $\lambda=\sum_{i=1}^{r(r-1)/2} a_i[\mu_i\otimes\nu_i]\in S^2(X_r)$, define the invertible sheaf $\cL(\lambda)$ over $\mc{B}_{X^*_r,m,l}$ as
$$
   \cL(\lambda)=\bigotimes_i c(\mu_i)\times c(\nu_i))^*\cP^{\otimes a_iN}.
$$
We have
\begin{equation*}
   \left(\cM_{X^*_r,m,l}\ra \sT_{n-r,m,l}\right)_*\cO_{\cM_{X^*_r,m,l}}=\bigoplus _{\lambda\in N^{-1}S^2(X_r)}\mr H^0\left(\cB_{X^*_r,m,l},\,\cL(\lambda)\right).
\end{equation*}

Now suppose $\sigma\subset C(X_r)^\circ$ is a cone generated by a set of elements that extends to a basis of the space of symmetric bilinear forms on $X_r$. Let $\mc{M}_{X^*_r,m,l}\hra \mc{M}_{X^*_r,m,l,\sigma}$ be the affine torus embedding over $\mc{B}_{X^*_r,m,l}$ corresponding to $\sigma$. Denote by $\sigma^\vee$ the dual cone of $\sigma$ consisting of elements in $S^2(X_r)\otimes\bR$ whose pairing with any element in $\sigma$ is non-negative. Let $\sigma^{\vee,\circ}$ be the interior of $\sigma^\vee$. Let $\cI^\sigma_{\mc{M}_{X^*_r,m,l,\sigma}}$ be the ideal sheaf inside $\cO_{\mc{M}_{X^*_r,m,l,\sigma}}$ attached to the boundary of the affine torus embedding. Then
\begin{align*}
  \numberthis\label{eq:O}\left(\cM_{X^*_r,m,l,\sigma}\ra \sT_{n-r,m,l}\right)_*\cO_{\cM_{X^*_r,m,l,\sigma}}&=\bigoplus _{\lambda\in N^{-1}S^2(X_r)\cap\,\sigma^\vee} \mr H^0\left(\cB_{X^*_r,m,l},\,\cL(\lambda)\right),\\
   \numberthis\label{eq:I} \left(\cM_{X^*_r,m,l,\sigma}\ra \sT_{n-r,m,l}\right)_*\cI^\sigma_{\cM_{X^*_r,m,l,\sigma}}&=\bigoplus _{\lambda\in N^{-1}S^2(X_r)\cap\,\sigma^{\vee,\circ}}\mr H^0\left(\cB_{X^*_r,m,l},\,\cL(\lambda)\right).
\end{align*}

Let $\wh{\mc{M}}_{X^*_r,m,l,\sigma}$ be the formal completion along the boundary of the torus embedding. The natural map $X_r\ra\mr{Hom}(X_r,S^2(X_r))$ defines a period subgroup $N^{-1}X_r\subset X^*_r\otimes\mb{G}_{m/\mb{Z}[N^{-1}S^2(X_r)]}$ with a polarization given by the duality between $X_r$ and $X_r^*$. The Mumford construction gives a principally polarized semi-abelian scheme $\mc{G}_{/\wh{\mc{M}}_{X^*_r,m,l,\sigma}}$,  together with a canonical principal level $N$ structure $\psi_{N,\mr{can}}:(\mb{Z}/N\mb{Z})^{2n}\ra\mc{G}_{/\wh{\mc{M}}_{X^*_r,m,l,\sigma}}[N]$ and a canonical trivialization $\phi_{p,\mr{can}}:\mu^n_{p^l}\stackrel{\sim}{\ra}\mc{G}_{/\wh{\mc{M}}_{X^*_r,m,l,\sigma}}[p^l]^\circ$, which comes from the level structure parametrized by $\sT_{n-r,m,l}$, the extension data parametrized by $\mc{B}_{X^*_r,m,l}$ plus the fixed basis of $X_r$.

\subsection{The fiber of the push-forward to the minimal compactification}
Let $\sT^{\star}_{SP,m,l}$ be the Stein factorization of $\sT^\Sigma_{SP,m,l}\ra X^{\star,\ord}$,
\begin{equation}\label{eq:Stein}
  \xymatrixcolsep{5pc}\xymatrix{
   \sT^{\Sigma}_{SP,m,l}\ar[r]^{f_{m,l}}\ar[d]^{\pi_\sT}& X^{\Sigma,\ord}_m\ar[d]^{\pi}\\
   \sT^{\star}_{SP,m,l}\ar[r]& X^{\star,\ord}_m.
}
\end{equation}
The scheme $\sT^{\Sigma}_{SP,m,l}$ (resp. $\sT^{\star}_{SP,m,l}$) can also be viewed as the partial toroidal (resp. minimal) compactification of the Igusa tower $\sT_{SP,m,l}$ over $Y^\ord_m$, which is a special case of the construction in \cite{LanOrd}. They admit a similar description as $X^\Sigma$, $X^\star$.

Let $\mf{C}_{\mb{V},\,p^l}\subset\fC_\bV$ be the orbit of $\{V_0,V_1,\cdots,V_r\}$ under the action of the group $ \Gamma_0(p^l)$ and define $\mc{C}_{\mb{V},\,p^l}$ from $\mf{C}_{\mb{V},\,p^l}$ in the same way as $\mc{C}_{\mb{V}}$ from $\fC_\bV$. The partial compactification  $\sT^{\star}_{SP,m,l}$ (resp. $\sT^{\Sigma}_{SP,m,l}$) is stratified by $\mf{C}_{\mb{V},\,p^l}/\Gamma\cap\Gamma_{SP}(p^l)$ (resp. the rational polyhedral cone decomposition $\Sigma_{\mc{C}_{\mb{V},\,p^l}}$ induced from $\Sigma$). The natural maps
\begin{align*}
   \numberthis\label{eq:pc}\fp_{\mf{C},l}:\mf{C}_{\mb{V},p^l}/\Gamma\cap\Gamma_{SP}(p^l)&\lra\mf{C}_{\mb{V}}/\Gamma,\\
   \fp_{\mc{C},l}:\Sigma_{\mc{C}_{\mb{V},p^l}}/\Gamma\cap\Gamma_{SP}(p^l)&\lra\Sigma_{\mc{C}_{\mb{V}}}/\Gamma
\end{align*}
are surjective.

In order to distinguish from the notation for the stratum indices for $X^\star$, $X^\Sigma$, we use an extra $\wt{\,\,\,}$ to denote the stratum indices for $\sT^{\star}_{SP,m,l}$, $\sT^{\Sigma}_{SP,m,l}$, {\it i.e.} we write $\wt{V}$, $\wt{\sigma}$ for elements in $\mc{C}_{\mb{V},\,p^l}/\Gamma\cap\Gamma_{SP}(p^l)$, $\Sigma_{\cC_{\bV,\,p^l}}/\Gamma\cap\Gamma_{SP}(p^l)$. The stratum in $\sT^{\star}_{SP,m,l}$ (resp. $\sT^{\Sigma}_{SP,m,l}$) associated to $\wt{V}$ (resp. $\wt{\sigma}$, $\wt{V}$) will be denoted as $\sT_{\wt{V},m,l}$ (resp. $Z_{\wt{\sigma},m,l}$, $Z_{\wt{V},m,l}$). 

The stratum $\sT_{\wt{V},m,l}$ is isomorphic to the quotient of the full level $l$ Igusa tower over $Y_{\wt{V},m}$ by $\mr{Im}\left(\Gamma_{\wt{V}}\cap \Gamma_{SP}(p^l)\ra\Sp(\wt{V}^\bot/\wt{V},\bZ/p^l)\right)$, where $\Gamma_{\wt{V}}\subset \Gamma$ is the subgroup mapping $\wt{V}$ to itself. For $\wt{V}=V_{\wt{\sigma}}$, the diagram below describes the completion of $\sT^{\Sigma}_{SP,m,l}$ along the stratum $Z_{\wt{\sigma},m,l}$,
$$\xymatrix{
   \mc{M}_{\wt{V},m,l}\ar@{^{(}->}[r]\ar[d]&\mc{M}_{\wt{V},m,l,\wt{\sigma}}\ar[dl]\ar@{^{(}->}[r]&\mc{M}_{\wt{V},m,l,\Sigma_{\wt{V}}}\ar[dll]\\
	 \mc{B}_{\wt{V},m,l}\ar[d]\\
	 \sT_{\wt{V},m,l}.
}$$
Here $\mc{B}_{\wt{V},m,l}$, $\mc{M}_{\wt{V},m,l}$, $\mc{M}_{\wt{V},m,l,\wt{\sigma}}$ are the objects constructed in \S\ref{sec:Mum} with $X^*_r=\wt{V}$, and $\mc{M}_{\wt{V},m,l,\Sigma_{\wt{V}}}$ is the torus embedding associated to the rational polyhedral cone decomposition $\Sigma_{\wt{V}}$ of $C(\bV/\wt{V}^\bot)$ given by $\Sigma_{\cC_\bV,\,p^l}$. Let $\wh{\mc{M}}_{\wt{V},m,l,\wt{\sigma}}$ be the completion of $\mc{M}_{\wt{V},m,l,\Sigma_{\wt{V}}}$ along the closure of the stratum attached to $\sigma$. Then the completion of $\sT^{\Sigma}_{SP,m,l}$ along $Z_{\wt{\sigma},m,l}$ is isomorphic to $\wh{\mc{M}}_{\wt{V},m,l,\wt{\sigma}}/\Gamma_{\GL(\bV/\wt{V}^\bot)}(p^l)$, where $\Gamma_{\GL(\bV/\wt{V}^\bot)}(p^l)$ equals $\mr{Im}\left(\Gamma_{\wt{V}}\cap\Gamma_{SP}(p^l)\ra \GL(\bV/\wt{V}^\bot)\right)$.

Denote by $\cI^r_{\sT^\Sigma_{SP,m,l}}$ (resp. $\cI^r_{\sT^\star_{SP,m,l}}$) the ideal sheaf attached to the union of all strata inside $\sT^{\Sigma}_{SP,m,l}$ (resp. $\sT^{\star}_{SP,m,l}$) with cusp labels of rank $> r$. The ideal sheaf $\cI^r_{\sT^\Sigma_{SP,m,l}}$ equals $f^*_{m,l}\cI^r_{X^{\Sigma,\ord}_m}$ as $f_{m,l}$ is \'{e}tale. 

Since $\pi_{\sT,*}\cO_{\sT^{\Sigma}_{SP,m,l}}=\cO_{\sT^{\star}_{SP,m,l}}$, applying $\pi_{\sT,\ast}$ to the short exact sequence
\begin{equation*}
   0\lra \cI^r_{\sT^\Sigma_{SP,m,l}} \lra \cO_{\sT^{\Sigma}_{SP,m,l}}\lra \iota_{r,*}\cO_{\coprod_{\wt{V}\in\fC_{\bV,\,p^l}/\Gamma\cap\Gamma_{SP}(p^l),\,\mr{rk}\wt{V}> r}Z_{\wt{V},m,l}}\lra 0,
\end{equation*}
we get
\begin{equation}\label{eq:Imsheaf}
   0\lra \pi_{\sT,\ast}\cI^r_{\sT^\Sigma_{SP,m,l}} \lra \cO_{\sT^{\star}_{SP,m,l}}\lra (\pi\circ\iota_{r})_*\cO_{\coprod_{\wt{V}\in\fC_{\bV,\,p^l}/\Gamma\cap\Gamma_{SP}(p^l),\,\mr{rk}\wt{V}\geq r}Z_{\wt{V},m,l}},
\end{equation}
where $\iota_r:\coprod_{\wt{V}\in\fC_{\bV,\,p^l}/\Gamma\cap\Gamma_{SP}(p^l),\,\mr{rk}\wt{V}> r}Z_{\wt{V},m,l}\ra \sT^\Sigma_{SP,m,l}$ is the canonical closed embedding. Since by definition the stratum $\sT_{\wt{V},m,l}$ is the image of $Z_{\wt{V},m,l}$, we see from \eqref{eq:Imsheaf} that
\begin{equation*}
    \cI^r_{\sT^\star_{SP,m,l}}=\pi_{\sT,\ast}\cI^r_{\sT^\Sigma_{SP,m,l}}.
\end{equation*}

The above description of the completion of $\sT^{\Sigma}_{SP,m,l}$ along $Z_{\wt{\sigma},m,l}$, combined with \eqref{eq:O} and \eqref{eq:I}, gives the following description of the fiber of the structure and ideal sheaves at a closed point $x\in \sT_{\wt{V},m,l}\subset\sT^\star_{SP,m,l}$,
\begin{align*}
   \numberthis\label{eq:piO}\left(\cO_{\sT^{\star}_{SP,m,l}}\right)^\wedge_x=\left(\pi_{\sT,*}\cO_{\sT^\Sigma_{SP,m,l}}\right)^{\wedge}_x&=\left(\bigcap_{\sigma\in C(\bV/\wt{V}^\bot)^\circ}\prod_{\lambda\in N^{-1}S^2(\bV/\wt{V}^\bot)\cap\,\sigma^\vee} \mr H^0(\wh{\cB}_{\wt{V},m,l,x},\,\cL(\lambda))\right)^{\Gamma_{\GL(\bV/\wt{V}^\bot)}(p^l)}\\
   &=\left(\prod_{\lambda\in N^{-1}S^2(\bV/\wt{V}^\bot)_{\geq 0}} \mr H^0(\wh{\cB}_{\wt{V},m,l,x},\,\cL(\lambda))\right)^{\Gamma_{\GL(\bV/\wt{V}^\bot)}(p^l)},
\end{align*}
and
\begin{align*}
  \numberthis\label{eq:piI} &\left(\cI^r_{\sT^{\star}_{SP,m,l}}\right)^\wedge_x=\left(\pi_{\sT,*}\cI^r_{\sT^\Sigma_{SP,m,l}}\right)^{\wedge}_x\\
  =&\left(\pi_{\sT,*}\cO_{\sT^\Sigma_{SP,m,l}}\right)^{\wedge}_x\cap \left(\bigcap_{\substack{\sigma\in C(\bV/\wt{V}^\bot)^\circ\\\mr{rk}\,\wt{V}_{\wt{\sigma}}>r}}\prod_{\lambda\in N^{-1}S^2(\bV/\wt{V}^\bot)\cap\,\sigma^{\vee,\circ}} \mr H^0(\wh{\cB}_{\wt{V},m,l,x},\,\cL(\lambda))\right)^{\Gamma_{\GL(\bV/\wt{V}^\bot)}(p^l)}\\
   =&\left(\prod_{\substack{\lambda\in N^{-1}S^2(\bV/\wt{V}^\bot)_{\geq 0}\\ \mr{rk}\, \lambda\geq \,\mr{rk}\,\wt{V}-r}} \mr H^0(\wh{\cB}_{\wt{V},m,l,x},\,\cL(\lambda))\right)^{\Gamma_{\GL(\bV/\wt{V}^\bot)}(p^l)},
\end{align*}
where $\wh{\cB}_{\wt{V},m,l,x}$ denotes the completion of $\cB_{\wt{V},m,l}$ along its fiber over $x$. 

\begin{rem}\label{rmk:Vbc}
The invertible sheaf $\cL(\lambda)$ is the pull-back of an ample line bundle on a quotient of the abelian scheme $\cB_{\wt{V},m,l}$. Thus in particular, taking the global sections commutes with base change (cf. \cite[p. 155]{FC}). Therefore \eqref{eq:piI} implies that for the ideal sheaf $\cI^r_{\sT^\star_{SP,m,l}}$, the push-forward $\pi_{\sT,*}$ commutes with the base change. Since $\sT^\star_{SP,m,l}$ is affine, we see that the base change property holds, \textit{i.e.}
\begin{equation}\label{eq:bcV}
   V^{SP,r}_{m,l}=V^{SP,r}_{m+1,l}\otimes\bZ/p^m.
\end{equation}
\end{rem}

\subsection{\texorpdfstring{The quotient $\left.V^{SP,r}_{m,l}\middle/V^{SP,r-1}_{m,l}\right.$}{The quotient Vr/Vr}}
Since $X^{\star,\ord}$ is affine, we have
$$ \left.V^{SP,r}_{m,l}\middle/V^{SP,r-1}_{m,l}\right.=\mr H^0\left(X^{\star,\ord}_m,\,\left.\pi_* f_{m,l,*}\cI^r_{\sT^\Sigma_{SP,m,l}}\middle/\pi_* f_{m,l,*}\cI^{r-1}_{\sT^\Sigma_{SP,m,l}}\right.\right).
$$
We need to analyze the quotient $\left.\pi_* f_{m,l,*}\cI^r_{\sT^\Sigma_{SP,m,l}}\middle/\pi_* f_{m,l,*}\cI^{r-1}_{\sT^\Sigma_{SP,m,l}}\right.$. It is easily seen that this quotient sheaf is supported on $\coprod\limits_{W\in\fC_{\bV}/\Gamma,\mr{rk} W\geq r} Y^{\ord}_{W,m}\subset X^{\star,\ord}_m$. For $x\in \sT_{\wt{W},m,l}$ with $\wt{W}\in\fC_{\bV,\,p^l}/\Gamma\cap\Gamma_{SP}(p^l)$ of rank $\geq r$, using \eqref{eq:piI} we get
\begin{align*}
   \numberthis\label{eq:piq} \left(\left.\pi_{\sT,*}\cI^r_{\sT^\Sigma_{SP,m,l}}\middle/\pi_{\sT,*}\cI^{r-1}_{\sT^\Sigma_{SP,m,l}}\right.\right)^{\wedge}_x&=\left(\prod_{\substack{\lambda\in N^{-1}S^2(\bV/\wt{W}^\bot)_{\geq 0}\\ \mr{rk}\lambda=\mr{rk}\wt{W}-r}} \mr H^0(\wh{\cB}_{\wt{W},m,l,x},\,\cL(\lambda))\right)^{\Gamma_{\GL(\bV/\wt{W}^\bot)}(p^l)}\\
   &=\left(\prod_{\substack{\wt{V}\in\fC_{\bV,\,p^l}/\Gamma\cap\Gamma_{SP}(p^l)\\\wt{V}\subset\wt{W},\,\mr{rk}\wt{V}=r}}\,\prod_{\substack{\lambda\in N^{-1}S^2(\bV/\wt{W}^\bot)_{\geq 0}\\ \mr{ker}\lambda=\wt{V}}}\mr H^0(\wh{\cB}_{\wt{W},m,l,x},\,\cL(\lambda))\right)^{\Gamma_{\GL(\bV/\wt{W}^\bot)}(p^l)}\\
   &=\prod_{\substack{\wt{V}\in\fC_{\bV,\,p^l}/\Gamma\cap\Gamma_{SP}(p^l)\\\wt{V}\subset\wt{W},\,\mr{rk}\wt{V}=r}}\left(\prod_{\lambda\in N^{-1}S^2(\wt{V}^\bot/\wt{W}^\bot)_{>0}} \mr H^0(\wh{\cB}_{\wt{W},\wt{V},m,l,x},\,\cL(\lambda))\right)^{\Gamma_{\GL(\wt{V}^\bot/\wt{W}^\bot)}(p^l)}.
\end{align*} 
Here $\cB_{\wt{W},\wt{V},m,l}$ is the abelian scheme over $\sT_{\wt{W},m,l}$ obtained as the quotient of $\cB_{\wt{W},m,l}$ by $\wt{V}$. It is $p$-power isogenous to 
\begin{equation*}
   \Hom_{\sT_{\wt{W},m,l}}(N^{-1}(\wt{V}^\bot/\wt{W}^\bot),\,\cA_{\sT_{\wt{W},m,l}}).
\end{equation*}
The invertible sheaf $\cL(\lambda)$ over $\cB_{\wt{W},\wt{V},m,l}$ with $\lambda\in N^{-1}S^2(\wt{V}^\bot/\wt{W})_{>0}$ is defined in the way as described in \S\ref{sec:Mum}. The group $\Gamma_{\GL(\wt{V}^\bot/\wt{W}^\bot)}(p^l)$ is the image of the stablilizer of $\wt{V}^\bot/\wt{W}^\bot$ inside $\Gamma_{\GL(\bV/\wt{W}^\bot)}(p^l)$. 

For each $\wt{V}\in\fC_{\bV,\,p^l}/\Gamma\cap\Gamma_{SP}(p^l)$, there is a closed embedding
\begin{equation*}
   \iota^\star_{\wt{V}}:\sT^\star_{\wt{V},m,l}\lhra \sT^\star_{SP,m,l},
\end{equation*}
where $\sT^\star_{\wt{V},m,l}$ is the partial minimal compactification of the stratum $\sT_{\wt{V},m,l}$. The image is the Zariski closure of the stratum $\sT_{\wt{V},m,l}$ inside $\sT^\star_{SP,m,l}$, which equals the union of all strata with cusp labels containing $\wt{V}$. Like before one can define the sheaf of ideals $\cI^s_{\sT^\star_{\wt{V},m,l}}$ for $0\leq r\leq \mr{rk}\,\wt{V}$.

We define the group $P^\circ_{n,r}(\bZ/p^l)$ as the image of the map
\begin{align*}
\Gamma_{V_r}\cap \Gamma_0(p^l)&\lra \GL(n,\bZ/p^l)\\
\begin{blockarray}{ccccc}
   r&n-r&r&n-r\\
   \begin{block}{(cccc)c}
   \alpha&u&\ast&\ast&r\\
   0&a&\ast&b&n-r\\
   0&0&\alpha^{-1}&0&r\\
   0&c&v&d&n-r\\
   \end{block}
\end{blockarray} &\longmapsto \begin{pmatrix}\alpha&u\\0&a\end{pmatrix}\mod p^l,
\end{align*}
which is easily seen equal to $\begin{pmatrix}\SL(r,\bZ/p^l)&\ast\\0&\GL(n-r,\bZ/p^l)\end{pmatrix}$.

\begin{prop} 
There are the following short exact sequences, 
\begin{equation}\label{eq:quot1}
   0\lra\pi_{\sT,*}\cI^{r-1}_{\sT^\Sigma_{SP,m,l}}\lra\pi_{\sT,*}\cI^{r}_{\sT^\Sigma_{SP,m,l}}\lra \bigoplus_{\substack{\wt{V}\in\fC_{\bV,\,p^l}/\Gamma\cap\Gamma_{SP}(p^l)\\\mr{rk}\,\wt{V}=r}}\iota^\star_{\wt{V},*}\cI^0_{\sT^\star_{\wt{V},m,l}}\lra 0,
\end{equation}
\begin{equation}\label{eq:quot2}
   0\lra\pi_{*}f_{m,l,*}\cI^{r-1}_{\sT^\Sigma_{SP,m,l}}\lra\pi_{*}f_{m,l,*}\cI^{r}_{\sT^\Sigma_{SP,m,l}}\lra \bigoplus_{\substack{V\in\fC_{\bV}/\Gamma\\\mr{rk}\,V=r}}\left(\bigoplus_{\wt{V}\in\fp_{\fC,l}^{-1}(V)}\iota^\star_{\wt{V},*}\cI^0_{\sT^\star_{\wt{V},m,l}}\right)\lra 0,
\end{equation}
where $\fp_{\fC,l}$ is the projection defined in \eqref{eq:pc} and
\begin{equation}\label{eq:quot3}
   \fp_{\fC,l}^{-1}(V)\simeq\Gamma_{\wt{V}}\cap\Gamma_0(p^l)\left\backslash \Gamma\cap\Gamma_0(p^l)\right/\Gamma\cap\Gamma_{SP}(p^l)\simeq P^\circ_{n,r}(\bZ/p^l)\left\backslash\GL(n,\bZ/p^l)\right/SP(\bZ/p^l).
\end{equation}
\end{prop}
\begin{proof}
The short exact sequence \eqref{eq:quot1} follows directly from our above description in \eqref{eq:piI} and \eqref{eq:piq} of the fibers of the relevant sheaves on the partial minimal compactification. The term at the right end is a direct sum because the intersection between $\sT^\star_{\wt{V},m,l}$ and $\sT^\star_{\wt{V}',m,l}$, $\wt{V}\neq\wt{V}'$, lies inside the closed subscheme defining the ideal sheaf $\cI^0_{\sT^\star_{\wt{V}m,l}}$. The exact sequence \eqref{eq:quot2} is obtained from \eqref{eq:quot1} by rewriting the term at the right end. 
\end{proof}

By taking global sections, \eqref{eq:quot2} gives
\begin{equation}\label{eq:quotV}
   0\lra V^{SP,r-1}_{m,l}\lra V^{SP,r}_{m,l}\lra \bigoplus_{\substack{V\in\fC_{\bV}/\Gamma\\\mr{rk}\,V=r}}\left(\bigoplus_{\wt{V}\in\fp_{\fC,l}^{-1}(V)} H^0\left(\sT^\star_{\wt{V}m,l},\,\cI^0_{\sT^\star_{\wt{V}m,l}}\right)\right)\lra 0.
\end{equation}
We see that the quotient $\left.V^{SP,r}_{m,l}\middle/V^{SP,r-1}_{m,l}\right.$ is a direct sum of cuspidal $p$-adic Siegel modular forms with $p^m$-torsion coefficients of level $l$ and degree $n-r$ with respect to certain parabolics. 

The action of the group $T_P(\bZ_p)$ permutes the summands of $\left.V^{SP,r}_{m,l}\middle/V^{SP,r-1}_{m,l}\right.$. There are too many summands in the quotient in order for it to form a nice $\bZ/p^m\llbracket T_P(\bZ_p)\rrbracket $-module after taking direct limit with respect to $l$, or in other words the structure of the $T_P(\bZ_p)$-action on \eqref{eq:quot3} is in some sense too complicated as $l$ grows. The idea is that we pick out a single $T_P\llbracket \bZ_p\rrbracket $-orbit from   \eqref{eq:quot3} which patch nicely with $l$ growing.

\subsection{\texorpdfstring{The space $\pV^{SP,r,\flat}$}{The space Vr,flat}}
For $V\in\fC_{\bV}/\Gamma$ of rank $r$, consider 
\begin{equation*}
   \sT_{Z^\ord_V,SP,m,l}=Z^\ord_{V,m}\times_{X^\Sigma_m} \sT^\Sigma_{SP,m,l}=\coprod_{\wt{V}\in \fp_{\fC,l}^{-1}(V)}Z_{\wt{V},m,l},
\end{equation*}
the restriction of the $SP$-Igusa tower to the stratum $Z^\ord_{V,m}$. It is not connected if the set $\fp_{\fC,l}^{-1}(V)\simeq P^\circ_{n,r}(\bZ/p^l)\left\backslash\GL(n,\bZ/p^l)\right/SP(\bZ/p^l)$ has more than one element. For $r\leq n_d$, we will define a subscheme $\sT^\flat_{Z^\ord_V,SP,m,l}\subset \sT_{Z^\ord_V,SP,m,l}$ consisting of certain connected components which form a single orbit for the  $T_P(\bZ_p)$-action. The space $V^{SP,r,\flat}_{m,l}$ will be defined as the subspace of $V^{SP,r}_{m,l}$ consisting of sections vanishing outside $\sT^\flat_{Z^\ord_V,SP,m,l}$.

Recall that the semi-abelian scheme $\mc{G}_{/\wh{\mc{M}}_{X^*_r,m,l,\sigma}}$ in the Mumford construction carries canonical level structures
\begin{align*}
   \psi_{N,\mr{can}}:(\mb{Z}/N\mb{Z})^{2n}&\lra\mc{G}_{/\wh{\mc{M}}_{X^*_r,m,l,\sigma}}[N], &\phi_{p,\mr{can}}:\mu^n_{p^l}&\stackrel{\sim}{\lra}\mc{G}_{/\wh{\mc{M}}_{X^*_r,m,l,\sigma}}[p^l]^\circ.
\end{align*}
We assume that if $\wt{V}=V_r$, the standard submodule of $\bV$ of rank $r$, and $\wt{\sigma}\in C(\bV/\wt{V}^\bot)$, then the restriction of the semi-abelian scheme $(\cG_{/\sT^\Sigma_{SP,m,l}},\psi_N,(E_i,\varepsilon_i)_{1\leq i\leq d,\,p^l})$ to the formal completion along $Z^\ord_{\wt{\sigma},m,l}$ is isomorphic to the one induced from $(\mc{G}_{/\wh{\mc{M}}_{X^*_r,m,l,\wt{\sigma}}}, \,\psi_{N,\mr{can}},\,\phi_{p,\mr{can}})$ (in other words, the level structures parametrized by cusps at infinity are the canonical ones).

Then for $\gamma=\begin{pmatrix}a_\gamma&b_\gamma\\c_\gamma&d_\gamma\end{pmatrix}\in\Gamma_0(p^l)$, $\wt{V}=\gamma\cdot V_r$ and $\wt{\sigma}=\gamma\cdot \sigma\in C(\bV/\wt{V}^\bot)$ for some $\sigma\in\Sigma$, the restriction of $\left(\cG_{/\sT^\Sigma_{SP,m,l}},\psi_N,(E_i,\varepsilon_i)_{1\leq i\leq d,\,p^l}\right)$ to the formal completion along $Z^\ord_{\wt{\sigma},m,l}$ is isomorphic to the one induced from
\begin{equation*}
   \left(\mc{G}_{/\wh{\mc{M}}_{X^*_r,m,l,\wt{\sigma}}}, \,\psi_{N,\mr{can}}\circ\gamma,\,\phi_{p,\mr{can}}\circ a_\gamma\right).
\end{equation*}

If we fix $V\in\fC_{\bV,p^l}$, the connected components of $\sT_{Z^\ord_V,SP,m,l}$ can be thought of in terms of the relation between the two-step filtration of $\cG_{/Z^\ord_{V,m}}[p^\infty]^\circ$ induced from 
\begin{equation}\label{eq:sabfil}
   0\lra V\otimes\bG_m\lra \cG_{/Z_{V},m}[p^\infty]^\circ \lra \cA_{/Y_V}\times_{Y_L}Z_V\lra 0,
\end{equation}
and the $d$-step filtration 
$$\{0\}=E_{\wt{V},0}\subset E_{\wt{V},1}\subset\cdots\subset E_{\wt{V},d}=\mc{G}_{/Z^\ord_{V},m}[p^l]^\circ$$
induced from the universal object $\left(\cG_{/\sT^\Sigma_{SP,m,l}},\psi_N,(E_i,\varepsilon_i)_{1\leq i\leq d,\,p^l}\right)$ restricted to $Z_{\wt{V},m,l}\subset \sT^\Sigma_{SP,m,l}$.

From now on assume $r\leq n_d$. Define 
\begin{align*}
   \fp^{-1}_{\fC,l}(V)^\flat&=\left\{\wt{V}\in\fp^{-1}_{\fC,l}(V):\, E_{\wt{V},d-1}\cap V\otimes \mu_{p^l}=0\right\},\\
   \sT^\flat_{Z^\ord_V,SP,m,l}&=\coprod_{\wt{V}\in\fp^{-1}_{\fC,l}(V)^\flat} Z_{\wt{V},m,l}\subset \sT_{Z^\ord_V,SP,m,l},
\end{align*}
{\it i.e.} the union of the connected components of $\sT_{Z^\ord_V,SP,m,l}$ for which the first $d-1$ steps of the parametized filtrations of $\cG_{/Z^\ord_{V,m}}[p^l]^\circ$ intersect trivially with the $p^l$-torsion of the torus part in \eqref{eq:sabfil}.

Under the natural map
\begin{align*}
   \fp_{\fC,l}^{-1}(V)&\lra SP(\bZ/p^l)\left\backslash\GL(n,\bZ/p^l)\right/P^\circ_{n,r}(\bZ/p^l)\\
   \wt{V}=\gamma\cdot V_r\quad(\gamma\in\Gamma_0(p^l))&\longmapsto a_\gamma,
\end{align*}
the set $\fp^{-1}_{\fC,l}(V)^\flat$ corresponds to
\begin{equation*}
   SP(\bZ/p^l)\left\backslash SP(\bZ/p^l)\begin{pmatrix}0&I_r\\I_{n-r}&0\end{pmatrix}P_{n,r}(\bZ/p^l)\right/P^\circ_{n,r}(\bZ/p^l)\subset SP(\bZ/p^l)\left\backslash\GL(n,\bZ/p^l)\right/P^\circ_{n,r}(\bZ/p^l),
\end{equation*}
with $P_{n,r}=\begin{pmatrix}\GL(r,\bZ/p^l)&M_{r,n-r}(\bZ/p^l)\\0&\GL(n-r,\bZ/p^l)\end{pmatrix}$. The action of $T_P(\bZ_p)$ on $\fp^{-1}_{\fC,l}(V)^\flat$ is transitive, and we have
\begin{align*}
  \numberthis\label{eq:index} \fp^{-1}_{\fC,l}(V)^\flat
   =\left\{\begin{array}{ll}\left\{\begin{pmatrix}0&I_r\\I_{n-r}&0\end{pmatrix}\right\},&\text{ if }r< n_d,\\ \begin{pmatrix}0&I_r\\I_{n-r}&0\end{pmatrix}\begin{pmatrix}I_{N_{d-1}}&0\\0&\GL(n_d,\bZ/p^l)/\SL(n_d,\bZ/p^l)\end{pmatrix}\simeq\left(\bZ/p^l\right)^\times,&\text{ if } r=n_d.\end{array}\right.
\end{align*}

We now define $\cI^{r,\flat}_{\sT^\Sigma_{SP,m,l}}\subset\cI^r_{\sT^\Sigma_{SP,m,l}}$ to be the sheaf of ideals associated to the closed subscheme given as the complement of $\coprod\limits_{V\in\fC_{\bV}/\Gamma,\,\mr{rk}{V}<r}\sT_{Z^\ord_V,SP,m,l}\cup\coprod\limits_{V\in\fC_{\bV}/\Gamma,\,\mr{rk}{V}=r}\sT^\flat_{Z^\ord_V,SP,m,l}$ inside $\sT^\Sigma_{SP,m,l}$, and define
\begin{align*}
   V^{SP,r,\flat}_{m,l} & =\mr{H}^0\left(\sT^{\Sigma}_{SP,m,l},\,\cI^{r,\flat}_{\sT^\Sigma_{SP,m,l}} \right)\subset V^{SP,r}_{m,l},\\
	 \pV^{SP,r,\flat} & =\varinjlim\limits_{m}\varinjlim\limits_l V^{SP,r,\flat}_{m,l}\subset\pV^{SP,r}.
\end{align*}
If follows from the definition and \eqref{eq:quotV} that
\begin{equation}\label{eq:quotVb}
   \left.V^{SP,r,\flat}_{m,l}\middle/V^{SP,r-1}_{m,l}\right.=\bigoplus_{\substack{V\in\fC_{\bV}/\Gamma\\\mr{rk}\,V=r}}\left(\bigoplus_{\wt{V}\in\fp_{\fC,l}^{-1}(V)^\flat} \mr H^0\left(\sT^\star_{\wt{V}m,l},\,\cI^0_{\sT^\star_{\wt{V}m,l}}\right)\right).
\end{equation}
The natural $T_P(\bZ_p)$-action on the left hand side induces a $T_P(\bZ_p)$-action on
\begin{equation}\label{eq:summand}
   \bigoplus_{\wt{V}\in\fp_{\fC,l}^{-1}(V)^\flat} \mr H^0\left(\sT^\star_{\wt{V}m,l},\,\cI^0_{\sT^\star_{\wt{V}m,l}}\right).
\end{equation}

Let 
\begin{align*}
   \numberthis P_{n-r}&=\left\{\left.\begin{pmatrix}a_1&*&*\\&\ddots&*\\&&a_d\end{pmatrix}\in\mr{GL}(n-r)\,\right|\,a_i\in\mr{GL}(n_i),\,1\leq i\leq d-1,\,a_d\in\GL(n_d-r)\right\},\\
   \numberthis SP_{n-r}&=\left\{\left.\begin{pmatrix}a_1&*&*\\&\ddots&*\\&&a_d\end{pmatrix}\in\mr{SL}(n-r)\,\right|\,a_i\in\mr{SL}(n_i),\,1\leq i\leq d-1,\,a_d\in\SL(n_d-r)\right\},\\
  \numberthis\label{eq:qT} T_{P_{n-r}}&=P_{n-r}/SP_{n-r}=\left\{\begin{array}{ll}\bG^d_m&\text{ if }r< n_d,\\ \bG^{d-1}_m&\text{ if }r=n_d.\end{array}\right.
\end{align*}
We know that for each $\wt{V}\in\fp^{-1}_{\fC,l}(V)^\flat$, we have
\begin{equation}\label{eq:qlevel}
   \mr{Im}\left(\Gamma_{\wt{V}}\cap \Gamma_{SP}(p^l)\ra\Sp(\wt{V}^\bot/\wt{V},\bZ/p^l)\right)\simeq \Gamma(N)\cap SP_{n-r}(\bZ).
\end{equation}

The embedding $P_{n-r}\hra P$ induces a morphism $T_{P_{n-r}}\ra T_P$, and the induced action of $T_{P_{n-r}}(\bZ_p)$ on  \eqref{eq:summand} preserves each direct summand, so equips each $\mr H^0\left(\sT^\star_{\wt{V}m,l},\,\cI^0_{\sT^\star_{\wt{V}m,l}}\right)$, $\wt{V}\in\fp^{-1}_{\fC,l}(V)^\flat$, with an $\cO_F\llbracket T_{P_{n-r}}(\bZ_P)\rrbracket $-module structure.

From \eqref{eq:qlevel}, we also know that for each $\wt{V}\in\fp^{-1}_{\fC,l}(V)^\flat$, the scheme $\sT^\star_{\wt{V},m,l}$ is isomorphic to the minimal compactification of the quotient by $SP_{n-r}(\bZ/p^l)$ of the full level $p^l$ Igusa tower over $Y^\ord_{V,m}$. Denote by $V^{SP_{n-r},0}_{V,m,l}$ the space of cuspidal sections over that Igusa tower over $Y^\ord_{V,m}$, which carries a natural $T_{P_{n-r}}(\bZ_p)$-action. 

Then 
\begin{equation*}
 \mr  H^0\left(\sT^\star_{\wt{V}m,l},\,\cI^0_{\sT^\star_{\wt{V}m,l}}\right)\simeq V^{SP_{n-r},0}_{V,m,l}
\end{equation*} 
as $\cO_F\llbracket T_{P_{n-r}}(\bZ_p)\rrbracket $-modules. Furthermore, we have
\begin{equation*}
   \bigoplus_{\wt{V}\in\fp_{\fC,l}^{-1}(V)^\flat} \mr H^0\left(\sT^\star_{\wt{V}m,l},\,\cI^0_{\sT^\star_{\wt{V}m,l}}\right)\simeq \bZ_p\llbracket T_P(\bZ_p)\rrbracket \otimes_{\bZ_p\llbracket T_{P_{n-r}}(\bZ_p)\rrbracket } V^{SP_{n-r},0}_{V,m,l},
\end{equation*}
because by \eqref{eq:index} and \eqref{eq:qT} the action of $T_P(\bZ/p^l)/T_{P_{n-r}}(\bZ/p^l)$ on $ \fp_{\fC,l}^{-1}(V)^\flat$ is simply transitive. 

Summarizing the above discussion, if we let \[\pV^{SP_{n-r},0}_V = \varinjlim\limits_{m}\varinjlim\limits_l V^{SP_{n-r},0}_{V,m,l}\]  we get
\begin{prop}\label{prop:exact}
There are the following short exact sequences of $\cO_F\llbracket T_P(\bZ_p)\rrbracket $-modules,
\begin{equation}\label{eq:exactf}
   0\lra V^{SP,r-1}_{m,l}\lra V^{SP,r,\flat}_{m,l}\lra \bigoplus_{\substack{V\in\fC_{\bV}/\Gamma\\\mr{rk}\,V=r}} \bZ_p\llbracket T_P(\bZ_p)\rrbracket \otimes_{\bZ_p\llbracket T_{P_{n-r}}(\bZ_p)\rrbracket } V^{SP_{n-r},0}_{V,m,l}\lra 0,
\end{equation}
\begin{equation}\label{eq:exact1}
   0\lra\pV^{SP,r-1}\lra\pV^{SP,r,\flat}\lra \bigoplus_{\substack{V\in\fC_{\bV}/\Gamma\\\mr{rk}\,V=r}} \bZ_p\llbracket T_P(\bZ_p)\rrbracket \otimes_{\bZ_p\llbracket T_{P_{n-r}}(\bZ_p)\rrbracket } \pV^{SP_{n-r},0}_V\lra 0.
\end{equation}
\end{prop}

\subsection{\texorpdfstring{The $q$-expansions}{The q-expansions}}\label{sec:qexp}
Later our analysis of the action of the $\bU^P_{p}$-operators on $V^{SP,r,\flat}_{m,l}$, $\pV^{SP,r,\flat}$ will mostly rely on $q$-expansions.

Specializing the construction in \S\ref{sec:Mum} to the case $r=n$, for $\gamma_N\in\mr{Sp}(2n,\mb{Z}/N\mb{Z})$ and $a_p\in\mr{GL}(n,\mb{Z}_p)$, the evaluation at the testing object 
\begin{equation*}
   \left(\mc{G}_{/\wh{\mc{M}}_{X^*_n,m,l,\sigma}}, \,\psi_{N,\mr{can}}\circ\gamma_N,\,\phi_{p,\mr{can}}\circ a_p\right),\quad\sigma\in\Sigma,
\end{equation*}
defines the $q$-expansion map 
\begin{equation*}
    \varepsilon^{\gamma_N,a_p}_{\qexp,m,l}:V_{m,l}\lra\bigcap_{\sigma\in\Sigma}\cO_F/p^m\llbracket N^{-1}S^2(X_n)\cap\sigma^\vee\rrbracket =\cO_F/p^m\llbracket N^{-1}S^2(X_n)_{\geq 0}\rrbracket .
\end{equation*}
These $\varepsilon^{\gamma_N,a_p}_{\qexp,m,l}$'s glue to the $q$-expansion map on $\pV$,
\begin{equation}\label{eq:qexp}
   \varepsilon^{\gamma_N,a_p}_{\qexp}:\pV\lra F/\mc{O}_F\llbracket N^{-1}S^2(X_n)_{\geq 0}\rrbracket =F/\mc{O}_F\llbracket N^{-1}\Sym(n,\bZ)^*_{\geq 0}\rrbracket .
\end{equation}
With our fixed basis $x_1,\cdots,x_n$ of $X_n$, we will freely identify $S^2(X_n)$ with $\mr{Sym}(n,\mb{Z})^*$, the set of symmetric $n\times n$ matrices with integers as diagonal entries and half-integers as off-diagonal entries, by identifying $\beta\in\mr{Sym}(n,\mb{Z})$ with $\sum\limits_{1\leq i,j\leq n}\beta_{ij} x_i\otimes x_j$. For $\beta\in N^{-1}S^2(X_n)$ and $f\in\pV$, write $\varepsilon^{\gamma_N,a_p}_{\qexp}(\beta,f)$ for the $\beta$-th Fourier coefficient of $f$, {\it i.e.} the coefficient associated with $\beta$ in $\varepsilon^{\gamma_N,a_p}_{\qexp}(f)$. One can check that given $a\in\mr{GL}(n,\mb{Z})$
\begin{equation}\label{eq:aqexp}
   \varepsilon^{m(a)\gamma_N,aa_p}_{\qexp}(\beta,f)=\varepsilon^{\gamma_N,a_p}_{\qexp}(\,\ltrans{a}\beta a,f),
\end{equation}
where $m(a)=\begin{pmatrix}a&0\\0&\ltrans{a}^{-1}\end{pmatrix}$.

As illustrated in \cite[V Lemma 1.4 and its proof]{FC}, since the closure of every stratum associated with a cone in the toroidal compactification is irreducible and contains a stratum corresponding to a top dimensional cone, many properties of ($p$-adic) Siegel modular forms can be verified by examining the $q$-expansion. 

\begin{rem}
The method here can be generalized to other PEL type Shimura varieties. When $q$-expansions are not available, one needs prove the desired properties of $\bU_p$-operators in \S\ref{sec:Up} by working with Fourier--Jacobi expansions.
\end{rem}

The following two propositions give a characterization of the space $\pV^{SP,r}$, $\pV^{SP,r,\flat}$ in terms of $q$-expansions.

\begin{prop}\label{prop:Vr}
Given $f\in\pV$, it belongs to $\pV^{SP,r}$ if and only if $\varepsilon^{\gamma_N,a_p}_{\qexp}(\beta,f)$ vanishes for all $\gamma_N\in\mr{Sp}(2n,\mb Z)$, $a_p\in\GL(n,\bZ_p)$ and $\beta\in N^{-1}S^2(X_n)_{\geq 0}$ of rank less or equal to $n-r-1$. 
\end{prop}
\begin{proof}
Given $\wt{\sigma}\in\Sigma_{\mc C_{\mb V,p^l}}$, pick a top dimensional cone $\wt{\tau}\in\Sigma_{\mc C_{\mb V,p^l}}$ with $\wt{\sigma}$ as a face. Fix an isomorphism of $\mathrm{Spf}\left(\mc O_F/p^m\mc O_F\llbracket N^{-1}S^2(X_n)\cap\wt{\tau}^\vee\rrbracket \right)$ with the completion of $\sT^\Sigma_{SP,m,l}$ along the point $Z_{\wt{\tau},m,l}$ (here for a cone in $\Sigma_{\mc C_{\mb V,p^l}}$ we use the same notation to denote a corresponding cone in $C(X_n)$). Then the embedding of the completion of the Zariski closure of $Z_{\wt{\sigma},m,l}$ along the point $Z_{\wt{\tau},m,l}$ to the completion of $\sT^\Sigma_{SP,m,l}$ along $Z_{\wt{\tau},m,l}$ corresponds to the quotient map from $\mc O_F/p^m\mc \llbracket N^{-1}S^2(X_n)\cap\wt{\tau}^\vee\rrbracket $ onto $\mc O_F/p^m\mc \llbracket N^{-1}S^2(X_n)\cap\wt{\tau}^\vee\cap\wt{\sigma}^\bot\rrbracket $, sending all $\beta\in N^{-1}S^2(X_n)\cap\wt{\tau}^\vee$ that does not belong to $\wt{\sigma}^\bot$ to $0$. This description shows that the vanishing condition in the proposition implies the vanishing of $f$ along $Z_{\wt{\sigma},m,l}$, and the proposition follows.
\end{proof}

In the following, by the radical of $\beta\in N^{-1}S^2(X_n)_{\geq 0}$, we mean the sub-$\bZ$-module of $X^*_n$ consisting of elements that pair trivially with $\beta$ via the natural map $X^*_n\times S^2(X_n)\ra X_n$, and by a primitive vector in $X^*_n$, we mean an element not divisible by $p$ in $X^*_n$.

\begin{prop}\label{prop:Vrb}
Given $f\in\pV^{SP,r}$, it belongs to $\pV^{SP,r,\flat}$ if and only if $\varepsilon^{\gamma_N,a_p}_{\qexp}(\beta,f)$ vanishes for all $\beta$ of corank $r$ such that the radical of $\ltrans{a_p}\beta a_p$ contains a primitive vector inside $\mb Z\cdot x^*_1 + \cdots + \mb Z\cdot x^*_{N_{d-1}} + p\mb Z\cdot x^*_{N_{d-1}+1} +\cdots+ p\mb Z\cdot x^*_{n}$. 
\end{prop}
\begin{proof}
We use the description of the completion of the Zariski closure of $Z_{\wt{\sigma},m,l}$ along the point $Z_{\wt{\tau},m,l}$ given in the proof of the previous proposition, and assume that $\wt{V}_{\wt{\sigma}}$ is of rank $r$. Identify $V_n$ and $X^*_n$ (together with standard basis). Take a $\gamma\in\Gamma_0(p^l)$ such that $\wt{V}_{\wt{\tau}}=\gamma^{-1}\cdot V_n$ and use it  to fix an isomorphism between $\mc O_F/p^m\mc O_F\llbracket N^{-1}S^2(X_n)\cap\wt{\tau}^\vee\rrbracket $ and the formal completion of the structure sheaf at the point $Z_{\wt{\tau},m,l}$. Then the evaluation of $f$ at the formal neighborhood of $Z_{\wt{\tau},m,l}$ corresponds to the $q$-expansion $\varepsilon^{\gamma,a_\gamma}_{\qexp}(\beta,f)$, and the $E_{\wt{V}_{\wt{\tau}},d-1}$  corresponds to the $\bZ$-span of $a_\gamma(x^*_1)/p^l,\ldots,a_\gamma(x^*_{N_{d-1}})/p^l$ (recall that $E_{\wt{V}_{\wt{\tau}},d-1}$ is the $d-1$-th step of the filtration in the level structure of the $SP$-Igusa tower). On the other hand, the canonical two-step filtration of the semi-abelian scheme over $Z_{\wt{\sigma},m,l}$ corresponds to $\wt{V}_{\wt{\sigma}}\subset \wt{V}_{\wt{\tau}}$. Therefore the vanishing condition in the definition of $\pV^{SP,r,\flat}$ requires the vanishing of $\varepsilon^{\gamma,a_\gamma}_{\qexp}(\beta,f)$ for all $\beta\in N^{-1}S^2(X_n)\cap \wt{\tau}^\vee\cap\wt{\sigma}^\bot$ with $\wt{V}_{\wt{\sigma}}$ containing a primitive element in $a_\gamma\left(\mb Z\cdot x^*_1 + \cdots + \mb Z\cdot x^*_{N_{d-1}} + p\mb Z\cdot x^*_{N_{d-1}+1} +\cdots+ p\mb Z\cdot x^*_{n}\right)$. Also, for a semi-positive definite $\beta$ inside $\wt{\sigma}^\bot$, the radical of $\beta$ equals $\wt{V}_{\wt{\sigma}}$. Hence the vanishing condition in the proposition agrees with that for defining $\pV^{SP,r,\flat}$.
\end{proof}

\subsection{\texorpdfstring{The $\mb U^P_p$-operators}{The Up-operators}}\label{sec:Up}
To each matrix
\begin{align*}
\gamma_{p,i}=\begin{pmatrix}
  p I_{i} &0 &0 &0\\
	0&  I_{n-i} & 0& 0\\
	0&0 & p^{-1}I_{i} & 0\\
	0& 0&0 & I_{n-i}
\end{pmatrix},
&& 1\leq i\leq n,
\end{align*}
corresponds a Hecke operator $U^P_{p,i}$ acting on $\pV^{SP}$. The ordinarity condition for $P$ requires the eigenvalues of $U^P_{p,N_1}, U^P_{p,N_2},\dots, U^P_{p,N_d}$ to be $p$-adic units (recall that $N_i=\sum_{j=1}^i n_j$). In \cite[\S5.1.4]{PilHida}, only these $U^P_{p,N_1}, U^P_{p,N_2},\dots, U^P_{p,N_d}$ are introduced as they are sufficient for defining the ordinary projection in order to establish Hida theory. However, given an automorphic representation $\pi$ of $\Sp(2n,\bA)$ generated by a holomorphic Siegel modular form ordinary for the parabolic $P$, in order to retrieve the full information on $\pi_p$, one needs to consider the action of all the $U^P_{p,i}$, $1\leq i\leq n$ (see \S\ref{sec:Ep} for details). If $i\neq N_1,\dots,N_d$, the eigenvalue of $U^P_{p,i}$ on $P$-ordinary forms is not necessarily a $p$-adic unit.

When shall use the expression $\mb U^P_p$-operators when we do not need to specify which Hecke operators at $p$ we are using.

Let $\sT^\circ_{SP,m,l}$ be the restriction of $\sT_{SP,m,l}$ to $Y^\ord_m\subset X^{\Sigma,\ord}_m$. The algebraic correspondence inside $\sT^\circ_{SP,m,l}\times \sT^\circ_{SP,m,l}$ associated to $\gamma_{p,i}$ is defined as follows. For $N_j\leq i< N_{j+1}$, let $C_{i,m,l}$ be the moduli scheme over $\cO_F/p^m$ parametrizing the quintuple $\left(A,\lambda,\psi_N,(E_i,\varepsilon_i)_{1\leq i\leq d,\,p^l},L\right)$, where $\left(A,\lambda,\psi_N\right)$ is an ordinary abelian scheme of genus $n$ with principal polarization $\lambda$ and principal level structure $\psi_N:(\bZ/N)^{2n}\stackrel{\sim}{\ra}A[N]$, defined over an $\cO_F/p^m$-algebra, $(E_r,\varepsilon_r)_{1\leq r\leq d,\,p^l}$ is the structure used to define $\sT_{SP,m,l}$ in \S\ref{sec:Igusa}, and $L\subset A[p^2]$ is a Lagrangian subgroup such that $\mr{rank}_{\bZ/p}L[p]=2n-i$, $L[p]\cap E_j[p]=0$, $L[p]+E_{j+1}[p]=A[p]$. Denote by $p_1$ the projection from $C_{i,m,l}$ to $\sT_{SP,m,l}$ which forgets $L$. There is another projection $p_2$ sending $\left(A,\lambda,\psi_N,(E_r,\varepsilon_r)_{1\leq r\leq d,\,p^l},L\right)$ to $\left(A/L,\lambda',p\circ\pi\circ\psi_N,(E'_r,\varepsilon'_r)_{1\leq r\leq d,\,p^l}\right)$, where $\pi:A\ra A/L$ is the natural isogeny, $\lambda'$ is defined by $\pi^*\lambda'=p^2\lambda$, and
\begin{align*}
   E'_r&=\pi(E_r), &\varepsilon'_r&=\pi\circ\varepsilon, &1\leq r\leq j&\\
   E'_r&=\pi\left(p^{-1}(E_r\cap p^{-l+1}L)\right), &\varepsilon'_r&=p^{-\mr{min}\{N_r-i,\,n_r\}}\pi\circ\varepsilon_r, &j+1\leq r\leq d.&
\end{align*}

For $N_j\leq i <N_{j+1}$, we have the following composition
\begin{equation*}
   H^0\left(\sT^\circ_{SP,m,l},\cO_{\sT^\circ_{SP,m,l}}\right)\stackrel{p^*_2}{\lra}H^0\left(C_{i,m,l},\cO_{C_{i,m,l}}\right)\stackrel{\mr{Tr}p_1}{\lra}p^{i(n+1)}H^0\left(\sT^\circ_{SP,m,l},\cO_{\sT^\circ_{SP,m,l}}\right).
\end{equation*}
The image of $\mr{Tr}p_1$ belongs to $p^{i(n+1)}H^0\left(\sT^\circ_{SP,m,l},\cO_{\sT^\circ_{SP,m,l}}\right)$ because the pure inseparability degree of $p_1$ is $p^{i(n+1)}$ \cite[Appendice]{PilHida}. One can also check (for example by $q$-expansions) that such defined $U^P_{p,i}$ preserves various kinds of growth conditions along the boundary, i.e. the above map restricts to a map from $V^{SP}_{m,l}$ to $p^{i(n+1)}V^{SP}_{m,l}$. If $m>i(n+1)$, there is a well defined map $p^{-i(n+1)}:p^{i(n+1)}V^{SP}_{m,l}\ra V^{SP}_{m-i(n+1),l}$. Now given $f\in V^{SP}_{m,l}$, thanks to \eqref{eq:bcV}, we can take $\wt{f}\in V^{SP}_{m+i(n+1)}$ such that $f\equiv \wt{f}\mod p^{i(n+1)}$, and we define
\begin{equation*}
   U^P_{p,i}(f)=p^{-i(n+1)}\circ\mr{Tr}p_1\circ p^*_2 (\wt{f}).
\end{equation*}

In this section, only $U^P_{p,N_1}, U^P_{p,N_2},\dots, U^P_{p,N_d}$ will be used. In order to show the desired properties of their action on $\pV^{SP,r,\flat}$, $\pV^{SP,r}$, we use the following proposition and Proposition \ref{prop:Vr}, \ref{prop:Vrb} to reduce to computations on $q$-expansions. 

\begin{prop}[cf. {\cite[Proposition 3.5]{HPEL}}]\label{prop:UpF}
For $f\in\pV^{SP}$, $\gamma_N\in\Sp(2n,\bZ)$ and $a_p\in T(\bZ_p)\subset\GL(n,\bZ_p)$, the formula on $q$-expansions for the action of the $\mb U^P_p$-operators on $f$ is given by
\begin{equation*}
   \varepsilon^{\gamma_N,a_p}_{\qexp}(\beta,\,U^P_{p,N_i}f)=\sum_{x\in M_{N_i,n-N_i}(\mb Z/p\mb Z)}\varepsilon^{(\gamma^P_{p,i})^{-1}\gamma_N,a_p}_{\qexp}\left(\begin{pmatrix}pI_{N_i}&0\\N\ltrans{x}& I_{n-N_i}\end{pmatrix}\beta\begin{pmatrix}p{ I}_{N_i}&Nx\\0& I_{n-N_i}\end{pmatrix},\,f\right),
\end{equation*} 
for $\beta\in N^{-1}S^2(X_n)$ and $1\leq i\leq d$.
\end{prop}
One can also write down the formula for general $a_p\in\GL(n,\bZ_p)$ which is a little bit more complicated. We omit it here because the case $a_p$ being diagonal suffices for our purpose thanks to \eqref{eq:aqexp}.

\begin{prop}\label{prop:UpStab}
All the spaces $\pV^{SP,r}$, $0\leq r\leq n$, and $\pV^{SP,r,\flat}$, $0\leq r\leq n_d$, are stable under the $\mb U^P_p$-operators. 
\end{prop}
\begin{proof}
The statement for $\pV^{SP,r}$ follows immediately from Proposition \ref{prop:Vr}, \ref{prop:UpF}. By Proposition \ref{prop:Vrb}, \ref{prop:UpF}, in order to show the statement for $\pV^{SP,r,\flat}$, it is enough to show that if the radical of $\beta$ contains a primitive vector inside $\mb Z\cdot x^*_1 + \cdots + \mb Z\cdot x^*_{N_{d-1}} + p\mb Z\cdot x^*_{N_{d-1}+1} +\cdots+ p\mb Z\cdot x^*_{n}$, then the same holds for $\begin{pmatrix}pI_{N_i}&0\\N\ltrans{x}& I_{n-N_i}\end{pmatrix}\beta\begin{pmatrix}p{ I}_{N_i}&Nx\\0& I_{n-N_i}\end{pmatrix}$. In fact, it is not difficult to check that if $v\in \mb Z\cdot x^*_1 + \cdots + \mb Z\cdot x^*_{N_{d-1}} + p\mb Z\cdot x^*_{N_{d-1}+1} +\cdots+ p\mb Z\cdot x^*_{n}$ is a primitive vector, then for all $x\in M_{N_i,n-N_i}(\bZ)$, 
\begin{equation*}
   \bQ\cdot\begin{pmatrix}pI_{N_i}&Nx\\0&I_{n-N_i}\end{pmatrix}^{-1} v_\beta\cap X_n\subset  \mb Z\cdot x^*_1 + \cdots +\mb Z\cdot x^*_{N_{d-1}} + p\mb Z\cdot x^*_{N_{d-1}+1} +\cdots+ p\mb Z\cdot x^*_{n}.
\end{equation*}
\end{proof}

Now we want to define a $\bU^P_p$-action on the quotient of the exact sequences in Proposition \ref{prop:exact}, and verify that the exact sequences are $\bU^P_p$-equivariant.

For $V\in\fC_\bV$ with rank $r\leq n_d$, we define the $\bU^P_p$-action on $\pV^{SP_{n-r},0}_V$ as follows. For $\gamma=\mr{diag}\left(a_1,\dots,a_n,a^{-1}_1,\dots,a^{-1}_n\right)\in\Sp(2n)$, set $\gamma'=\mr{diag}\left(a_1,\dots,a_{n-r},a^{-1}_1,\dots,a^{-1}_{n-r}\right)\in\Sp(2n-2r)$. We make $U^P_{p,N_i}$ act on $\pV^{SP_{n-r},0}_V$ (the space of $p$-adic Siegel modular forms of degree $n-r$ for the parabolic $P_{n-r}$) by the $\bU^{P_{n-r}}_p$-operator attached to $\gamma^{\prime}_{p,N_i}$.

Let us denote by $\mb U^{P,[N]}_p\subset\bU^{P}_p$ the subalgebra generated by the $\varphi(N)$-powers of $U^P_{p,N_i}$, $1\leq i\leq d$. Here $\varphi(N)=N\cdot \prod_{q\text{ prime factors of }N}(1-\frac{1}{q})$.  Rather than showing the $\bU^P_p$-equivariance of the exact sequences in Proposition \ref{prop:exact}, we are only able to show the $\mb U^{P,[N]}_p$-equivariance. However, this suffices for establishing Hida theory for $\pV^{SP,r}$.

\begin{prop}\label{prop:UpEquiv}
The exact sequences in Proposition \ref{prop:exact} are $\mb U^{P,[N]}_p$-equivariant.
\end{prop}
\begin{proof}
We show the $\mb U^{P,[N]}_p$-equivariance of the projection $\fp_{\wt{V}}:V^{SP,r,\flat}_{m,l}\ra V^{SP_{n-r},0}_{V,m,L}$, from $V^{SP,r,\flat}_{m,l}$ to the summand of $\left.V^{SP,r,\flat}_{m,l}\middle/V^{SP,r-1}_{m,l}\right.$ indexed by $\wt{V}\in\fp_{\fC,l}^{-1}(V)^\flat$, by computing the $q$-expansions. Pick $\gamma\in\Gamma_0(p^l)$ such that $\gamma^{-1}V_n$ contains $\wt{V}$ (where $V_n$ is identified with $X^*_n$ with standard basis), and we view $\wt{V}$ as a subspace of $X^*_n$ via $\gamma$. Then $\wt{V}\in\fp_{\fC,l}^{-1}(V)^\flat$ implies that $\wt{V}$ is spanned by
\begin{equation*}
   (x_1,\dots,x_n)\begin{pmatrix}I_{N_{d-1}}&0\\0&w\end{pmatrix}\begin{pmatrix} \alpha_1\\\alpha_2\end{pmatrix},
\end{equation*}
with $w\in\SL(n_d,\bZ)$, $w\equiv I_{n_d}\mod N$, $\alpha_1\in M_{n-r,r}(\bZ)$ and $\alpha_2\in M_{r,r}(\bZ)\cap\GL(r,\bZ_p)$.

Take $s\geq\mr{max}\set{l,\varphi(N)}$. There exists 
$$\begin{blockarray}{ccc}
   & n-r&r\\
   \begin{block}{c(cc)}
   n-r&A&B\,\,\,\\
   r&C&D\,\,\,\\
   \end{block}
\end{blockarray}\in\GL(n,\bZ)$$
with $A\equiv\begin{pmatrix}A_1&0\\0&A_2\end{pmatrix}\mod p^s $, $A_1\in\GL(N_i,\bZ)$, $A_2\in\GL(n-r-N_i,\bZ)$, $C\equiv 0\mod p^s$, such that
$$\begin{pmatrix}A&B\end{pmatrix}\begin{pmatrix}\alpha_1\\\alpha_2\end{pmatrix}=0.$$
Define
\begin{align*}
   \ffi_{\wt{V}}:\Sym(n-r,\bQ)&\lra\Sym(n,\bQ)\\
   \beta'&\longmapsto \begin{pmatrix}I_{N_{d-1}}&0\\0&\ltrans{w}^{-1}\end{pmatrix}\begin{pmatrix}\ltrans{A}\\\ltrans{B}\end{pmatrix}\beta'\begin{pmatrix}A&B\end{pmatrix}\begin{pmatrix}I_{N_{d-1}}&0\\0&w^{-1}\end{pmatrix}.
\end{align*}
Then for $\gamma'_N\in\Sp(2n-2r,\bZ)$ and $a'_p\in T_{n-r}(\bZ_p)$, there exists $\gamma_N\in\Sp(2n,\bZ)$, $a_p\in T(\bZ_p)$ such that
$$\varepsilon^{\gamma'_{N},\,a'_p}_{\qexp,\wt{V}}\left(\beta',\,\fp_{\wt{V}}(f)\right)=\varepsilon^{\gamma_{N},\,a_p}_{\qexp}\left(\ffi_{\wt{V}}(\beta'),\,f\right).$$
To prove the proposition, it suffices to check that for $\beta'\in\Sym(n-r,\bQ)_{>0}$ and $f\in V^{SP,r,\flat}_{m,l}$, 
\begin{equation}\label{eq:qequiv}
\varepsilon^{(\gamma^{P'}_{p,i})^{\varphi(N)}\gamma'_{N},\,a'_p}_{\qexp,\wt{V}}\left(\beta',\,(U^P_{p,N_i})^{\varphi(N)}\fp_{\wt{V}}(f)\right)=\varepsilon^{(\gamma^P_{p,i})^{\varphi(N)}\gamma_{N},\,a_p}_{\qexp}\left(\ffi_{\wt{V}}(\beta'),\,(U^P_{p,N_i})^{\varphi(N)}f\right).
\end{equation}
We have
\begin{align*}
   &\text{LHS of }\eqref{eq:qequiv}\\
   =&\sum_{x\in M_{N_i,n-r-N_i}(\bZ/p^{\varphi(N)}\bZ)}\varepsilon^{\gamma'_{N},\,a'_p}_{\qexp,\wt{V}}\left(\begin{pmatrix}p^{\varphi(N)}I_{N_i}&0\\N\ltrans{x}&I_{n-r-N_i}\end{pmatrix}\beta'\begin{pmatrix}p^{\varphi(N)}I_{N_i}&Nx\\0&I_{n-r-N_i}\end{pmatrix},\,\fp_{\wt{V}}(f)\right)\\
	\numberthis\label{eq:lhs} =&\sum_{x\in M_{N_i,n-r-N_i}(\bZ/p^{\varphi(N)}\bZ)}\varepsilon^{\gamma_{N},\,a_p}_{\qexp}\left(\begin{pmatrix}\ltrans{A}\\\ltrans{B}\end{pmatrix}\begin{pmatrix}p^{\varphi(N)}I_{N_i}&0\\N\ltrans{x}&I_{n-r-N_i}\end{pmatrix}\beta'\begin{pmatrix}p^{\varphi(N)}I_{N_i}&Nx\\0&I_{n-r-N_i}\end{pmatrix}\begin{pmatrix}A&B\end{pmatrix},\,f\right).
\end{align*}
Set
\begin{align*}
   x_A&=A_1^{-1}xA_2\in M_{N_i,n-r-N_i}(\bZ/p^{\varphi(N)}),&y(x)=N^{-1}A^{-1}_1\begin{pmatrix}-I_{N_i}&Nx\end{pmatrix}B\in M_{N_i,r}(\bZ/p^{\varphi(N)}).
\end{align*}  
The map $x\mapsto x_A$ is a bijection from $ M_{N_i,n-r-N_i}(\bZ/p^{\varphi(N)}\bZ)$ to itself. One can check that, by the definition of $x_A$, $y(x)$,
\begin{equation}
\left[\begin{pmatrix}p^{\varphi(N)}I_{N_i}&Nx&0\\0&I_{n-r-N_i}&0\\0&0&I_r\end{pmatrix}^{-1}\begin{pmatrix}A&B\\C&D\end{pmatrix}\begin{pmatrix}p^{\varphi(N)}I_{N_i}&Nx_A&Ny(x)\\0&I_{n-r-N_i}&0\\0&0&I_r\end{pmatrix}\right]^{-1}\begin{pmatrix}A&B\\C&D\end{pmatrix}\in\Gamma(N)\cap\Gamma_{SP}(p^s),
\end{equation}
where $x,x_A,y(x)$ can be taken to be any lift to $\bZ$. Then
\begin{align*}
   &\begin{psm}p^{\varphi(N)}I_{N_i}&Nx\\0&I_{n-r-N_i}\end{psm}\begin{psm}A&B\end{psm}=\begin{psm}p^{\varphi(N)}I_{N_i}&Nx\\0&I_{n-r-N_i}\end{psm}\begin{psm}I_{n-r}&0\end{psm}\begin{psm}A&B\\C&D\end{psm}\\
   =&\begin{psm}I_{n-r}&0\end{psm}\begin{psm}p^{\varphi(N)}I_{N_i}&Nx&0\\0&I_{n-r-N_i}&0\\0&0&I_r\end{psm}\begin{psm}A&B\\C&D\end{psm}\\
   =&\begin{psm}A&B\end{psm}\begin{psm}p^{\varphi(N)}I_{N_i}&Nx_A&Ny(x)\\0&I_{n-r-N_i}&0\\0&0&I_r\end{psm}\begin{psm}p^{\varphi(N)}I_{N_i}&Nx_A&Ny(x)\\0&I_{n-r-N_i}&0\\0&0&I_r\end{psm}^{-1}\begin{psm}A&B\\C&D\end{psm}^{-1}\begin{psm}p^{\varphi(N)}I_{N_i}&Nx&0\\0&I_{n-r-N_i}&0\\0&0&I_r\end{psm}\begin{psm}A&B\\C&D\end{psm}\\
    =&\begin{psm}A&B\end{psm}\begin{psm}p^{\varphi(N)}I_{N_i}&Nx_A&y(x)\\0&I_{n-r-N_i}&0\\0&0&I_r\end{psm}\left[\begin{psm}p^{\varphi(N)}I_{N_i}&Nx&0\\0&I_{n-r-N_i}&0\\0&0&I_r\end{psm}^{-1}\begin{psm}A&B\\C&D\end{psm}\begin{psm}p^{\varphi(N)}I_{N_i}&Nx_A&Ny(x)\\0&I_{n-r-N_i}&0\\0&0&I_r\end{psm}\right]^{-1}\begin{psm}A&B\\C&D\end{psm}\\
    \in&\begin{psm}A&B\end{psm}\begin{psm}p^{\varphi(N)}I_{N_i}&Nx_A&Ny(x)\\0&I_{n-r-N_i}&0\\0&0&I_r\end{psm}\cdot\Gamma(N)\cap\Gamma_{SP}(p^s).
\end{align*}
Plugging into \eqref{eq:lhs}, we get
\begin{equation}\label{eq:lhs2}
\begin{aligned}
    &\text{LHS of }\eqref{eq:qequiv}\\
	 =&\sum_{x\in M_{N_i,n-r-N_i}(\bZ/p^{\varphi(N)}\bZ)}\varepsilon^{\gamma_{N},\,a_p}_{\qexp}\left(\begin{psm}p^{\varphi(N)}I_{N_i}&0&0\\N\ltrans{x}_A&I_{n-r-N_i}&0\\N\ltrans{y}(x)&0&I_r\end{psm}\begin{psm}\ltrans{A}\\\ltrans{B}\end{psm}\beta'\begin{psm}A&B\end{psm}\begin{psm}p^{\varphi(N)}I_{N_i}&Nx_A&Ny(x)\\0&I_{n-r-N_i}&0\\0&0&I_r\end{psm},\,f\right).
\end{aligned}
\end{equation}
Next we need to use the condition $f\in V^{SP,r,\flat}_{m,l}$ to show that its Fourier coefficient in $\varepsilon^{\gamma_{N},\,a_p}_{\qexp}(f)$ indexed by
\begin{equation}\label{eq:pxybeta}
\begin{pmatrix}p^{\varphi(N)}I_{N_i}&0&0\\N\ltrans{x}_A&I_{n-r-N_i}&0\\N\ltrans{y}&0&I_r\end{pmatrix}\begin{pmatrix}\ltrans{A}\\\ltrans{B}\end{pmatrix}\beta'\begin{pmatrix}A&B\end{pmatrix}\begin{pmatrix}p^{\varphi(N)}I_{N_i}&Nx_A&Ny\\0&I_{n-r-N_i}&0\\0&0&I_r\end{pmatrix}
\end{equation}
is nonzero only if $y=y(x)$ in $M_{N_i, r}(\bZ/p^{\varphi(N)})$. By Proposition \ref{prop:Vrb}, the coefficient indexed by \eqref{eq:pxybeta} is nonzero only if the radical of \eqref{eq:pxybeta} does not contain a primitive vector inside $\mb Z\cdot x^*_1 + \cdots +\mb Z\cdot x^*_{N_{d-1}} + p\mb Z\cdot x^*_{N_{d-1}+1} +\cdots+ p\mb Z\cdot x^*_{n}$. The radical tensored with $\bQ$ is spanned by the columns of 
\begin{align*}
   \begin{pmatrix}p^{\varphi(N)}I_{N_i}&Nx_A&Ny\\0&I_{n-r-N_i}&0\\0&0&I_r\end{pmatrix}^{-1}\begin{pmatrix}-A^{-1}B\\I_r\end{pmatrix}=\begin{pmatrix}-\begin{pmatrix}p^{-\varphi(N)}Ny\\0\end{pmatrix}-\begin{pmatrix}p^{-\varphi(N)}I_{N_i}&-p^{-\varphi(N)}Nx_A\\0&I_{n-r-N_i}\end{pmatrix}A^{-1}B\\I_r\end{pmatrix},
\end{align*}
which contains no primitive vector in $\mb Z\cdot x^*_1 + \cdots +\mb Z\cdot x^*_{N_{d-1}} + p\mb Z\cdot x^*_{N_{d-1}+1} +\cdots+ p\mb Z\cdot x^*_{n}$ only if 
\begin{equation*}
   Ny+\begin{pmatrix}I_{N_i}&-Nx_A\end{pmatrix}A^{-1}B\equiv 0\mod p^{\varphi(N)},
\end{equation*}
and this equation is satisfied exactly when $y=y(x)$. Therefore, the coefficient indexed by \eqref{eq:pxybeta} is nonzero only if $y=y(x)$ in $M_{N_i, r}(\bZ/p^{\varphi(N)})$, and from \eqref{eq:lhs2} we get
\begin{align*}
   &\text{LHS of \eqref{eq:qequiv}}\\
	 =&\sum_{\substack{x\in M_{N_i,n-r-N_i}(\bZ/p^{\varphi(N)}\bZ)\\y\in M_{N_i,r}(\bZ/p^{\varphi(N)}\bZ)}}\varepsilon^{\gamma_{N},\,a_p}_{\qexp}\left(\begin{psm}p^{\varphi(N)}I_{N_i}&0&0\\N\ltrans{x}_A&I_{n-r-N_i}&0\\N\ltrans{y}&0&I_r\end{psm}\begin{psm}\ltrans{A}\\\ltrans{B}\end{psm}\beta'\begin{psm}A&B\end{psm}\begin{psm}p^{\varphi(N)}I_{N_i}&Nx_A&Ny\\0&I_{n-r-N_i}&0\\0&0&I_r\end{psm},\,f\right)\\
	 =&\sum_{x\in M_{N_i,n-N_i}(\bZ/p^{\varphi(N)}\bZ)}\varepsilon^{\gamma_{N},\,a_p}_{\qexp}\left(\begin{psm}p^{\varphi(N)}I_{N_i}&0\\N\ltrans{x}&I_{n-N_i}\end{psm}\ffi_{\wt{V}}(\beta')\begin{psm}p^{\varphi(N)}I_{N_i}&Nx\\0&I_{n-N_i}\end{psm},f\right)\\
	 =&\text{RHS of \eqref{eq:qequiv}}.
\end{align*}
\end{proof}

\begin{prop}\label{prop:key}
Let $s \geq m,l$. Then $\left(U^P_{p,N_{d-1}}\right)^{2s} V^{SP,r}_{m,l} \subset V^{SP,r,\flat}_{m,l} $.
\end{prop}
\begin{proof}
By Proposition \ref{prop:Vrb}, what we need to show is that for all $\gamma_N\in\Sp(2n,\bZ)$, $a_p\in T(\bZ_p)$, $f\in V^{SP,r}_{m,l}$ and $\beta\in N^{-1}S^2(X_n)_{\geq 0}$ whose radical is of rank $r$ and contains a primitive vector $v_\beta$ inside $\mb Z\cdot x^*_1 + \cdots +\mb Z\cdot x^*_{N_{d-1}} + p\mb Z\cdot x^*_{N_{d-1}+1} +\cdots+ p\mb Z\cdot x^*_{n}$,
\begin{equation*}
   \varepsilon^{\gamma_N,a_p}_{\qexp}\left(\beta,\,\left(U^P_{p,N_{d-1}}\right)^{2s}f\right)=0.
\end{equation*} 
One can easily check that for all $x\in M_{N_{d-1},n_d}(\bZ)$,
\begin{equation*}
    \bQ\cdot\begin{pmatrix}p^sI_{N_{d-1}}&Nx\\0&I_r\end{pmatrix}^{-1}v_\beta\cap X^*_n\subset \bZ\cdot x^*_1+\cdots+\bZ\cdot x^*_{N_{d-1}}+p^s\bZ\cdot x^*_{N_{d-1}+1}+\cdots +p^s\bZ\cdot x^*_n,
\end{equation*}
{\it i.e.} the radical of $\begin{pmatrix}p^sI_{N_{d-1}}&0\\N\ltrans{x}&I_r\end{pmatrix}\beta\begin{pmatrix}p^sI_{N_{d-1}}&Nx\\0&I_r\end{pmatrix}$ contains a primitive vector inside $\bZ\cdot x^*_1+\cdots+\bZ\cdot x^*_{N_{d-1}}+p^s\bZ\cdot x^*_{N_{d-1}+1}+\cdots +p^s\bZ\cdot x^*_n$. Since by Proposition \ref{prop:UpF},
\begin{equation*}
  \varepsilon^{\gamma_N,a_p}_{\qexp}\left(\beta,\,\left(U^P_{p,N_{d-1}}\right)^{2s}f\right)=\sum_{x\in M_{N_{d-1},n_d}(\bZ)}\varepsilon^{\gamma^{-s}_{p,n-1}\gamma_N,a_p}_{\qexp}\left(\begin{pmatrix}p^sI_{N_{d-1}}&0\\N\ltrans{x}&I_r\end{pmatrix}\beta\begin{pmatrix}p^sI_{N_{d-1}}&Nx\\0&I_r\end{pmatrix},\,\left(U^P_{p,N_{d-1}}\right)^{s}f\right),
\end{equation*}
we reduce to showing 
\begin{equation*}
   \varepsilon^{\gamma_N,a_p}_{\qexp}\left(\beta,\,\left(U^P_{p,N_{d-1}}\right)^{s}f\right)=0,
\end{equation*}
for $\beta$ whose radical contains a primitive vector $v_\beta$ inside $\mb Z\cdot x^*_1 + \cdots +\mb Z\cdot x^*_{N_{d-1}} + p^s\mb Z\cdot x^*_{N_{d-1}+1} +\cdots+ p^s\mb Z\cdot x^*_{n}$.

Write $v_\beta=\ltrans{(}v_{\beta,1},\dots,v_{\beta,n})$. Then $p^s\mid v_{\beta,i}$, $N_{d-1}+1\leq i\leq n$, and there exists $1\leq j\leq N_{d-1}$ such that $v_{\beta,j}$ is not divisible by $p$. Put $w_\beta=\ltrans{(}\underbrace{0,\dots,0}_{j-1},v_{\beta,n},0,\dots,0,-v_{\beta,j})\in\bZ^n$. Then $I_n-N\eta\cdot v_\beta\ltrans{w_\beta}$ belongs to $\GL(n,\bZ)$ for all integer $\eta$. Moreover,
\begin{equation*}
    I_n-N\eta \cdot v_{\beta}\ltrans{w}_\beta\equiv I_n+\begin{pmatrix}0&\cdots&0&N\eta v_{\beta,j}\cdot v_{\beta,1}\\&\ddots&&\vdots\\0&\cdots&0&N\eta v_{\beta,j}\cdot v_{\beta,j}\\&\ddots&&\vdots\\0&\cdots&0&N\eta v_{\beta,j}\cdot v_{\beta,N_{d-1}}\\0&\cdots&0&0\end{pmatrix} \mod Np^s.
\end{equation*} 
Let $x_\beta=v_{\beta,j}\cdot\ltrans{(}v_{\beta,1},\cdots,v_{\beta,N_{d-1}})\in \bZ^{N_{d-1}}$. Then
\begin{equation*}
   \begin{pmatrix}I_{n-1}&N\eta x_\beta\\0&I_r\end{pmatrix}^{-1}\left(I_n-\eta\cdot v_\beta\ltrans{w}_\beta\right)\equiv I_n\mod Np^s
\end{equation*}
and
\begin{equation}\label{eq:bl1}
   \begin{pmatrix}p^sI_{n-1}&Nx+N\eta x_\beta\\0&I_r\end{pmatrix}^{-1}\left(I_n-\eta\cdot v_\beta\ltrans{w}_\beta\right)\begin{pmatrix}p^sI_{n-1}&Nx\\0&I_r\end{pmatrix}\in\mr{Im}\left(\Gamma\cap\Gamma_{SP}(p^s)\ra\GL(n,\bZ)\right).
\end{equation}
By definition the vector $x_\beta\in\bZ^{N_{d-1}}$ is not divisible by $p$. Thus we can pick $C\subset M_{N_{d-1},n_d}(\bZ/p^s)$ such that
$$ M_{N_{d-1},n_d}(\bZ/p^s)=(\bZ/p^s)\cdot x_\beta\oplus C.$$
We have
\begin{align*}
   &\varepsilon^{\gamma_N,a_p}_{\qexp}\left(\beta,\,\left(U^P_{p,N_{d-1}}\right)^s f\right)\\
   =&\sum_{x\in M_{N_{d-1},n_d}(\bZ/p^s)}\varepsilon^{(\gamma^P_{p,d-1})^{-s}\gamma_N,a_p}_{\qexp}\left(\begin{pmatrix}p^sI_{N_{d-1}}&0\\N\ltrans{x}&I_r\end{pmatrix}\beta\begin{pmatrix}p^sI_{N_{d-1}}&Nx\\0&I_r\end{pmatrix},\,f\right)\\
   =&\sum_{x\in C}\,\sum_{\eta\in\bZ/p^s\bZ}\varepsilon^{(\gamma^P_{p,d-1})^{-s}\gamma_N,a_p}_{\qexp}\left(\begin{pmatrix}p^sI_{N_{d-1}}&0\\N\ltrans{x}+N\eta \ltrans{x}_\beta&I_r\end{pmatrix}\beta\begin{pmatrix}p^sI_{N_{d-1}}&Nx+N\eta x_\beta\\0&I_r\end{pmatrix},\,f\right).
\end{align*}
Applying \eqref{eq:bl1}, we get
\begin{align*}
   &\varepsilon^{\gamma_N,a_p}_{\qexp}\left(\beta,\,\left(U^P_{p,N_{d-1}}\right)^s f\right)\\
   =&\sum_{x\in C}\,\sum_{\eta\in\bZ/p^s\bZ}\varepsilon^{(\gamma^P_{p,d-1})^{-s}\gamma_N,1}_{\qexp}\left(\begin{pmatrix}p^s I_{N_{d-1}}&0\\N\ltrans{x}&I_r\end{pmatrix}\left(I_n-\eta\cdot w_\beta\ltrans{v}_\beta\right)\beta\left(I_n-\eta\cdot v_\beta\ltrans{w}_\beta\right)\begin{pmatrix}p^sI_{N_{d-1}}&Nx\\0&I_r\end{pmatrix},\,f\right).
\end{align*}
Since $v_\beta$ belongs to the radical of $\beta$, we know that $\left(I_n-\eta\cdot w_\beta\ltrans{v}_\beta\right)\beta\left(I_n-\eta\cdot v_\beta\ltrans{w}_\beta\right)=\beta$, and
\begin{align*}
    &\varepsilon^{\gamma_N,a_p}_{\qexp}\left(\beta,\,\left(U^P_{p,N_{d-1}}\right)^sf\right)\\
   =&\sum_{x\in C}\,\sum_{\eta\in\bZ/p^s\bZ}\varepsilon^{(\gamma^P_{p,d-1})^{-s}\gamma_N,a_p}_{\qexp}\left(\begin{pmatrix}p^s I_{N_{d-1}}&0\\N\ltrans{x}&I_r\end{pmatrix}\beta\begin{pmatrix}p^sI_{N_{d-1}}&Nx\\0&I_r\end{pmatrix},\,f\right)\\
   =&\sum_{x\in C}p^s\cdot\varepsilon^{(\gamma^P_{p,d-1})^{-s}\gamma_N,a_p}_{\qexp}\left(\begin{pmatrix}p^s I_{N_{d-1}}&0\\N\ltrans{x}&I_r\end{pmatrix}\beta\begin{pmatrix}p^s I_{N_{d-1}}&Nx\\0&I_r\end{pmatrix},\,f\right)\\
   =&\,0.
\end{align*}
\end{proof}

\subsection{\texorpdfstring{Hida families of $p$-adic Siegel modular forms vanishing along strata with cusp labels of rank $> r$}{Hida families of p-adic Siegel modular forms vanishing along strata with cusp labels of rank  >r}}\label{sec:HF}
Set $U^P_p=\prod\limits_{i=1}^d U^P_{p,N_i}$. We first show the existence of an ordinary projector on $\pV^{SP,r}$ by applying induction on $r$ and using $\pV^{SP,r,\flat}$ plus Proposition \ref{prop:key}. 
\begin{prop}\label{prop:bconv}
For each $f\in \pV^{SP,r}$, the limit $\lim\limits_{j\ra\infty}\left(U^P_p\right)^{j!}f$ exists in $\pV^{SP,r}$.
\end{prop}
\begin{proof}
We remark that for any endomorphism of finitely generated $\cO_F/p^m$-modules, its $j!$-th power stabilizes when $j$ is large enough. 

Given $f$ inside an $\cO_F/p^m$-module with an action by $U^P_p$, we define the following finiteness property for $f$.
\begin{equation}
   \parbox{\dimexpr\linewidth-4.5em}{%
    \strut
    The submodule generated by $\left(U^P_p\right)^{n\varphi(N)}f$, $n\geq 0$, is finitely generated over $\cO_F/p^m$.
    \strut
  }\tag{F}
\end{equation}
It suffices to show that (F) holds for all elements in $V^{SP,r}_{m,l}$. For $r=0$ this is known thanks to \cite{HPEL,PilHida}. Now assume (F) holds for $V^{SP,r-1}_{m,l}$. Due to Proposition \ref{prop:key}, we only need to show that (F) for all $f\in V^{SP,r,\flat}_{m,l}$. Take $f\in V^{SP,r,\flat}_{m,l}$, since $(F)$ holds for the quotient in \eqref{eq:exactf}, there exists $a_0,a_1,\dots,a_{j_1}\in\cO_F$ such that
\begin{equation}\label{eq:form}
  g=\left(U^P_p\right)^{(j_1+1)\varphi(N)}f-\sum\limits_{i=0}^{j_1} a_i\left(U^P_p\right)^{i\varphi(N)}f\in V^{SP,r-1,\flat}_{m,l}.
\end{equation}
Then we apply (F) for \eqref{eq:form}. There exists $b_0,b_1,\dots,b_{j_2}\in\cO_F$ such that
\begin{align*}
  &\left(U^P_p\right)^{(j_2+1)\varphi(N)}\left(\left(U^P_p\right)^{(j_1+1)\varphi(N)} f-\sum\limits_{i=0}^{j_1} a_i\left(U^P_p\right)^{i\varphi(N)} f\right)\\
  =&\sum\limits_{s=0}^{j_2}b_s\left(U^P_p\right)^{s\varphi(N)}\left(\left(U^P_p\right)^{(j_1+1)\varphi(N)} f-\sum\limits_{i=0}^{j_1} a_i\left(U^P_p\right)^{i\varphi(N)}f\right).
\end{align*}
Therefore $\left(U^P_p\right)^{(j_1+j_2+2)\varphi(N)}f$ belongs to the $\cO_F/p^m$-span of $f,\left(U^P_p\right)^{\varphi(N)}f,\dots,\left(U^P_p\right)^{(j_1+j_2+1)\varphi(N)}f$, and (F) holds for $f$.
\end{proof}

The above proposition shows that $\lim\limits_{j\ra\infty}\left(U^P_p\right)^{j!}$ can be well defined on $\pV^{SP,r}$. Define the $P$-ordinary projector on $\pV^{SP,r}$ as
\begin{equation*}
  e_P=\lim\limits_{j\ra\infty}\left(U^P_p\right)^{j!}.
\end{equation*}
It is an idempotent projecting the space into the subspace spanned by generalized $\bU^{P}_p$-eigenvectors with eigenvalues being $p$-adic units. Similarly a $P_{n-r}$-ordinary projector can be defined for the quotient terms in the exact sequences in Proposition \ref{prop:exact}.

Set
\begin{align*}
   \pV^{r}_{\Pord}=e_P\pV^{SP,r}=e_P\pV^{SP,r,\flat},&&\pV^{0}_{V,P_{n-r}\text{-}\,\ord}=e_{P_{n-r}}\pV^{SP_{n-r},0}_V,\,\,V\in\mf C_{\mb V},\,\mathrm{rk}\,V=r\leq n_d.	 
\end{align*}
Applying $e_P$, $e_{P_{n-r}}$ to \eqref{eq:exact1}, we get
\begin{equation}\label{eq:exacto}
   0\lra\pV^{r-1}_{\Pord}\lra\pV^{r}_{\Pord}\lra \bigoplus_{\substack{V\in\fC_{\bV}/\Gamma\\\mr{rk}\,V=r}} \bZ_p\llbracket T_P(\bZ_p)\rrbracket \otimes_{\bZ_p\llbracket T_{P_{n-r}}(\bZ_p)\rrbracket } \pV^{0}_{V,P_{n-r}\text{-}\,\ord}\lra 0.
\end{equation}

Define $\pV^{r,*}_{\Pord}$ to be the Pontryagin dual of $\pV^{r}_{\Pord}$, {\it i.e.} $\mr{Hom}_{\mb Z_p}(\pV^{r}_{\Pord},\mb Q_p/\mb Z_p)$, and similarly define $\pV^{0,*}_{V,P_{n-r}\text{-}\,\ord}$. Then \eqref{eq:exacto} gives
\begin{align}\label{eq:exact2}
0\lra \bigoplus_{\substack{V \in\mf C_{\mb V}/\Gamma\\ \mathrm{rk}\,V=r}}\pV^{0,*}_{V,P_{n-r}\text{-}\,\ord}\otimes_{\mb Z_p\llbracket T_{P_{n-r}}(\mb Z_p)\rrbracket }\mb Z_p\llbracket T_P(\mb Z_p)\rrbracket \lra & \pV^{r,*}_{\Pord}\lra \pV^{r-1,*}_{\Pord} \lra 0.
\end{align}

Let $\Lambda_P=\mc O_F\llbracket T_P(\mb Z_p)^\circ\rrbracket $ (resp. $\Lambda_{P_{n-r}}=\mc O_F\llbracket T_{P_{n-r}}(\mb Z_p)^\circ\rrbracket $), where $T_P(\mb Z_p)^\circ$ is the maximal $p$-profinite subgroup of $T_P(\mb Z_p)$ (resp. $T_{P_{n-r}}(\bZ_p)$).
\begin{prop}\label{prop:bfree}
$\pV^{r,*}_{\Pord}$, $0\leq r\leq n_d$, is a free $\Lambda_P$-module of finite rank.
\end{prop}
\begin{proof}
We prove the proposition by induction. For $r=0$ the control theorem in \cite[Th\'eor\`eme 1.1 (7)]{PilHida} for $\pV^{0,*}_{\Pord}$ (resp. $\pV^{0,*}_{V,P_{n-r}\text{-}\,\ord}$) says that it is a free $\Lambda_P$-module (rep. $\Lambda_{P_{n-r}}$-module) of finite rank. Suppose that $\pV^{r-1,*}_{\Pord}$ is a free $\Lambda_P$-module of finite rank. Then the terms at the two ends of \eqref{eq:exact2} are free $\Lambda_P$-modules of finite rank. Since $\mathrm{Ext}^1_{\Lambda_P}(M,N)$ vanishes if $M$ is a free $\Lambda_P$-module, $\pV^{r,*}_{\Pord}$ is isomorphic, as a $\Lambda_P$-module, to the direct sum the terms at the two ends of \eqref{eq:exact2}.
\end{proof}
Now we have established {\it (i)} in Theorem \ref{thm:main}.

For $0\leq r\leq n_d$, the $\cO_F\llbracket T_P(\bZ_p)\rrbracket $-module of Hida families of $p$-adic Siegel modular forms ordinary for the parabolic $P$ vanishing along the strata associated with cusp labels of rank $> r$, is defined as 
\begin{align*}
  \mc M^{r}_{\Pord}:=\mr{Hom}_{\Lambda_P}\left(\pV^{r,*}_{\Pord},\Lambda_P\right).
\end{align*}
Similarly, define
\begin{align*}
\mc M^{0}_{P_{n-r}\text{-}\,\ord}:=\mr{Hom}_{\Lambda_{P_{n-r}}}\left(\pV^{0,*}_{V,P_{n-r}\text{-}\,\ord},\Lambda_{P_{n-r}}\right).
\end{align*}
Applying $\mr{Hom}_{\Lambda_P}\left(\cdot,\Lambda_P\right)$ to \eqref{eq:exact2} gives {\it (iii)} of Theorem \ref{thm:main}.

Let $\utau^P\in\Hom_{\cont}\left(T_P(\bZ_p),\ol{\bQ}^\times_p\right)$ be an arithmetic dominant weight. Attached to it is a prime ideal $\cP_{\utau^P}$ of $\cO_F\llbracket T_P(\bZ_p)\rrbracket $. Then unfolding the definitions, one gets the following isomorphisms, 
\begin{equation}\label{eq:bspec}
   \left.\cM^{r}_{\Pord}\otimes\cO_F\llbracket T_P(\bZ_p)\rrbracket \middle/\cP_{\utau^P}\right.\stackrel{\sim}{\lra} \Hom\left((\pV^{r}_{\Pord}[\utau^P])^*,\cO_F\right)\stackrel{\sim}{\lra}\varprojlim_m\varinjlim_l e_PV^{SP,r}_{m,l}[\utau^P]
\end{equation}
equivariant under the action of the unramfied Hecke algebra away from $Np$ and the $\bU^P_p$-operators. Combining \eqref{eq:bspec} with the embedding \eqref{eq:embed2} proves {\it (ii)} in Theorem \ref{thm:main}.


For applications in \S\ref{sec:pthm}, define $\mb{T}^{r,N}_{\Pord}$ as the $\mc{O}_F\llbracket T_P(\mb{Z}_p)\rrbracket $-algebra generated by all the unramified Hecke operators away from $Np\infty$ and the $\mb{U}_p$-operators $U^P_{p,1}, U^P_{p,2},\dots,U^P_{p,n}$ acting on $\cM^r_{\Pord}$. The algebra $\mb{T}^{r,N}_{\Pord}$ is finite and torsion free over $\mc{O}_F\llbracket T_P(\mb{Z}_p)\rrbracket $.

\section{\texorpdfstring{$p$-adic $L$-functions}{p-adic L-functions}}
In this section, for a given geometrically irreducible component of $\mr{Spec}\left(\mb{T}^{0,N}_{\Pord}\otimes F\right)$, we construct the $(d+1)$-variable $p$-adic standard $L$-function and its $d$-variable improvement as called in \cite{SSS} (missing the cyclotomic variable). The construction uses the doubling method formula as the integral representation for the standard $L$-function. The $d$-variable improvement will be used to employ the Greenberg--Stevens method to prove Theorem \ref{thm:mainzero} on the derivatives of cyclotomic $p$-adic $L$-functions at the so-called semi-stable trivial zeros. The Hida theory for non-cuspidal Siegel modular forms developed in the previous section will be used for the construction of the $d$-variable improved $p$-adic $L$-function.

Before starting the construction, we briefly mention several works on constructing $p$-adic $L$-functions using the doubling method. It is B\"{o}cherer and Schmidt \cite{BS} who first carried out such a construction in the special case where $\pi$ is fixed and is $\GL(n)$-ordinary with $\pi_\infty$ isomorphic to a scalar weight holomorphic discrete series. Later, the case where $\pi$ varies in a cuspidal Hida family which is ordinary for the Borel subgroup is treated in \cite{LiuSLF} for symplectic groups and in \cite{EisWan,EHLS} for unitary groups. Here we look at the more general case of $P$-ordinary Hida families for a general parabolic $P$. Moreover, we also construct its improvement as an important input for applying the Greenberg--Stevens method. We also compute the archimdean zeta integrals by using \cite{LiuAZI} and the archimedean functional equations. For groups of general rank, other than the doubling method, there are works on the construction of $p$-adic $L$-functions for $\GL(n)\times\GL(n-1)$ \cite{Januszewski1,Januszewski2} and for $\GL(2n)$ \cite{DJR}, but in their paper, the local zeta integrals at $p$ and $\infty$ have not been fully computed to verify completely the conjectured interpolation properties of the constructed $p$-adic $L$-functions.

\subsection{\texorpdfstring{Generalities on standard $L$-functions for symplectic groups}{Generalities on standard L-functions for symplectic groups}}
Let $\pi\subset\mc{A}_0(G(\mb{Q})\backslash G(\mb{A}))$ be an irreducible cuspidal automorphic representation of $G(\mb{A})$ and $\xi:\mb{Q}^\times\backslash\mb{A}^\times\ra\mb{C}^\times$ be a finite order Dirichlet character. Take $S$ to be a finite set of places of $\mb{Q}$ containing the archimedean place and all the finite places where $\pi_v$ or $\xi_v$ is ramified. 

For $v\notin S$, there exist unramified characters $\theta_i:\mb{Q}_v^\times\ra\mb{C}^\times$, $1\leq i\leq n$, such that $\pi_v$ is isomorphic to the normalized induction $\mr{Ind}_{B_G(\mb{Q}_v)}^{G(\mb{Q}_v)}(\theta_1,\dots,\theta_n)$ as $G(\mb{Q}_v)$-representations. Put $\alpha_{v,i}=\theta_i(q_v)$ where $q_v$ is the cardinality of the residue field of $\mb{Q}_v$. Then $\alpha^{\pm 1}_{v,1},\dots,\alpha^{\pm 1}_{v,n}$ are the Satake parameters of $\pi_v$, and the unramified local $L$-factor (for the standard representation $^LG^\circ=\mr{SO}(2n+1,\mb{C})\ra\mr{GL}(2n+1,\mb{C})$) is defined as
\begin{equation*}
   L_v(s,\pi_v\times \xi_v)=(1-\xi_v(q_v)q_v^{-s})^{-1}\prod_{i=1}^{n}(1-\xi_v(q_v)\alpha_{v,i}q_v^{-s})^{-1}(1-\xi_v(q_v)\alpha_{v,i}^{-1}q_v^{-s})^{-1}.
\end{equation*}
The analytic properties (meromorphic continuation, functional equation, location of possible poles) of the partial standard $L$-function
\begin{equation*}
  L^S(s,\pi\times \xi)=\prod_{v\notin S} L_v(s,\pi_v\times \xi_v)
\end{equation*}
are established in \cite{GPSR,KRPoles}.

Assuming $\pi_\infty\cong\mc{D}_{\ut}$, the holomorphic discrete series of weight $\ut=(t_1,\dots,t_n)$ (so $t_1\geq\cdots\geq t_n\geq n+1$), the critical points of $L^S(s,\pi \times \xi)$ are integers $s_0$ such that
\begin{equation*}
   1\leq s_0\leq t_n-n,\,(-1)^{s_0+n}=\xi(-1),\text{ or } n+1-t_n\leq s_0\leq 0, \,(-1)^{s_0+n+1}=\xi(-1).
\end{equation*}
The algebracity of these critical $L$-values divided by certain automorphic periods (expressed in terms of Petersson inner product) is obtained in \cite{Ha81,Sh00,BS}.  

\subsection{The doubling method for symplectic groups}\label{sec:doubling}
One standard way to study the standard $L$-function $L^S(s,\pi\times\xi)$ and its critical values is to apply the doubling method developed by Piatetski-Shapiro--Rallis \cite{PSR}, Garrett \cite{GaPull} and Shimura \cite{ShEuler}. 

For the convenience of the reader, we briefly recall the setting for the doubling method used in \cite{LiuSLF}. Let $\mb{V}'$ be another copy of $\mb{V}$ with standard basis $e'_1,\dots,e'_n,f'_1,\dots,f'_n$. Put $\mb{W}=\mb{V}\oplus\mb{V'}$, for which we fix the basis $e_1,\dots, e_n,e'_1,\dots,e'_n,f_1,\dots,f_n,f'_1,\dots,f'_n$. Then $\mb{W}$ is endowed with a symplectic pairing induced from that of $\mb{V}$ and $\mb{V}'$. Let $H=\Sp(\bW)=\mr{Sp}(4n)$. There is the (holomorphic) embedding $\iota$ of $G\times G$ into $H$ given by
\begin{align*}
   \iota:G\times G&\lhra H\\
	   \left(\begin{array}{cc}a&b\\c&d\end{array}\right)\times \left(\begin{array}{cc}a'&b'\\c'&d'\end{array}\right)&\longmapsto \left(\begin{array}{cccc}a&0&b&0\\0&a'&0&b'\\c&0&d&0\\0&c'&0&d'\end{array}\right).
\end{align*}

The space $W=\sum_{i=1}^n\mb{Z}e_i+\mb{Z}e'_i$ is a maximal isotropic subspace of $\bW$. Its stabilizer $Q_H$ is the standard Siegel parabolic subgroup of $H$. Besides $W$, there is another maximal isotropic subspace relevant to us, which is 
$W^\diamondsuit=\left\{(v,\mvw(v)):v\in V\right\}$, where $\mvw:\bV\ra\bV$ is the involution given by the matrix $\begin{pmatrix}0&I_n\\I_n&0\end{pmatrix}$ with respect to our fixed basis. Note that $\mvw$ dose not preserve the symplectic pairing but has similitude $-1$. The space $W^\diamondsuit$ is spanned by $e_i+f'_i,f_i+e'_i$, $1\leq i\leq n$. The doubling Siegel parabolic $Q^\diamondsuit_H$ is defined to be the stabilizer of $W^\diamondsuit$. We have
\begin{equation*}
   Q^\diamondsuit_H=\mc{S} Q_H \mc{S}^{-1} \quad\text{ with }\quad\mc{S}=\begin{pmatrix}I_n&0&0&0\\0&I_n&0&0\\0&I_n&I_n&0\\I_n&0&0&I_n\end{pmatrix}.
\end{equation*} 

For an element $g\in G$, define $g^\mvw$ to be $\mvw g\mvw\in G$. This conjugation by $\mvw$ is called the $\mr{MVW}$ involution. The $\mr{MVW}$ involution of an irreducible smooth representation of $G(\mb{Q}_v)$ is isomorphic to its contragredient \cite[p. 91]{MVW}. For $\varphi\in\pi$ we define its MVW involution $\varphi^\mvw$ as $\varphi^\mvw(g)=\varphi(g^\mvw)$. Thanks to the multiplicity one theorem \cite{Arthur}, $\varphi^\mvw$ lies inside $\overline{\pi}\subset\mc{A}_0(G(\mb{Q})\backslash G(\mb{A}))$.

\begin{rem}\label{rem:notationMVW}
Our formulation here aligns with those in \cite{GaPull,Sh00} but differs from \cite{GPSR} in that the embedding used in \cite{GPSR} corresponds to the above defined $\iota$ composite with a conjugation by $\vartheta$ on the second copy of $G$. Hence in our later computation using formulas from \cite{GPSR}, the involution $\vartheta$ shows up a lot.
\end{rem}

Let $s$ be a complex variable. Denote by $\xi_s$ (resp. $\xi^\diamondsuit_s$) the character of $Q_H(\mb{A})$ (resp. $Q^\diamondsuit_H(\mb{A})$) sending $\begin{pmatrix}A&B\\0&\ltrans A^{-1}\end{pmatrix}$ (resp. $\mc{S}\begin{pmatrix}A&B\\0&\ltrans A^{-1}\end{pmatrix}\mc{S}^{-1}$) to $\xi(\mr{det A})|\mr{det} A|^s$. Let $I_{Q_H}(s,\xi)=\mr{Ind}_{Q_H(\mb{A})}^{H(\mb{A})}\xi_s$ (resp. $I_{Q^\diamondsuit_H}(s,\xi)=\mr{Ind}_{Q^\diamondsuit_H(\mb{A})}^{H(\mb{A})}\xi^\diamondsuit_s$) be the normalized induction consisting of smooth functions $f$ on $H(\mb{A})$ that satisfy $f(qh)=\xi_s(q)\delta^{1/2}_{Q_H}(q)f(h)$ (resp. $f(qh)=\xi^\diamondsuit_s(q)\delta^{1/2}_{Q^\diamondsuit_H}(q)f(h)$) for all $h\in H(\mb{A})$ and $q\in Q_H(\mb{A})$ (resp. $q\in Q^\diamondsuit_H(\mb{A})$). Recall that the modulus character $\delta_{Q_H}$ takes value $|\mr{det} A|^{\frac{2n+1}{2}}$ at $\begin{pmatrix}A&B\\0&\ltrans A^{-1}\end{pmatrix}$. The local degenerate principal series $I_{Q_H,v}(s,\xi)$, $I_{Q^\diamondsuit_H,v}(s,\xi)$ for all places of $\mb{Q}$ are defined similarly. There is the simple isomorphism
\begin{align*}
   I_{Q_H}(s,\xi)&\lra I_{Q^\diamondsuit_H}(s,\xi)\\
   f(s,\xi)&\longmapsto f^\diamondsuit(s,\xi)(h)=f(s,\xi)(\mc{S}^{-1}h).
\end{align*}

Given $f(s,\xi)\in I_{Q_H}(s,\xi)$, the associated Siegel Eisenstein series is defined as 
\begin{equation*}
   E(h,f(s,\xi))=\sum_{\gamma\in Q_H(\mb{Q})\backslash H(\mb{Q})} f(s,\xi)(\gamma h)=\sum_{\gamma\in Q^\diamondsuit_H(\mb{Q})\backslash H(\mb{Q})} f^\diamondsuit(s,\xi)(\gamma h).
\end{equation*}
This sum is absolutely convergent for $\mr{Re}(s)\gg 0$ and admits a meromorphic continuation. 

For a finite place $v$ we fix the Haar measure on $G(\mb{Q}_v)$ such that the maximal compact subgroup $G(\bZ_v)$ has volume $1$. For the archimedean place we fix for $G(\mb{R})$ the product measure of the one on the maximal compact subgroup $K_{G,\infty}=\left\{\begin{pmatrix}a&b\\-b&a\end{pmatrix}:a+bi\in U(n,\mb{R})\right\}$ which has total volume $1$ with the one on $G(\mb{R})/K_{G,\infty}=\mb{H}_n=\left\{z\in \Sym(n,\mb{C}):\mr{Im} z>0\right\}$ given by $\mr{det}(y)^{-n-1}\prod\limits_{1\leq i\leq j\leq n}dx_{ij}\,dy_{ij}$. The Haar measures on $G(\mb{A})$ is taken to be the product of the local ones. 

For a given irreducible cuspidal automorphic representation $\pi\subset\mc{A}_0(G(\mb{Q})\backslash G(\mb{A}))$ and its complex conjugation $\ol{\pi}\subset\mc{A}_0(G(\mb{Q})\backslash G(\mb{A}))$, which is isomorphic to the contragredient of $\pi$, we fix isomorphisms $\pi\cong\bigotimes'_v\pi_v$ and $\ol{\pi}\cong \bigotimes'_v\tilde{\pi}_v$ such that for factorizable $\varphi_1,\varphi_2\in\pi$ with images $\bigotimes_v\ol{\varphi}_{1,v}\in\bigotimes'_v\tilde{\pi}_v$ and $\bigotimes_v\varphi_{2,v}\in\bigotimes'_v\pi_v$, we have
\begin{equation*}
   \left<\ol{\varphi}_1,\varphi_2\right>=\prod_v\left<\ol{\varphi}_{1,v},\varphi_{2,v}\right>_v,
\end{equation*}
where the pairing on the left hand side is the bi-$\mb{C}$-linear Petersson inner product with respect to our fixed Haar measure on $G(\mb{A})$ and the pairing on the right hand side is the natural pairing between $\pi_v$ and its contragredient $\tilde{\pi}_v$. 

For a local section $f_v(s,\xi)\in I_{Q_H,v}(s,\xi)$, define
\begin{align*}
   T_{f_v(s,\xi)}:\pi&\lra\pi\\
	  \varphi & \longmapsto \left(T_{f_v(s,\xi)}\varphi\right)(g)=\int_{G(\mb{Q}_v)}f^\diamondsuit_v(s,\xi)(\iota(g'_v,1))\varphi(gg'_v)d_vg'_v.
\end{align*}
We need to be careful with the convergence issue here, especially for $v=p,\infty$. The doubling local zeta integral is defined as
\begin{align*}
   Z_v(f_v(s,\xi),\cdot,\cdot):\tilde{\pi}_v\times\pi_v&\lra\mb{C}\\
   (\tilde{v}_1,v_2)&\longmapsto Z_v(f_v(s,\xi),\tilde{v}_1,v_2)=\int_{G(\mb{Q}_v)}f^\diamondsuit_v(s,\xi)(\iota(g_v,1))\left<\tilde{\pi}_v(g_v)\tilde{v}_1,v_2\right>_vd_vg_v.\numberthis\label{eq:localzeta}
\end{align*}
For factorizable $\varphi_1,\varphi_2\in\pi$, we have
\begin{equation*}
  \left<T_{f_v(s,\xi)}\ol{\varphi}_1,\varphi_2\right>=\frac{Z_v(f_v(s,\xi),\ol{\varphi}_{1,v},\varphi_{2,v})}{\left<\ol{\varphi}_{1,v},\,\varphi_{2,v}\right>_v}\left<\ol{\varphi}_1,\varphi_2\right>.
\end{equation*}

Given $\varphi\in\pi$, we define the linear form
\begin{align*}
   \sL_{\ol{\varphi}}:\mc{A}(G(\mb{Q})\times G(\bQ)\backslash G(\mb{A})\times G(\bA))&\lra \mc{A}(G(\mb{Q})\backslash G(\mb{A}))\\
   F&\longmapsto \sL_{\ol{\varphi}}(F)(g)=\int_{G(\mb{Q})\backslash G(\mb{A})}F(g',g)\ol{\varphi}(g')\,dg'.
\end{align*}
The doubling method formula is a formula on
\begin{equation*}
   \sL_{\ol{\varphi}}\left(E(\cdot,f(s,\xi))|_{G\times G}\right),
\end{equation*}
involving the partial standard $L$-function of $\pi$ and some local zeta integrals.

For a finite place $v$ where $\xi$ is unramified, we denote by $f^{\mr{ur}}_v(s,\xi)$ the unique section in $I_{Q_H,v}(s,\xi)$ that is fixed by the maximal compact subgroup $H(\bZ_v)\subset H(\mb{Q}_v)$ and takes value $1$ at the identity. 
\begin{theo}[\cite{GPSR,GaPull,ShEuler}]\label{thm:doubling}
Suppose $f(s,\xi)=\bigotimes_{s\notin S}f^{\mr{ur}}_v(s,\xi)\otimes\bigotimes_{v\in S}f_v(s,\xi)$ is a section inside to $I_{Q_H}(s,\xi)$.  If $\varphi\in\pi^{K^S_G}$ with $K^S_G=\prod_{v\notin S}G(\bZ_v)$, then
\begin{equation}\label{eq:Edouble}
   \sL_{\ol{\varphi}}\left(E(\cdot,f(s,\xi))|_{G\times G}\right)=d^S(s,\xi)^{-1}\cdot L^S(s+\frac{1}{2},\pi\times\xi)\cdot\left(\prod_{v\in S}T_{f_v(s,\xi)}\ol{\varphi}\right)^\mvw.
\end{equation}
Equivalently for all factorizable $\varphi_1,\varphi_2\in\pi^{K^S_G}$,
\begin{equation*}
  \left<E(\cdot,f(s,\xi))|_{G\times G},\,\ol{\varphi}_1\otimes\varphi^{\mvw}_2\right>=d^S(s,\xi)^{-1}\cdot L^S(s+\frac{1}{2},\pi\times\xi)\cdot\prod_{v\in S}\frac{Z_v(f_v(s,\xi),\ol{\varphi}_{1,v},\varphi_{2,v})}{\left<\ol{\varphi}_{1,v},\,\varphi_{2,v}\right>_v}\left<\ol{\varphi}_1,\varphi_2\right>.
\end{equation*}
Here $d^S(s,\xi)=\prod_{v\notin S}d_v(s,\xi)$ with
\begin{equation*}
   d_v(s,\xi):=L_v(s+\frac{2n+1}{2},\xi)\prod_{j=1}^{n}L_v(2s+2n+1-2j,\xi^2).
\end{equation*}
\end{theo}
For later use we also define the normalized Siegel Eisenstein series
\begin{equation*}
   E^*(h,f(s,\xi))=d^S(s,\xi)E(h,f(s,\xi)).
\end{equation*}
Then the identity \eqref{eq:Edouble} from the above theorem becomes
\begin{equation}\label{eq:E*double}
   \sL_{\ol{\varphi}}\left(E^*(\cdot,f(s,\xi))|_{G\times G}\right)= L^S(s+\frac{1}{2},\pi\times\xi)\cdot\left(\prod_{v\in S}T_{f_v(s,\xi)}\ol{\varphi}\right)^\mvw.
\end{equation}

The identities provided by the doubling method reduce the study of the standard $L$-function $L^S(s,\pi\times\xi)$ to that of the Siegel Eisenstein series $E(\cdot,f(s,\xi))$, or more precisely its restriction to $G\times G$, and local zeta integrals at places $v\in S$. 

\subsection{\texorpdfstring{The modified Euler factor at $p$}{The modified Euler factor at p}}\label{sec:Ep}
Before starting the construction of $p$-adic $L$-functions, we first recall some basic theory of Jacquet modules and unfold the definition in \cite{CoaMot} in our case to write down explicitly the expected modified Euler factor at $p$ in the interpolation formula. We also define the modified Euler factor at $p$ for the improved $p$-adic $L$-function, and see that when restricting to the leftmost critical points with $\chi=\epsilon^P_d$, the difference of the two factors lies inside a finite extension of $\cO_F[[T_P(\bZ_p)^\circ]]$.

\subsubsection{\texorpdfstring{Jacquet modules and $\bU^P_p$-operators}{Jacquet modules and Up-operators}}\label{sec:JM}
Suppose $\pi_p$ is the component at the place $p$ of an irreducible automorphic representation $\pi$ generated by a $P$-ordinary Siegel modular form. Our discussion on Jacquet modules aims to: (1) show the uniqueness of the $P$-ordinary forms inside $\pi$, or more precisely that the space of $P$-ordinary Siegel modular forms projects into a one dimensional subspace inside $\pi_p$; (2) explain how to retrieve the information on $\pi_p$ from the eigenvalues of the $\bU^P_p$-operators. The uniqueness result will also play an important role in our later computation of the local zeta integral at $p$. 

Let $P_G$ (resp. $SP_G$, $U_{P_G}$) be the inverse image of $P$ (resp. $SP$, $U_{P}$) of the projection \eqref{eq:QG}. The Jacquet module of $\pi_p$ associated to the parabolic $P_G$ is defined as
\begin{equation*}
   {\mc{J}_{P_G}(\pi_p)=\left.V_{\pi_p}\middle/\left\{\pi_p(u)\cdot v-v:u\in U_{P_G}(\mb{Q}_p),\, v\in V_{\pi_p}\right\}\right.}.
\end{equation*}
It follows from Jacquet's lemma \cite[Theorem 4.1.2, Proposition 4.1.4]{CasselJ} that $\mc{J}_{P_G}(\pi_p)$ is naturally isomorphic to the following subspace of $V_{\pi_p}$, 
\begin{equation}\label{eq:Jm2}
   \bigcap_{\substack{\ua=(\underbrace{a_1,\dots,a_1}_{n_1},\dots,\underbrace{a_d,\dots,a_d}_{n_d})\\a_1\geq\cdots\geq a_d\geq 0}}\left\{\int_{U_{P_G}(\mb{Q}_p)}\pi_p(up^{\ua})\cdot v\,du:v\in V_{\pi_p}\right\},
\end{equation}
where $p^{\ua}=\mr{diag}(p^{a_1},\dots,p^{a_1},\dots,p^{a_d},\dots,p^{a_d},p^{-a_1},\dots,p^{-a_1},\dots,p^{-a_d},\dots,p^{-a_d})$. Denote by $M_P$ the Levi subgroup of $P$ and we identify it with the Levi subgroup of $P_G$ via \eqref{eq:QG}. Both $\mc{J}_{P_G}(\pi_p)$ and the space in \eqref{eq:Jm2} are equipped with a natural action of $M_P(\mb{Q}_p)$, and the isomorphism between them is $M_P(\mb{Q}_p)$-equivariant. 

Given irreducible smooth admissible representations $\sigma_i$ of $\mr{GL}(n_i,\mb{Q}_p)$, $1\leq i\leq d$, Frobenius reciprocity gives
\begin{equation}\label{eq:Frobreci}
   \mr{Hom}_{G(\mb{Q}_p)}\left(\pi_p,\mr{Ind}_{P_G(\mb{Q}_p)}^{G(\mb{Q}_p)}\sigma_1\boxtimes\sigma_2\boxtimes\cdots\boxtimes\sigma_d\right)\cong\mr{Hom}_{M_P(\mb{Q}_p)}\left(\mc{J}_{P_G}(\pi_p),(\sigma_1\boxtimes\sigma_2\boxtimes\cdots\boxtimes\sigma_d)\otimes \delta^{1/2}_{P_G}\right),
\end{equation}
where $\delta_{P_G}$ is the modulus character sending $\mr{diag}(b_1,b_2,\dots,b_d)\in M_P(\mb{Q}_p)$, $b_i\in\GL(n_i,\bQ_p)$, to $\prod_{i=1}^d|\mr{det}(b_i)|_p^{2n+1+n_i-2N_i}$.

Suppose that the $P$-ordinary Siegel modular form generating $\pi$ is of weight $\ut=\imath(\ut^P)$ with $t^P_1\geq\dots\geq t^P_d\geq n+1$, so $\pi_\infty\cong\mc{D}_{\imath(\ut^P)}$. Denote by $\pi_{\imath(\ut^P)}$ the subspace of $\pi$ consisting of forms whose projection to $\pi_\infty$ belongs to the lowest $K_\infty$-type in $\mc{D}_{\imath(\ut^P)}$. Let $X^\Sigma_{\Gamma\cap\Gamma_{SP}(p^l)}$ be a toroidal compactification of the Shimura variety of level ${\Gamma\cap\Gamma_{SP}(p^l)}$. There is the canonical embedding
\begin{equation}\label{eq:AdUp}
   \mr H^0\left(X^\Sigma_{\Gamma\cap\Gamma_{SP}(p^l)},\,\omega_{\imath(\ut^P)}\right)\lhra  M_{\imath(\ut^P)}\left(\bH_n,\Gamma\cap\Gamma_{SP}(p^l)\right)\stackrel{\varphi_G(\cdot,\mf{e}_{\mr{can}})}{\lhra}\mc{A}\left(G(\mb{Q})\backslash G(\mb{A})/\widehat{\Gamma}\cap\wh{\Gamma}_{SP}(p^l)\right)_{\imath(\ut^P)}
\end{equation}
from Siegel modular forms defined as global sections of the automorphic sheaf $\omega_{\ut}$ into automorphic forms on $G(\mb{A})$ of  $K_\infty$-type $\ut=\imath(\ut^P)$ (see, for example, \cite[(2.3.1)(2.4.1)]{LiuSLF} for precise definition of this embedding). Here for a congruence subgroup $\Gamma$, we denote by $\widehat{\Gamma}$ the corresponding compact subgroup of $G(\mb{A})$.

Under the embedding \eqref{eq:AdUp}, the $\mb{U}^P_p$-operator $U^P_{p,\ua}=\prod\limits_{i=1}^n\left( U^{P}_{p,i}\right)^{a_i}$, for $\ua=(a_1,a_2,\dots,a_n)$ with $a_1\geq a_2\geq \dots \geq a_n\geq 0$, on the left hand side corresponds to the following operator on the right hand side,
\begin{equation}\label{eq:AUp}
   U^P_{p,\ua}=p^{\left<\ut+2\rho_{G,c},\,\ua\right>}\int_{SP_G(\mb{Z}_p)}\pi_p(up^{\ua})\,du,
\end{equation}
where $\rho_{G,c}=(\frac{n-1}{2},\frac{n-3}{2},\dots,\frac{1-n}{2})$ is the half sum of positive compact roots of $G$.

We have assumed that $\pi$ contains a $P$-ordinary Siegel modular. It follows immediately from the definition of the $P$-ordinarity and \eqref{eq:Jm2}, \eqref{eq:AUp} that
\begin{equation*}
   \mc{J}_{P_G}(\pi_p)^{\mr{SL}(n_1,\mb{Z}_p)\times\dots\times\mr{SL}(n_d,\mb{Z}_p)}\neq\{0\},
\end{equation*}
which combined with the Frobenius reciprocity \eqref{eq:Frobreci} implies that there exists spherical representations $\sigma_i$ of $\mr{GL}(n_i,\mb{Q}_p)$, $1\leq i\leq d$, and continuous characters $\eta_1,\dots,\eta_d$ of $\mb{Q}_p^\times$ taking value $1$ at $p$, such that 
\begin{equation}\label{eq:embed}
   \pi_p\lhra \mr{Ind}_{P_G(\mb{Q}_p)}^{G(\mb{Q}_p)}(\sigma_1\otimes\eta_1\circ\mr{det})\boxtimes(\sigma_2\otimes\eta_2\circ\mr{det})\boxtimes\cdots\boxtimes(\sigma_d\otimes\eta_d\circ\mr{det}).
\end{equation}
In particular, $\pi_p$ embeds into a principal series. In general, $\pi_p$ being isomorphic a subquotient of a principal series is equivalent to $\pi$ containing a finite slope form. 
   
\begin{rem}
It does not make sense to say $P$-ordinarity for a purely local representation $\pi_p$ as the normalization in the definition of the $\mb{U}^P_p$-operators depends on the weight of the holomorphic discrete series at the archimedean space. However, being of finite slope is a purely local property.
\end{rem}

Next, we say more about the relation between the Satake parameters of the $\sigma_i's$ in \eqref{eq:embed} and the eigenvalues of the $\mb{U}^P_p$-operators. Let $\utheta=(\theta_1,\dots,\theta_n)$ be an $n$-tuple of  continuous characters of $\mb{Q}_p^\times$ (valued in $\mb{C}^\times$), viewed as a character of $T_G(\mb{Q}_p)$ via our fixed isomorphism of $\mb{G}^n_m$ with $T_G$, such that $\pi_p$ is isomorphic to a subquotient of the principal series $\mr{Ind}_{B_G(\mb{Q}_p)}^{G(\mb{Q}_p)} \utheta$. We consider the eigenvalues for the $\bU^P_p$-action on $\mr{Ind}_{B_G(\mb{Q}_p)}^{G(\mb{Q}_p)} \utheta$.

Denote by $W_G$ (resp. $W_{P_G}$) the Weyl group with respect to $T_G$ (resp. the subgroup of $W_G$ that maps $P_G$ to itself). Define $[W_G/W_{P_G}]$ to be the subset of $W_G$ consisting of representatives of smallest lengths of elements in $W_G/W_{P_G}$. An element $w\in W_G$ acts on $\utheta$ by sending it to $\utheta^w(t)=\utheta(w^{-1}tw)$, $t\in T_G$. Like $\utheta$, via our fixed isomorphism between $T_G$ and $\bG^n_m$, we can write $\utheta^w$ as an $n$-tuple of characters $(\theta^w_1,\dots,\theta^w_n)$.

It follows from \cite[Proposition 6.3.1, 6.3.3]{CasselJ} that the $M_P(\bQ_p)$-representation $\mc{J}_{P_G}\left(\mr{Ind}_{P_G(\mb{Q}_p)}^{G(\mb{Q}_p)} \utheta \right)$ has a filtration with graded pieces as
\begin{equation*}
    \btimes_{i=1}^d\, \mr{Ind}^{\mr{GL}(n_i,\mb{Q}_p)}_{B_{n_i}(\mb{Q}_p)} \left(\theta^w_{N_{i-1}+1},\dots,\theta^w_{N_{i}}\right)\cdot\delta_{B_{n_i}}^{-1/2}\left.\delta_{B_G}^{1/2}\right|_{B_{n_i}}, \quad w\in [W_G/W_{P_G}],
\end{equation*}
where $B_{n_i}$ is the standard Borel subgroup of $\mr{GL}(n_i)$ with modulus character $\delta_{B_{n_i}}$.

Thus, the dimension of the $\mr{SL}(n_1,\mb{Z}_p)\times\dots\times\mr{SL}(n_d,\mb{Z}_p)$-invariant space inside $\mc{J}_{P_G}(\pi_p)$ is at most $|W_G/W_{P_G}|=2^n\cdot|\mf{S}_n/(\mf{S}_{n_1}\times\cdots\times\mf{S}_{n_d})|$. Each $w\in [W_G/W_{P_G}]$ corresponds to an eigensystem of the $\bU^P_p$-operators, and the existence of a $P$-ordinary Siegel modular form in $\pi$ indicates that there exists $w\in W_G$ satisfying
\begin{align}
  \label{eq:TNewton}\sum_{j=1}^r v_p\left(\theta^w_{N_{i-1}+j}(p)\right)&\geq -r\left(t^P_i-N_{i-1}-\frac{r+1}{2}\right), &1\leq r\leq n_i,\,1\leq i\leq d, \\
 \sum_{j=1}^{n_i}v_p\left(\theta^w_{N_{i-1}+j}(p)\right)&=-n_i\left(t^P_i-\frac{N_{i-1}+N_{i}+1}{2}\right) &1\leq i\leq d.\label{eq:UNewton}
\end{align}
These conditions on the $p$-adic valuation of $\theta^w_{i}$, $1\leq i\leq n$, imply
\begin{align}
   &-(t^P_i-N_{i-1}+1)\leq v_p\big(\theta^w_{N_{i-1}+1}(p)\big),\dots,v_p\left(\theta^w_{N_{i}}(p)\right)\leq -(t^P_i-N_i), &1\leq i\leq d\label{eq:ures}.
\end{align} 
It is easily seen that given $\utheta$, there is at most one $w\in [W_G/W_{P_G}]$ to make \eqref{eq:ures} hold. By rearranging the $\theta^{\pm 1}_i$'s, we can assume that $w=1$ in \eqref{eq:TNewton}, \eqref{eq:UNewton}, \eqref{eq:ures}, and that $v_p(\theta_1(p))\leq\cdots\leq v_p(\theta_n(p))\leq 0$. 

The above discussion proves the following proposition.
\begin{prop}\label{prop:JL}
Suppose that $\pi$ is an irreducible automorphic representation of $G(\mb{A})$ containing a nonzero $P$-ordinary holomorphic Siegel modular form of weight $\imath(\ut^P)$, $t^P_d\geq n+1$, and $p$-nebentypus $\uep^P$. 
\begin{itemize}[leftmargin=*]
\item There exists unramified characters $\theta_1,\dots,\theta_n:\bQ^\times_p\ra\bC^\times$ satisfying 
\begin{align*}
   \utheta|_{\bZ^\times_p}&=(\underbrace{\epsilon^{P-1}_1,\dots,\epsilon^{P-1}_1}_{n_1},\dots,\underbrace{\epsilon^{P-1}_d,\dots,\epsilon^{P-1}_d}_{n_d}),\\
   \utheta(p)&=(\alpha_1,\alpha_2,\dots,\alpha_n), &\dots\leq -(t^P_i-N_{i-1}+1)\leq v_p(\alpha_{N_{i-1}+1})\leq\dots\leq v_p(\alpha_{N_i})\\
   &&\leq -(t^P_i-N_i)\leq-(t_{i+1}^P-N_i+1)\leq v_p(\alpha_{N_i+1})\leq\dots
\end{align*} 
such that $\pi_p\hra \mr{Ind}^{G(\bQ_p)}_{B_G(\bQ_p)}\utheta$.
\item Let $\mf{a}_i$ be the eigenvalue for the action of $ U^P_{p,i}$ on the $P$-ordinary form in $\pi$. Then $\mf{a}_{N_1},\dots,\mf{a}_{N_d}$ are $p$-adic units given by
\begin{equation}\label{eq:Ueig}
  \mf{a}_{N_i}=\prod_{j=1}^i p^{n_j\left(t^P_j-\frac{N_{j-1}+N_{j}+1}{2}\right)}\alpha_{N_{j-1}+1}\alpha_{N_{j-1}+2}\cdots\alpha_{N_j},
\end{equation}
More generally, for $N_j\leq i\leq N_{j+1}$, the eigenvalue $\mf{a}_i$ is an $p$-adic integer given as 
\begin{equation}\label{eq:Teig}
  \mf{a}_{i}=\mf{a}_{N_{j}} \cdot p^{(i-N_j)\left(t^P_{i+1}-\frac{N_j+i+1}{2}\right)}\sum_{\varrho\in\mf{S}_{n_{j+1}}/\mf{S}_{i-N_j}\times\mf{S}_{N_{j+1}-i}}\alpha_{N_{j}+\varrho(1)}\alpha_{N_{j}+\varrho(2)}\cdots\alpha_{N_{j}+\varrho(r)}.
\end{equation}
\item Inside $\mr{Ind}^{G(\bQ_p)}_{B_G(\bQ_p)}\utheta$, there is a unique generalized eigenvector (up to scalar) for the operator $U^P_p=\prod_{i=1}^d U^P_{p,N_i}$ with eigenvalue being a $p$-adic unit. In particular, under the projection $\pi\ra\pi_p$, the image of $P$-ordinary Siegel modular forms is one dimensional.
\end{itemize}
\end{prop}

\begin{rem}\label{rem:mono}
Let $\utheta$ be as in the above proposition. If $\epsilon^{P}_d=\mr{triv}$ and $\alpha_n=p^{-1}$, since the $P$-ordinary condition implies that $\pi_p\hra \mr{Ind}^{G(\bQ_p)}_{B_G(\bQ_p)}\utheta$, in the Weil--Deligne representation attached to $\pi_p$, there should be a nontrivial monodromy between the eigenspaces with Frobenius  eigenvalues $1$ and $\alpha_n=p^{-1}$.
\end{rem}

\subsubsection{\texorpdfstring{The modified Euler factor at $p$ for $p$-adic interpolation}{The modified Euler factor at p for p-adic interpolation}}
If we consider the Weil--Deligne representation attached to $\pi_p$, the eigenvalues of Frobenius are $1$, $\alpha^{\pm 1}_1,\dots,\alpha^{\pm 1}_n$. Meanwhile, for the $p$-adic representation associated to $\pi$ \cite{Arthur,BP,Shin,CH}, the Hodge--Tate weights are $0,\pm (t^P_1-1),\dots,\pm (t^P_1-n_1),\dots,\pm (t^P_d-(N_{d-1}+1)),\dots,\pm (t^P_d-n)$. Thus, \eqref{eq:TNewton} and \eqref{eq:UNewton} essentially say that the Newton polygon is above the Hodge polygon and the two polygons meet at the points with horizontal coordinates $0, N_1, N_2,\dots,N_d,2n+1-N_d,\dots,2n+1-N_2,2n+1-N_1,2n+1$.

Since the definition of the modified Euler factor in \cite{CoaMot}, formulated in terms of the Weil--Deligne representation, does not depend on the monodromy operator, our above description of the Weil--Deligne representation associated to $\pi_p$ is enough for us to unfold the definition in this case to obtain the explicit modified Euler factor at $p$ in terms of the $\bU^P_p$-eigenvalues.

From now on we fix a (tame) finite order Dirichlet character $\phi:\mb{Q}^\times\backslash\mb{A}^\times\ra\mb{C}^\times$ unramified away from $N\infty$. Suppose $\chi\in\mr{Hom}_{\mr{cont}}(\mb{Z}^\times_p,\ol{\mb{Q}}^\times_p)$ is of finite order. We also view it also as a $\mb{C}^\times$-valued character of $\mb{Q}^\times\backslash\mb{A}^\times$ which sends the uniformizer in $\bQ_v$ to $\chi(q_v)$ for finite places $v\neq p$. In the same way, we view the finite order characters $\epsilon^P_1,\dots,\epsilon^P_d$ as adelic characters. Let $\utheta$ be as in Proposition \ref{prop:JL}. Denote by $\sigma_i$ the unramified representation of $\GL(n_i,\bQ_p)$ such that $\sigma_i\otimes\epsilon^P_{i,p}\circ\det =\mr{Ind}^{\GL(n_i,\bQ_p)}_{B_{n_i}(\bQ_p)}\left(\theta_{N_{i-1}+1},\dots,\theta_{N_i}\right)$. Denote by $\wt{\sigma}_i$ the contragredient of $\sigma_i$.

The modified Euler factor at $p$ for $p$-adically interpolating the critical values of $L^S(s,\pi\times\phi\chi)$ to the left of the center is
\begin{equation}\label{eq:Ep}
   E_p(s,\pi\times\phi\chi)=\prod_{i=1}^d \gamma_p\left(s,\sigma_i\otimes\phi_p\chi_p\epsilon^{P-1}_{i,p}\right)^{-1}.
\end{equation}
Here we omit our fixed additive character $\be_p$ from the usual notation for gamma factors. One can also write the gamma factors in terms of the Satake parameters as
\begin{equation}\label{eq:Epgamma}
   \gamma_p\left(s,\sigma_i\otimes\phi_p\chi_p\epsilon^{P-1}_{i,p}\right)^{-1}=\begin{cases}\prod\limits_{j=N_{i-1}+1}^{N_i}\frac{1-\phi_p(p)^{-1}\alpha_j^{-1}p^{s-1}}{1-\phi_p(p)\alpha_jp^{-s}},&\text{if } \chi\epsilon^{P-1}_{i}\text{ is trivial},\\
   G(\chi^{-1}\epsilon^{P}_{i})^{n_i}\prod\limits_{j=N_{i-1}+1}^{N_i}\left(\phi_p(p)^{-1}\alpha^{-1}_jp^{s-1}\right)^{c(\chi\epsilon^{P-1}_i)}, &\text{otherwise},\end{cases}
\end{equation}
where $p^{c(\chi\epsilon^{P-1}_i)}$ is the conductor of $\chi\epsilon^{P-1}_i$, and the Gauss sum is defined as
\begin{equation*}
   G(\chi^{-1}\epsilon^{P}_{i})=\int_{p^{-m}\mb{Z}_p}\chi_p\epsilon^{P-1}_{i,p}(x)\be_p(-x)\,dx,\quad m\gg 0.
\end{equation*}

We also define the improved modified Euler factor at $p$ for the $d$-variable improved $p$-adic $L$-function. The improved $p$-adic $L$-function is supposed to interpolate the leftmost critical $L$-values with $\chi=\epsilon^{P}_d$. Define
\begin{align*}
   E^{\imp}_p(s,\pi\times\phi\epsilon^{P}_{d})&=\prod_{i=1}^{d-1}\gamma_p\left(1-s,\wt{\sigma}_i\otimes\phi^{-1}_p\epsilon^{P-1}_{d,p}\epsilon^{P}_{i,p}\right)\cdot L_p(s,\sigma_d\otimes\phi_p)
\end{align*}
It is easy to see that by \eqref{eq:Ueig} and \eqref{eq:Teig} both $E_p(s,\pi\times\phi\chi)$ and $E^{\imp}_p(s,\pi\times\phi\epsilon^{P}_{d})$ can be written in term of the $\bU^P_p$-eigenvalues of the $P$-ordinary Siegel modular form contained in $\pi$. 

We have
\begin{equation}\label{eq:Ecompare}
   E_p(n+1-t^P_d,\pi\times\phi\epsilon^{P}_d)=\mc{A}^P(\pi\times\phi \epsilon^{P}_{d})\cdot E^{\imp}_p(n+1-t^P_d,\pi\times\phi\epsilon^{P}_{d})
\end{equation}
with
\begin{align*}
   \mc{A}^P(\pi\times\phi\epsilon^{P}_{d})&=\prod_{j=N_{d-1}+1}^n\left(1-\phi_p(p)^{-1}\alpha^{-1}_jp^{n-t^P_d}\right)\\
    &=\,1+\mf{a}_n^{-1}\mf{a}_{N_{d-1}}\left(-\phi_p(p)^{-1}p^{\frac{n_d-1}{2}}\right)^{n_d}+\mf{a}_n^{-1}\sum_{r=1}^{n_d-1}\mf{a}_{N_{d-1}+r}\left(-\phi_p(p)^{-1}p^{\frac{n_d-1-r}{2}}\right)^{n_d-r}.
\end{align*}
Since all the $\mf{a}_i$'s are the $\mb{U}^P_p$-eigenvalues of the $P$-ordinary Siegel modular forms, when (the eigensystem of) $\pi$ varies in a $P$-ordinary Hida family, $\mc{A}^P(\pi,\phi)$ becomes a $d$-variable $p$-adic analytic function lying inside a finite extension of $\mc{O}_F\llbracket T_P(\mb{Z}_p)^\circ\rrbracket$. This explains why when restricting to the leftmost critical values with $\chi=\epsilon^{P}_d$, one expects the existence of the improved $d$-variable $p$-adic $L$-function with the improved modified Euler factor $E^{\imp}_p(\pi,\phi)$ at $p$ (improved in the sense of saving part of the numerator from $E_p(n+1-t^P_d,\pi\times\phi\epsilon^{P}_d)$).

\subsection{The choices of local sections for the Siegel Eisenstein series}
Our choices of local sections for Siegel Eisenstein series on $\Sp(4n)$,  and the formulae for local Fourier coefficients as well as the doubling local zeta integrals corresponding to those selected sections are summarized in the two tables in \S\ref{sec:table}. This section explains the strategy for the section selections. The computation of the zeta integrals at the place $p$ is done in \S\ref{sec:zetaintegralp}.

\subsubsection{Criteria for selecting sections}We first describe the context   and criteria for choosing the sections for Siegel Eisenstein series on $\Sp(4n)$. Recall that $\mr{Hom}_{\mr{cont}}\left(T_P(\mb{Z}_p),\ol{\mb{Q}}^\times_p\right)$ is the weight space for Hida families which are ordinary for the parabolic for $P$. An arithmetic point in it corresponds to a character $\utau^P=\utau^P_{\mr{alg}}\cdot\utau^P_{\mr{f}}\in \mr{Hom}_{\mr{cont}}\left(T_P(\mb{Z}_p),\ol{\mb{Q}}^\times_p\right)$, a product of the algebraic part $\utau^P_{\mr{alg}}=\ut^P=(t^P_1,\dots,t^P_d)$ and the finite order part $\utau^P_{\mr{f}}=\uep^P=(\epsilon^P_1,\dots,\epsilon^P_d)$. Similarly, the parameterization space for the cyclotomic variable is $\mr{Hom}_{\mr{cont}}\left(\mb{Z}^\times_p,\ol{\mb{Q}}^\times_p\right)$, and an arithmetic point in it is a character $\kappa=\kappa_{\mr{alg}}\cdot\kappa_{\mr{f}}$, with algebraic part $\kappa_{\mr{alg}}=k$ and finite order part $\kappa_{\mr{f}}=\chi$. We call an arithmetic point $\utau^P $ (resp. $(\kappa,\utau^P)$) admissible if $t^P_1\geq\dots\geq t^P_d\geq n+1$ and (resp. $t^P_1\geq\dots\geq t^P_d\geq k\geq n+1$).

Let $\sC_P$ be a geometrically irreducible component of $\mr{Spec}\left(\mb{T}^{0,N}_{\Pord}\otimes_{\mc{O}_F}F\right)$. The projection of $\sC_P$ to the weight space is one of its $|T_P(\bZ/p)|$ connected components. We say the parity of $\sC_P$ is compatible with $\phi$ if all the points $\utau^P$ in that connected component satisfy $\uptau^P_d(-1)=\phi(-1)$.

A point $x\in\sC_P(\ol{\mb{Q}}_p)$ is called arithmetic if its projection $\utau^P$ inside the weight space is arithmetic, and an arithmetic pair $(\kappa,x)$ (resp. an arithmetic point $x$) is called admissible if $(\kappa,\utau^P)$ (resp. $\utau^P$) is admissible. If the Hecke eigensystem parametrized by $x$ appears in an irreducible cuspidal automorphic representation $\pi_x\in\cA_0(G(\bQ)\backslash G(\bA))$ with $\pi_{x,\infty}\cong \cD_{\imath(\ut^P)}$, we call such an $x$ classical, and one can define the corresponding $L^{Np\infty}(s,\pi_{x}\times\phi\chi)$, $E_p(s,\pi_{x}\times\phi\chi)$, $E^{P\text{-imp}}_p(\pi_{x},\phi)$. Note that because of the lack of strong multiplicity one, $\pi_x$ may not be unique, but the partial $L$-functions and the modified Euler factors at $p$ do not depend on the choice of $\pi_x$.

The $(d+1)$-variable $p$-adic $L$-function is intended to interpolate the critical values
\begin{equation*}
   E_p(n+1-k,\pi_x\times\phi\chi)\cdot L^{Np\infty}(n+1-k,\pi_{x}\times\phi\chi)
\end{equation*}
divided by a Petersson inner product period for $(\kappa,x)$ admissible with $\kappa(-1)=\phi(-1)$ and $x$ classical (by our construction, if $(\kappa,x)$ is admissible but $x$ is not classical, one can see that the evaluation of our $p$-adic $L$-function there is $0$). Its $d$-variable improvement (assuming the parity of $\sC_P$ is compatible with $\phi$) is supposed to interpolate
\begin{equation*}
   E^{\imp}_p(n+1-t^P_d,\pi_{x}\times\phi\epsilon^P_d)\cdot L^{Np\infty}(n+1-t^P_{d},\pi_{x}\times\phi\epsilon^{P}_{d})
\end{equation*}
divided by a Petersson inner product period for classical $x$. 

From Theorem \ref{thm:doubling}, we see that in order to get the above $L$-values, we need to pick a Siegel Eisenstein series on $\Sp(4n)$ with nice properties for each admissible $(\kappa,\utau^P)\in\mr{Hom}_{\mr{cont}}(\mb{Z}^\times_p\times T_P(\mb{Z}_p,\ol{\mb{Q}}^\times_p))$ with $\kappa(-1)=\phi(-1)$ as well as for each admissible $\utau^P\in\mr{Hom}_{\mr{cont}}(T_P(\mb{Z}_p,\ol{\mb{Q}}^\times_p))$ with $\uptau^P_d(-1)=\phi(-1)$, so that we can deduce the desired congruences among the $L$-values from those of the Siegel Eisenstein series. Picking the Siegel Eisenstein series amounts to selecting sections in the degenerate principal series. 

More precisely, for $(\kappa,\utau^P)$ (resp. $\utau^P$) as above and each place $v$ of $\bQ$, we need to pick a section $f_{\kappa,\utau^P,v}$ (resp. $f_{\utau^P,v}$)  from $I_{Q_H,v}(\frac{2n+1}{2}-k,\phi\chi)$ (resp. $I_{Q_H,v}(\frac{2n+1}{2}-t^P_d,\phi\epsilon^{P}_d)$), such that 
\begin{itemize}
\item We have enough control of the local zeta integrals at places dividing $Np\infty$. In particular, we are able guarantee the non-vanishing of the archimedean zeta integrals and compute the zeta integrals at $p$.
\item The collection of the Eisenstein series $E(\cdot,f_{\kappa,\utau^P})|_{G\times G}$ (resp. $E(\cdot,f_{\utau^P})|_{G\times G}$) (after suitable normalizations)  assembles to a $p$-adic family.
\end{itemize}
The way we treat the second requirement is via looking at their Fourier coefficients and invoking the $q$-expansion principle. Also, the second requirement provides us a hint for making the choices of the sections at $p$ based on our selection of archimedean sections.

\subsubsection{The Fourier coefficients for Siegel Eisenstein series}
For $\bbeta\in \mr{Sym}(2n,\mb{Q})$, the $\bbeta$-th Fourier coefficient of $E\big(\cdot,f(s,\xi)\big)$ is defined as
\begin{equation*}
   E_{\bbeta}(h,f(s,\xi)):=\int_{\mr{Sym}(2n,\mb{Q})\backslash\mr{Sym}(2n,\mb{A})} E\left(\begin{pmatrix}I_{2n}&\varsigma\\0&I_{2n}\end{pmatrix}h,f(s,\xi)\right)\be_{\mb{A}}(-\mr{Tr} \bbeta \varsigma)\,d\varsigma.
\end{equation*}
Suppose $f(s,\xi)=\otimes_v f_v(s,\xi)$ is factorizable. If $\det(\bbeta)\neq 0$ or there exists a finite place $v$ such that $f_v(s,\xi)$ is supported on the ``big cell'' $Q_H(\mb{Q}_v)\begin{pmatrix}0&-I_{2n}\\I_{2n}&0\end{pmatrix} Q_H(\mb{Q}_v)$, then
\begin{equation}\label{eq:FourierProd}
   E_{\bbeta}(h,f(s,\xi))=\prod_{v} W_{\bbeta,v}(h,f(s,\xi))
\end{equation} 
with
\begin{equation*}
   W_{\bbeta,v}(h_v,f_v(s,\xi))=\int_{\mr{Sym}(2n,\mb{Q}_v)}f_v(s,\xi)\left(\begin{pmatrix}0&-I_{2n}\\I_{2n}&\varsigma\end{pmatrix}h_v\right)\be_v(-\mr{Tr}\bbeta\varsigma)\,d_v\varsigma.
\end{equation*}

For $\bz=\bm{x}+\sqrt{-1}\bm{y}$, a point in the Siegel upper half space $\mb{H}_{2n}$, set $h_{\bz}=1_{\mr{f}}\cdot\begin{pmatrix}\sqrt{\bm{y}}&\bm{x}\sqrt{\bm{y}}^{-1}\\0&\sqrt{\bm{y}}^{-1}\end{pmatrix}\in H(\mb{A})$. It is a standard fact that if $E(h_{\bz},f(s_0,\xi))$ is nearly holomorphic as a function in $\bz$ for some $s_0\in 2^{-1}\cdot\mb{Z}$, then $E_{\bbeta}(h_{\bz},f(s_0,\xi))$ gives the $\bbeta$-th coefficient of the $q$-expansion associated to $E(h,f(s_0,\xi))$ viewed as a $p$-adic form by the maps \eqref{eq:embed2}\eqref{eq:AdUp}.

\subsubsection{The unramified places}
For $v\nmid Np\infty$, we simply take
\begin{align*}
   f_{\kappa,\utau^P,v}&=\left.f^{\mr{ur}}_v(s,\phi\chi)\right|_{s=\frac{2n+1}{2}-k}, & f_{\utau^P,v}&=\left.f^{\mr{ur}}_v(s,\phi\epsilon^{P}_d)\right|_{s=\frac{2n+1}{2}-t^P_d}.
\end{align*}
The formulae for $W_{\bbeta}\left(1_v,f_{\kappa,\utau^P,v}\right)$ and $W_{\bbeta}\left(1_v,f_{\utau^P,v}\right)$ are computed by Shimura \cite[Theorem 13.6, Proposition 14.9]{ShEuler} and are listed in the tables in \S\ref{sec:table}. The formulae for the local zeta integrals are part of Theorem \ref{thm:doubling}.

\subsubsection{The archimedean place}\label{eq:finfty}
For an integer $k\geq n+1$ satisfying $\xi(-1)=(-1)^k$, the classical section of weight $k$ in $I_{Q_H,\infty}(s,\xi)$ is defined as
\begin{equation*}
   f_{\infty}^k(s,\mr{sgn}^k)(h)=j(h,i)^{-k}|j(h,i)|^{k-(s+\frac{2n+1}{2})}
\end{equation*} 
where $j(h,i)=\det\left(\mu(h,i)\right)=\det (Ci+D)$ for $h=\begin{pmatrix}A&B\\C&D\end{pmatrix}$. 

Let $\wh{\mu}^+_0=\left(\wh{\mu}^+_{0,ij}\right)_{1\leq i,j\leq n}$, where the entries are elements inside $(\mr{Lie}H)_{\mb{C}}$ given as
\begin{equation*}
  \wh{\mu}^+_{0,ij}=\begin{pmatrix} I_{2n}&\sqrt{-1}\cdot I_{2n}\\\sqrt{-1}\cdot I_{2n}&I_{2n}\end{pmatrix}\begin{pmatrix}0&0&0&E_{ij}\\0&0&E_{ji}&0\\0&0&0&0\\0&0&0&0\end{pmatrix}\begin{pmatrix} I_{2n}&\sqrt{-1}\cdot I_{2n}\\\sqrt{-1}\cdot I_{2n}&I_{2n}\end{pmatrix}^{-1},
\end{equation*} 
where $E_{ij}$ is the $n\times n$ matrix with $1$ as the $(i,j)$-entry and $0$ elsewhere. 

The $\wh{\mu}^+_{0,ij}$'s act on $\mc{A}(H(\mb{Q})\backslash H(\mb{A}))$ by differentiating the right translation of $H(\mb{R})$. Their realizations on the Siegel upper half space are the Maa\ss{}--Shimura differential operators (see \cite[\S2.4]{LiuSLF}). 

For admissible $(\kappa,\utau^P)$ (resp. $\utau^P$) with $\phi\chi(-1)=(-1)^k$ (resp. $\phi\epsilon^P_d(-1)=(-1)^{t^P_d}$), set
\begin{align*}
   f_{\kappa,\utau^P,\infty}(s)=&\prod_{i=1}^{d-1}\mr{det}\left(\frac{(\wh{\mu}^+_0)_{N_i}}{4\pi \sqrt{-1}}\right)^{t^P_i-t^P_{i+1}}\mr{det}\left(\frac{\wh{\mu}^+_0}{4\pi\sqrt{-1}}\right)^{t^P_d-k} f^k_\infty(s,\mr{sgn}^k),\\
   f_{\utau^P,\infty}(s)=&\prod_{i=1}^{d-1}\mr{det}\left(\frac{(\wh{\mu}^+_0)_{N_i}}{4\pi \sqrt{-1}}\right)^{t^P_i-t^P_{i+1}} f^{t^P_d}_\infty(s,\mr{sgn}^{t^P_d})=f_{\tau^P_d,\utau^P,\infty}(s)
\end{align*}
and 
\begin{align*}
  f_{\kappa,\utau^P,\infty}&=f_{\kappa,\utau^P,\infty}\left(\frac{2n+1}{2}-k\right), &f_{\utau^P,\infty}&=f_{\utau^P,\infty}\left(\frac{2n+1}{2}-t^P_d\right)=f_{\tau^P_d,\utau^P,\infty}.
\end{align*}

The formulae for the corresponding Fourier coefficients (listed in tables in \S\ref{sec:table}) are deduced from Shimura's computation \cite[Theorem 4.2]{Sh6} for the classical scalar weight section and formulae for the action of differential operators on $p$-adic expansions (see the proof of \cite[Proposition 4.4.1]{LiuSLF}). The proof of the non-vanishing of the corresponding archimedean zeta integrals is postponed to \S\ref{sec:Zinfty}.

\subsubsection{The ``big cell'' section at a finite place}
We choose our sections at $v|Np$ from a special type of sections, the so-called ``big cell'' sections. Given a finite place $v$ and a compactly supported locally constant function $\alpha_v$ on $\mr{Sym}(2n,\mb{Q}_v)$, the ``big cell'' section inside $I_{Q_H,v}(s,\xi)$ associated to $\alpha_v$ is defined as 
\begin{equation}\label{eq:bigcell}
   f_v^{\alpha_v}(s,\xi)\left(\begin{pmatrix}A&B\\C&D\end{pmatrix}\right)=\begin{cases}\xi^{-1}(\det C)|\det C|^{-(s+\frac{2n+1}{2})}\alpha_v(C^{-1}D)&\text{ if }\det C\neq 0,\\ 0 &\text{ if }\det C=0.\end{cases}
\end{equation}
An easy computation shows that
\begin{equation}\label{eq:ramFour}
   W_{\bbeta,v}(1_v,f_v^{\alpha_v}(s,\xi))= \int_{\mr{Sym}(2n,\bQ_v)} \alpha_v(\varsigma)\be_v(-\mr{Tr} \bbeta\varsigma) \,d_v\varsigma =\wh{\alpha}_v(\bbeta).
\end{equation}

\subsubsection{\texorpdfstring{The volume sections at places dividing $N$}{The volume sections at places dividing N}}
For a positive integer $N$ and a place $v|N$, the volume section $f^{\mr{vol}}_v(s,\xi)$ inside $I_{Q_H,v}(s,\xi)$ is defined as the ``big cell'' section associated to the characteristic function of the open compact subset
\begin{equation}\label{eq:svol}
   -\begin{pmatrix}0&I_n\\I_n&0\end{pmatrix}+N\Sym(2n,\bZ_v)\subset\Sym(2n,\bQ_v).
\end{equation} 
We set
\begin{align*}
   f_{\kappa,\utau^P,v}&=\left.f^{\mr{vol}}_v(s,\phi\chi)\right|_{s=\frac{2n+1}{2}-k}, & f_{\utau^P,v}&=\left.f^{\mr{vol}}_v(s,\phi\epsilon^{P}_d)\right|_{s=\frac{2n+1}{2}-t^P_d}.
\end{align*}

The Fourier coefficients associated to the volume sections are easily computed by computing the Fourier transform of the characteristic function of \eqref{eq:svol}. The computation of local zeta integrals is also easy (the same as \cite[Proposition 4.2.1]{LiuSLF}). See the tables for the formulae.

\subsubsection{\texorpdfstring{The place $p$}{The place p}}\label{sec:placep}
It remains to pick Schwartz functions $\alpha_{\kappa,\utau^P}$ and $\alpha_{\utau^P}$ on $\Sym(2n,\bQ_p)$, and our choices for $f_{\kappa,\utau^P,p}$ (resp. $f_{\utau^P,p}$) will the ``big cell'' section attached to $\alpha_{\kappa,\utau^P}$ (resp. $\alpha_{\utau^P}$). The criterion for picking them is to make the ($p$-adic) $q$-expansions of the resulting Siegel Eisenstein series $p$-adically interpolable. In fact we will first pick $\wh{\alpha}_{\kappa,\utau^P}$ and $\wh{\alpha}_{\utau^P}$, and then apply inverse Fourier transform to get $\alpha_{\kappa,\utau^P}$ and $\alpha_{\utau^P}$ for computing the local zeta integrals. 

The theory of nearly holomorphic forms and Maa\ss{}--Shimura differential operators formulated in terms of automorphic sheaves and their interpretations as $p$-adic Siegel modular forms are needed in our construction. We will freely use the formulation and notation in \cite[\S2]{LiuNHF} and \cite[\S2]{LiuSLF}.

Recall some notation \textit{loc. cit.}; denote by $\cV^r_{\ut}$ the automorphic bundle of degree $r$ and weight $\ut$ nearly holomorphic forms over the Siegel variety defined as in \cite[\S2.2]{LiuSLF}, and by $\mr N^r_{\ut}\left(\bH_n,\Gamma\cap\Gamma_{SP}(p^l)\right)$ the space of vector-valued nearly holomorphic Siegel modular forms on the Siegel upper half space $\bH_n$ of degree $r$, weight $\ut$ and level $\Gamma\cap\Gamma_{SP}(p^l)$ in the sense of Shimura. There is the embedding
\begin{equation}\label{eq:GGembed}
\begin{aligned}
   \mr H^0\left(X^\Sigma_{G,\Gamma\cap \Gamma_{SP}(p^l)}\times X^\Sigma_{G,\Gamma\cap \Gamma_{SP}(p^l)},\cV^r_{\ut}\boxtimes \cV^r_{\ut}\right)&\lhra \mr  N^r_{\ut}\left(\bH_n,\Gamma\cap \Gamma_{SP}(p^l)\right)\otimes \mr N^r_{\ut}\left(\bH_n,\Gamma\cap \Gamma_{SP}(p^l)\right)\\
   &\stackrel{\varphi(\cdot,\fe_{\mr{can}})}{\lhra}\cA(G(\bQ)\times G(\bQ)\backslash G(\bA)\times G(\bA)).
\end{aligned}
\end{equation}

Generalizing the embedding \eqref{eq:embed2}, as explained in \cite[Proposition 3.2.1]{LiuNHF}, the space of nearly holomorphic Siegel modular forms of level $\Gamma\cap\Gamma_{SP}(p^l)$ also embeds into the space of $p$-adic Siegel modular forms,
\begin{equation}\label{eq:nembed}
  \mr H^0\left(X^\Sigma_{G,\Gamma\cap \Gamma_{SP}(p^l)}\times X^\Sigma_{G,\Gamma\cap \Gamma_{SP}(p^l)},\cV^r_{\imath(\ut^P)}\boxtimes \cV^r_{\imath(\ut^P)}\right)\lhra \left(\varprojlim_m\varinjlim_l V^{SP}_{m,l}\otimes_{\cO_F}V^{SP}_{m,l}\right)[1/p].
\end{equation}

When choosing $\alpha_{\kappa,\utau^P}$ and $\alpha_{\utau^P}$, we want to ensure that the restrictions to $G\times G$ of the resulting adelic Siegel Eisenstein series lie inside the image of the embedding \eqref{eq:GGembed}, so that we can study them as $p$-adic Siegel modular forms via \eqref{eq:nembed}. Let $\Sym(2n,\bQ_p)^*$ be the dual space of $\Sym(2n,\bQ_p)$, which can be identified with $\Sym(2n,\bQ_p)$ via the Trace pairing, and let $\Sym(2n,\bZ_p)^*$ be the dual of $\Sym(2n,\bZ_p)$, which identified with the matrices $\alpha$ in $\Sym(2n,\bQ_p)$ such that $\mr{Tr}(\alpha a) \in \bZ_p$ for all $a \in \Sym(2n,\bZ_p)$.  Hence we require:
\begin{enumerate}
\item $\wh{\alpha}_{\kappa,\utau^P}$ and $\wh{\alpha}_{\utau^P}$ take values in a finite extension of $\bQ$.
\item $\wh{\alpha}_{\kappa,\utau^P}$ and $\wh{\alpha}_{\utau^P}$ are supported on $\Sym(2n,\bZ_p)^*$, and for $x\in\Sym(2n,\bQ_p)^*$, $a_1,a_2\in SP(\bZ_p)$,
\begin{align*}
   \wh{\alpha}_{\kappa,\utau^P}(x)&=\wh{\alpha}_{\kappa,\utau^P}\left(\begin{psm}\ltrans{a}_1&0\\0&\ltrans{a}_2\end{psm}x\begin{psm}a_1&0\\0&a_2\end{psm}\right), & \wh{\alpha}_{\utau^P}(x)&=\wh{\alpha}_{\utau^P}\left(\begin{psm}\ltrans{a}_1&0\\0&\ltrans{a}_2\end{psm}x\begin{psm}a_1&0\\0&a_2\end{psm}\right).
\end{align*}
\end{enumerate}

With $\wh{\alpha}_{\kappa,\utau^P}$ and $\wh{\alpha}_{\utau^P}$ satisfying these conditions, we can define 
\begin{align*}
   \cE_{\kappa,\utau^P},\,\cE_{\utau^P}&\in \varinjlim_r\varinjlim_l \mr H^0\left(X^\Sigma_{G,\Gamma\cap \Gamma_{SP}(p^l)}\times X^\Sigma_{G,\Gamma\cap \Gamma_{SP}(p^l)},\cV^r_{\imath(\ut^P)}\boxtimes \cV^r_{\imath(\ut^P)}\right),
\end{align*}
as the preimage of the adelic forms
\begin{equation}\label{eq:norE}
\begin{aligned}
   &\chi(-1)^{n} \frac{\Gamma_{2n}\left(\frac{2n+1}{2}\right)}{2^{2nk-2n^2+n}\pi^{n+2n^2}}\, \left.E^*\left(\cdot,f_{\kappa,\utau^P}\right)\right|_{G\times G}, &\quad\epsilon^P_d(-1)^{n} \frac{\Gamma_{2n}\left(\frac{2n+1}{2}\right)}{2^{2nt^P_d-2n^2+n}\pi^{n+2n^2}}\cdot \left.E^*\left(\cdot,f_{\utau^P}\right)\right|_{G\times G}.
\end{aligned}
\end{equation}

Here for a positive integer $m$,
\begin{equation*}
   \Gamma_m(s):=\pi^{\frac{m(m-1)}{4}}\prod_{j=0}^{m-1}\Gamma\left(s-\frac{j}{2}\right).
\end{equation*}
In the following we will not distinguish $\cE_{\kappa,\utau^P}$, $\cE_{\utau^P}$ from their images under the embedding \eqref{eq:nembed}.

Set 
\begin{align*}
   \varepsilon_{\qexp}=(\varprojlim_{m}\varinjlim_l \varepsilon^{1,1}_{\qexp,m,l},\,\varprojlim_{m}\varinjlim_l \varepsilon^{1,1}_{\qexp,m,l}):\varprojlim_m&\varinjlim_l V^{SP}_{m,l}\otimes_{\cO_F}V^{SP}_{m,l}\lra \cO_F\llbracket N^{-1}\Sym(n,\bZ)^{\ast\oplus 2}_{\geq 0}\rrbracket
\end{align*}
with 
\begin{equation*}
   \varprojlim_{m}\varinjlim_l \varepsilon^{1,1}_{\qexp,m,l}:\varprojlim_m\varinjlim_l V^{SP}_{m,l}\lra \cO_{F}\llbracket N^{-1}S^2(X_n)_{\geq 0}\rrbracket=\cO_F\llbracket N^{-1}\Sym(n,\bZ)^*_{\geq 0}\rrbracket.
\end{equation*}
being the $p$-adic $q$-expansion map at infinity.

For any element $G$ of $\varprojlim_m \varinjlim_l V^{SP}_{m,l}\otimes_{\cO_F}V^{SP}_{m,l}$ with $q$-expansion \[
\varepsilon_{\qexp}(G) =\sum_{\beta_1,\beta_2\in N^{-1}\Sym(n,\bZ)^*_{>0}}\,\sum_{\bbeta=\begin{psm}\beta_1&\beta_0\\\ltrans{\beta_0}&\beta_2\end{psm}\in N^{-1}\Sym(2n,\bZ)^*_{>0}}\fc(\bbeta) q^{\beta_1}q^{\beta_2} \]
and any pair $(\beta_1,\beta_2)$ of elements in $N^{-1}\Sym(n,\bZ)^*_{\geq 0}$ we define the $(\beta_1,\beta_2)$ Fourier coefficient of $G$ as 
\[
\varepsilon_{\qexp}\left(\beta_1,\beta_2,G \right) =\sum_{\bbeta=\begin{psm}\beta_1&\beta_0\\\ltrans{\beta_0}&\beta_2\end{psm}\in N^{-1}\Sym(2n,\bZ)^*_{>0}}\fc(\bbeta).
\]

Regarding the $q$-expansion of $\cE_{\kappa,\utau^P}$ and $\cE_{\utau^P}$, we have the following proposition.
\begin{prop}\label{prop:FourierCoeffEis}
Let $(\beta_1,\beta_2)$ be a pair of elements in $N^{-1}\Sym(n,\bZ)^*_{\geq 0}$. For admissible $(\kappa,\utau^P)\in \Hom_{\mr{cont}}\left(\bZ^\times_p\times T_P(\bZ_p),\ol{\bQ}^\times_p\right)$ with $\phi\chi(-1)=(-1)^k$, we have
\begin{align*}
   \varepsilon_{\qexp}\left(\beta_1,\beta_2,\cE_{\kappa,\utau^P}\right)&=\sum_{\bbeta=\begin{psm}\beta_1&\beta_0\\\ltrans{\beta}_0&\beta_2\end{psm}\in N^{-1}\Sym(2n,\bZ)^*_{\geq 0}}\fc_{\kappa,\utau^P}(\bbeta),
\end{align*}
where for $\bbeta>0$, the coefficient $\fc_{\kappa,\utau^P}(\bbeta)$ is given as
\begin{align*}
   \fc_{\kappa,\utau^P}(\bbeta)=&N^{-n(2n+1)}\prod_{v|N}\be_v(2\mr{Tr}\beta_0) \cdot L^{Np\infty}(n+1-k,\phi\chi\lambda_{\bbeta})\cdot \prod_{\substack{v|\det(2\bbeta)\\v\nmid Np\infty}} g_{\bbeta,v}\left(\phi\chi(q_v)q_v^{k-2n-1}\right)\\
   &\times\, (-1)^{nk}\chi(-1)^n\cdot\wh{\alpha}_{\kappa,\utau^P}(\bbeta) \prod_{i=1}^{d-1}\mr{det}((2\beta_0)_{N_i})^{t^P_i-t^P_{i+1}}\mr{det}(2\beta_0)^{t^P_d-k}.
\end{align*}
For admissible $\utau\in \Hom_{\mr{cont}}(T_P(\bZ_p),\ol{\bQ}^\times_p)$ with $\phi\epsilon^P_d(-1)=(-1)^{t^P_d}$, we have 
\begin{align*}
   \varepsilon_{\qexp}\left(\beta_1,\beta_2,\cE_{\utau^P}\right)&=\sum_{\bbeta=\begin{psm}\beta_1&\beta_0\\\ltrans{\beta}_0&\beta_2\end{psm}\in N^{-1}\Sym(2n,\bZ)^*_{\geq 0}}\fc_{\utau^P}(\bbeta),
\end{align*}
where for $\bbeta\geq 0$, the coefficient $\fc_{\utau^P}(\bbeta)$ is given as
\begin{align*}
   \fc_{\utau^P}(\bbeta)=&N^{-n(2n+1)}\prod_{v|N}\be_v(2\mr{Tr}\beta_0) \cdot L^{Np\infty}(n+1-t^P_d-\frac{r}{2},\phi\chi\lambda_{\bbeta})\cdot \prod_{j=1}^{\frac{\mr{corank}(\bbeta)}{2}}L^{Np\infty}(2n+3-2t^P_d-2j,(\phi\epsilon^{P}_d)^2)\\
   &\,\times(-1)^{nt^P_d}\epsilon^P_d(-1)^n\prod_{\substack{v|\det^*(2\bbeta)\\v\nmid Np\infty}} g_{\bbeta,v}\left(\phi\chi(q_v)q_v^{t^P_d-2n-1}\right)\cdot \wh{\alpha}_{\utau^P}(\bbeta) \prod_{i=1}^{d-1}\mr{det}((2\beta_0)_{N_i})^{t^P_i-t^P_{i+1}}
\end{align*}
if $\mr{rank}(\bbeta)$ is even, and
\begin{align*}
   \fc_{\utau^P}(\bbeta)=&N^{-n(2n+1)}\prod_{v|N}\be_v(2\mr{Tr}\beta_0) \cdot \prod_{j=1}^{\frac{\mr{corank}(\bbeta)+1}{2}}L^{NP\infty}(2n+3-2t^P_d-2j,(\phi\epsilon^{P}_d)^2)\\
   &\,\times(-1)^{nt^P_d}\epsilon^P_d(-1)^n\prod_{\substack{v|\det^*(2\bbeta)\\v\nmid Np\infty}} g_{\bbeta,v}\left(\phi\chi(q_v)q_v^{t^P_d-2n-1}\right)\cdot \wh{\alpha}_{\utau^P}(\bbeta) \prod_{i=1}^{d-1}\mr{det}((2\beta_0)_{N_i})^{t^P_i-t^P_{i+1}}   
\end{align*}
if $\mr{rank}(\bbeta)$ is odd.

Here for $\bbeta\in \Sym(2n,\bQ)\cap \Sym(2n,\bZ_v)^*$, $\det^*(2\bbeta)$ denotes the product of all the nonzero eigenvalues of $2\bbeta$. If $\mr{rank}(\bbeta)$ is even, the quadratic character $\lambda_{\bbeta}$ is defined as $\lambda_{\bbeta}(q_v)=\left(\frac{(-1)^{\mr{rank}(\bbeta)/2}\det^*(\bbeta)}{q_v}\right)$. The $g_{\bbeta,v}(\cdot)$ appearing in above formulae is a polynomial with coefficients in $\bZ$. For an integer $m$, $(2\beta_0)_m$ denotes the upper left $m\times m$-minor of $2\beta_0$.
\end{prop}
\begin{proof}
The proof is similar as \cite[Proposition 4.4.1]{LiuSLF}. It relies on the formulae for local Fourier coefficients as listed in the two tables in \S\ref{sec:table}, and uses formulae of differential operators on $p$-adic $q$-expansions. 
\end{proof}

It is not difficult to observe that all the terms in the above formulae for $\fc_{\kappa,\utau^P}(\bbeta)$, $\fc_{\utau^P}(\bbeta)$ are ready for $p$-adic interpolation with respect to $(\kappa,\utau^P)$, $\utau^P$, except the last terms
\begin{align}
  &\wh{\alpha}_{\kappa,\utau^P}(\bbeta) \prod_{i=1}^{d-1}\mr{det}((2\beta_0)_{N_i})^{t^P_i-t^P_{i+1}}\mr{det}(2\beta_0)^{t^P_d-k},\label{eq:rterm}\\
   &\wh{\alpha}_{\utau^P}(\bbeta)\prod_{i=1}^{d-1}\mr{det}((2\beta_0)_{N_i})^{t^P_i-t^P_{i+1}}. \label{eq:rterm2}
\end{align}
(The $p$-adic interpolation of the Dirichlet $L$-values in the formulae follows from the existence of the Kubota--Leopoldt $p$-adic $L$-function \cite[Theorem 4.4.1]{H}.)

In order to make \eqref{eq:rterm} and \eqref{eq:rterm2} $p$-adically interpolable, one needs to require that if $\bbeta$ belongs to the support of the Schwartz function $\wh{\alpha}_{\kappa,\utau^P}$ (resp. $\wh{\alpha}_{\utau^P}$), then $\det\left((2\beta_0)_{N_i}\right)$ is a $p$-adic unit for $1\leq i\leq d$ (resp. $1\leq i\leq d-1$). The very natural choices of $\wh{\alpha}_{\kappa,\utau^P}$ and $\wh{\alpha}_{\utau^P}$ are
\begin{align*}
\numberthis\label{eq:whalpha1}   \wh{\alpha}_{\kappa,\utau^P}(\bbeta) =&\mathds{1}_{p^2\Sym(n,\bZ_p)^*}(\beta_1)\mathds{1}_{\Sym(n,\bZ_p)^*}(\beta_2)\prod_{i=1}^{d}\mathds{1}_{\GL(N_i,\bZ_p)}\left((2\beta_0)_{N_i}\right)\\
   &\times\,\prod_{i=1}^{d-1}\epsilon^P_i\epsilon^{P-1}_{i+1}(\mr{det}((2\beta_0)_{N_i}))\cdot \epsilon^{P}_d\chi^{-1}(\mr{det}(2\beta_0)),
\end{align*}
and
\begin{align*}
\numberthis\label{eq:whalpha2}   \wh{\alpha}_{\utau^P}(\bbeta) =&\mathds{1}_{p^2\Sym(n,\bZ_p)^*}(\beta_1)\mathds{1}_{\Sym(n,\bZ_p)^*}(\beta_2)\mathds{1}_{M_n(\bZ_p)}(\beta_0)\prod_{i=1}^{d-1}\mathds{1}_{\GL(N_i,\bZ_p)}\left((2\beta_0)_{N_i}\right)\\
   &\times\,\prod_{i=1}^{d-1}\epsilon^P_i\epsilon^{P-1}_{i+1}(\mr{det}((2\beta_0)_{N_i})).
\end{align*}
Here $M_n$ denotes the space of $n\times n$ matrices.

Since $W_{\bbeta,p}(1_p,f_{\kappa,\utau^P,p})=\wh{\alpha}_{\kappa,\utau^P}(\bbeta)$, our choice \eqref{eq:whalpha1} makes $\fc_{\kappa,\utau^P}(\bbeta)$ vanish unless $\bbeta$ is invertible. The semi-positivity of $\bbeta$ implies that both $\beta_1$ and $\beta_2$ are positive definite. Thus $\varepsilon_{\qexp}\left(\beta_1,\beta_2,\cE_{\kappa,\utau^P}\right)$ is nonzero only if $\beta_1,\beta_2>0$. Similarly, $\varepsilon_{\qexp}\left(\beta_1,\beta_2,\cE_{\utau^P}\right)$ is nonzero only if $\beta_1,\beta_2\geq 0$ and their ranks are at least $n-n_d$.

With $\wh{\alpha}_{\kappa,\utau^P}$, $\wh{\alpha}_{\utau^P}$ being set as in \eqref{eq:whalpha1}, \eqref{eq:whalpha2}, we have
\begin{align*}
   \varepsilon_{\qexp}\left(\cE_{\kappa,\utau^P}\right)&=\sum_{\beta_1,\beta_2\in N^{-1}\Sym(n,\bZ)^*_{>0}}\,\sum_{\bbeta=\begin{psm}\beta_1&\beta_0\\\ltrans{\beta_0}&\beta_2\end{psm}\in N^{-1}\Sym(2n,\bZ)^*_{>0}}\fc_{\kappa,\utau^P}(\bbeta)\,q^{\beta_1}q^{\beta_2},\\
   \varepsilon_{\qexp}\left(\cE_{\utau^P}\right)&=\sum_{\substack{\beta_1,\beta_2\in\Sym(n,\bZ)^*_{\geq 0}\\\mr{rk}(\beta_1),\mr{rk}(\beta_2)\geq n-n_d}}\,\sum_{\bbeta=\begin{psm}\beta_1&\beta_0\\\ltrans{\beta_0}&\beta_2\end{psm}\in N^{-1}\Sym(2n,\bZ)^*_{\geq 0}}\fc_{\utau^P}(\bbeta)\,q^{\beta_1}q^{\beta_2},
\end{align*}
and each $\fc_{\kappa,\utau^P}(\bbeta)$ (resp. $\fc_{\utau^P}(\bbeta)$) appearing here admits $p$-adic interpolation with respect to $(\kappa,\utau^P)$ (resp. $\utau^P$).

If we look at the $q$-expansions of $\cE_{\kappa,\utau^P}$ and $\cE_{\utau^P}$ at other cusps in the ordinary locus (which is equivalent to look at $W_{\bbeta}\left(\iota(g_1,g_2),f_{\kappa,\utau^P}\right)$, $W_{\bbeta}\left(\iota(g_1,g_2),f_{\utau^P}\right)$ for $g_i\in G(\bA)$ with $g_{i,p}=1$, $g_{i,\infty}=g_{z_i}$, $i=1,2$), the support of $\wh{\alpha}_{\kappa,\utau^P}$ (resp. $\wh{\alpha}_{\utau^P}$) again makes the term indexed by degenerate $(\beta_1,\beta_2)$ (resp. $\beta_1$ or $\beta_2$ of rank $<n-n_d$) vanish. Hence
\begin{align*}
   \numberthis\label{eq:cEV}\cE_{\kappa,\utau^P}&\in \varprojlim_m\varinjlim_l V^{SP,0}_{m,l}\otimes_{\cO_F} V^{SP,0}_{m,l}, &\cE_{\utau^P}&\in \varprojlim_m\varinjlim_l V^{SP,n_d}_{m,l}\otimes_{\cO_F} V^{SP,n_d}_{m,l}.
\end{align*}

\begin{rem}\label{rem:biggersupport}
Compared to $\wh{\alpha}_{\kappa,\utau^P}$, the support of $\wh{\alpha}_{\utau^P}$ is enlarged. As the cyclotomic variable $\kappa$ is fixed to be equal to $\uptau^P_d$, the term $\det(2\beta_0)$ does not appear in \eqref{eq:rterm2}, and one does not need to require $2\beta_0\in\GL(n,\bZ_p)$ for the support of $\wh{\alpha}_{\utau^P}(\bbeta)$. Later, we will see that it is this relaxation on the support that saves us the factor $\cA^P(\pi \times \phi\epsilon^P_d)$ in the local zeta integral for $f^{\alpha_{\utau^P}}_p(\frac{2n+1}{2}-t^P_d,\phi\epsilon^{P}_d)$ compared to that for $\left.f^{\alpha_{\kappa,\utau^P}}_p(\frac{2n+1}{2}-k,\phi\chi)\right|_{\kappa=\uptau^P_d}$. This relaxation also means that the resulting $\cE_{\utau^P}$ is not necessarily cuspidal as $p$-adic forms. Thus, the Hida theory for non-cuspidal Siegel modular forms developed in \S\ref{sec:NCH} is needed to construct the improved $p$-adic $L$-function from the $\cE_{\utau^P}$'s.
\end{rem}

\subsubsection{The two tables}\label{sec:table}
In the tables on the next two pages, we summarize our choices of sections for Siegel Eisenstein series, the formulae of the corresponding local Fourier coefficients, and the local zeta integrals.

We explain some notation. In the tables, $\varphi\in\pi$ is a $P$-ordinary cuspidal holomorphic Siegel modular form  of weight $\imath(\ut^P)$ fixed by $\wh{\Gamma}\cap SP_G(\mb{Z}_p)$ with $p$-nebentypus $\uep^P$. 

The operator $\pW$ is defined as
\begin{equation}\label{eq:pW}
   \pW(\varphi)(g)=\int_{SP_G(\mb{Z}_p)}\ol{\varphi}^\mvw(gu)\,du.
\end{equation}
As observed before Remark \ref{rem:notationMVW}, the form $\ol{\varphi}^\mvw$ belongs to $\pi$ and so $\pW(\varphi)$ belongs to $\pi$ too. If $\varphi$ is holomorphic of weight $\ut^P$, then so is $\pW(\varphi)$. The operator $\pW$ should be viewed as an analogue of the operator sending a modular form $f$ of level $\Gamma_0(N)$ to $f^c\left|\begin{pmatrix} 0 & -1\\N & 0\end{pmatrix}\right.$. In the same way as \cite[Proposition 5.7.2]{LiuSLF}, one can show that the $P$-ordinary projection $e_P\pW(\varphi)$ is nonzero if $e_P\varphi$ is nonzero. 

We denote by $v_{\imath(\ut^P)}\in\cD_{\imath(\ut^P)}$ the highest weight vector inside the lowest $K_{G,\infty}$-type, and $v^\vee_{\imath(\ut^P)}\in\wt{\cD}_{\imath(\ut^P)}$ is the dual vector to $v_{\imath(\ut^P)}$.


\clearpage

\begin{table}
{\footnotesize
\begin{center}
\begin{tabular}{p{2em}|c|c|c}
 & \begin{minipage}{12.5em}\begin{equation*}f_{\kappa,\utau^P,v}\end{equation*}\end{minipage}
 &\begin{minipage}{18.5em}$W_{\bbeta,v}\left(h_{\bz,v},f_{\kappa,\utau^P,v}\right)$ with $\bbeta=\left(\begin{smallmatrix}\beta_1&\beta_0\\\ltrans{\beta}_0&\beta_2\end{smallmatrix}\right)$\end{minipage}
 & \begin{minipage}{14em}\begin{equation*}\left(T_{f_{\kappa,\utau^P,v}}\ol{\varphi}\right)^\mvw\end{equation*}\end{minipage} \\
\hline 
 $v\nmid Np\infty$ &\begin{minipage}{12.5em}the standard unramified section \begin{equation*}\left.f^{\mr{ur}}_v(s,\phi\chi)\right|_{s=\frac{2n+1}{2}-k}\end{equation*}\end{minipage} 
 &{\begin{minipage}{18.5em}  \begin{align*}&\mathds{1}_{\mr{Sym}(2n,\bZ_v)^*}(\bbeta)\\
 \times &\,d_v(n+1-k,\phi\chi)^{-1}L_v(n+1-k,\phi\chi\lambda_{\bbeta})\\
 \times &\, g_{\bbeta,v}\left(\phi\chi(q_v)q_v^{k-2n-1}\right)\end{align*}
for $\mr{det}(\bbeta)\neq 0$, where $g_{\bbeta,v}(T)\in\mb{Z}[T]$ with $g_{\bbeta,v}(0)=1$ and degree at most $4n\cdot \mr{val}_v(\mr{det}(2\bbeta))$ \\\end{minipage}} 
& {\begin{minipage}{14em}\begin{align*} &d_v(n+1-k,\phi\chi)^{-1}\\  \cdot &L_v(n+1-k,\pi\times\phi\chi)\cdot \pW(\varphi)\end{align*}\end{minipage}} \\
 \hline
 $v|N$ & \begin{minipage}{12.5em} \phantom{l}\\ the ``big cell'' section\begin{equation*}\left.f^{\mr{vol}}_{v}(s,\phi\chi)\right|_{s=\frac{2n+1}{2}-k}\end{equation*} associated to  the characteristic function of \begin{equation*}-\left(\begin{smallmatrix}0&I_n\\I_n&0\end{smallmatrix}\right)+N\mr{Sym}(2n,\bZ_v)\end{equation*}\\\end{minipage}
 & \begin{minipage}{18.5em}\begin{equation*}|N|_v^{n(2n+1)}\be_v(2\mr{Tr}\beta_0)\cdot\mathds{1}_{N^{-1}\mr{Sym}(2n,\bZ_v)^*}(\bbeta)\end{equation*}\end{minipage}
 &\begin{minipage}{14em}\begin{equation*}\phi_v\chi_v((-1)^n)\mr{vol}(\Gamma(N)_v)\cdot \pW(\varphi)\end{equation*}\end{minipage}\\
 \hline
 $v=p$ & \begin{minipage}{12.5em}the ``big cell'' section\begin{equation*} \left.f^{\alpha_{\kappa,\utau^P}}_p(s,\phi\chi)\right|_{s=\frac{2n+1}{2}-k},\end{equation*} where $\alpha_{\kappa,\utau^P}$ is the inverse Fourier transform of the Schwartz function in the next column\end{minipage} 
 &\begin{minipage}{18.5em}{\begin{align*}&\,\wh{\alpha}_{\kappa,\utau^P}(\bbeta)\\
 =&\,\mathds{1}_{p^2\mr{Sym}(n,\mb{Z}_p)^*}(\beta_1)\mathds{1}_{\mr{Sym}(n,\mb{Z}_p)^*}(\beta_2)\\
 &\times\prod_{i=1}^d\mathds{1}_{\mr{GL}(N_i,\mb{Z}_p)}((2\beta_0)_{N_i})\\
 &\times\prod_{i=1}^{d-1}\epsilon^P_i\epsilon^{P-1}_{i+1}(\mr{det}((2\beta_0)_{N_i}))\\&\times \epsilon^{P}_d\chi^{-1}(\mr{det}(2\beta_0))\end{align*}} \\\end{minipage}
 &\begin{minipage}{14em} {\begin{align*}&\chi(-1)^n\frac{\prod_{i=1}^d\prod_{j=1}^{n_i}(1-p^{-j})}{\prod_{j=1}^n(1-p^{-2j})}\\&\cdot E_p(n+1-k,\pi\times\phi\chi)\cdot e_P\pW(\varphi)\end{align*}}(after $f_{\kappa,\utau^P,p}$ being modified by appropriate $\bU^P_p$-operators in accordance with ordinary projection applied to $\left.E^*(\cdot,f_{\kappa,\utau^P})\right|_{G\times G}$) \end{minipage}\\
\hline
$v=\infty$ & \begin{minipage}{12.5em}{\begin{align*}\Bigg(&\prod_{i=1}^{d-1}\mr{det}\left(\frac{(\wh{\mu}^+_0)_{N_i}}{4\pi \sqrt{-1}}\right)^{t^P_i-t^P_{i+1}}\\&\cdot\mr{det}\left(\frac{\wh{\mu}^+_0}{4\pi\sqrt{-1}}\right)^{t^P_d-k}\\
&\cdot f^k_\infty(s,\mr{sgn}^k)\left.\Bigg)\right|_{s=\frac{2n+1}{2}-k}\end{align*}}\end{minipage}
&\begin{minipage}{19em}\phantom{l}\\vanishing unless $\bbeta\geq 0$  and equals {\begin{align*}&(-1)^{nk}\frac{2^{2nk-2n^2+n}}{\Gamma_{2n}\left(\frac{2n+1}{2}\right)}\pi^{n+2n^2}\\\times&\prod_{i=1}^{d-1}\mr{det}((2\beta_0)_{N_i})^{t^P_i-t^P_{i+1}}\mr{det}(2\beta_0)^{t^P_d-k}\\\times&\prod_{i=1}^{d-1}\mr{det}((y_0)_{N_i})^{t^P_i-t^P_{i+1}}\mr{det}(y_0)^{t^P_d-k}\\\times&\mr{det}(\mathbf{y})^{\frac{k}{2}}\be_\infty(\mr{Tr}\bbeta\bz)\\&+\text{terms irrelevant to } \varepsilon_{\qexp}(\cE_{\kappa,\utau^P})\end{align*}}where $\mathbf{y}=\mr{Im}(\bz)=\begin{psm}y_1&y_0\\\ltrans{y}_0&y_2\end{psm}$\end{minipage}
&\begin{minipage}{12em}{\begin{align*}&\frac{2^{-2n^2+n+2nk}\pi^{2n^2+n}}{\Gamma_{2n}\left(\frac{2n+1}{2}\right)}\\  \cdot &\,\frac{\sqrt{-1}^{\frac{n^2}{2}-\frac{n}{2}}2^{-\frac{3n^2}{2}-\frac{3n}{2}-\sum_{j=1}^d n_jt^P_j}}{\dim\left(\GL(n),\imath(\ut^P)\right)}\\ \cdot &E_\infty(1+n-k,\pi\times\phi\chi)\cdot \pW(\varphi)\end{align*}}\end{minipage}
\end{tabular}
\caption{data for $(d+1)$-variable $p$-adic $L$-function}
\label{tb:d+1}
\end{center}
}
\end{table}
\clearpage

\begin{table}
{\footnotesize
\begin{center}
\begin{tabular}{p{2em}|c|c|c}
 & \begin{minipage}{12.7em}\begin{equation*}f_{\utau^P,v}\end{equation*}\end{minipage}&\begin{minipage}{20.5em}$W_{\bbeta,v}\left(h_{\bz},f_{\utau^P,v}\right)$ with $\bbeta=\left(\begin{smallmatrix}\beta_1&\beta_0\\\ltrans{\beta}_0&\beta_2\end{smallmatrix}\right)$\end{minipage}& \begin{minipage}{12em}\begin{equation*}\left(T_{f_{\utau^P,v}}\ol{\varphi}\right)^\mvw \end{equation*}\end{minipage} \\
\hline 
 $v\nmid Np\infty$ &
 \begin{minipage}{12.7em}the standard unramified section \begin{equation*}\left.f^{\mr{ur}}_v(s,\phi\epsilon^{P}_d)\right|_{s=\frac{2n+1}{2}-t^P_d}\end{equation*}\end{minipage} & 
  {\begin{minipage}{20.5em} 
   \begin{align*}&\mathds{1}_{\mr{Sym}(2n,\bZ_v)^*}(\bbeta)\cdot d_v(n+1-k,\phi\epsilon^{P}_d)^{-1}\\
  \times&\left\{ 
  \begin{array}{l} 
 L_v(n+1-t^P_d-\frac{r}{2},\phi\epsilon^{P}_d\lambda_{\bbeta})\hspace{2.5em}r\text{ even,}\\
   \times\prod\limits_{j=1}^{r/2} L_v(2n+3-2t^P_d-2j,(\phi\epsilon^{P}_d)^2)\\[3ex]
   \prod\limits_{j=1}^{\frac{r+1}{2}}L_v(2n+3-2t^P_d-2j,(\phi\epsilon^{P}_d)^2)\hspace{0.5em}r\text{ odd,}
  \end{array}
  \right.\\
  \times &\, g_{\bbeta,v}\left(\phi\chi(q_v)q_v^{k-2n-1}\right)\end{align*}
for $\bbeta \geq 0$ with rank $2n-r$, where $g_{\bbeta,v}(T)\in\mb{Z}[T]$ with $g_{\bbeta,v}(0)=1$ and degree at most $4n\cdot \mr{val}_v(\mr{det}^*(2\bbeta))$ \\\end{minipage}} 
& {\begin{minipage}{12em}\begin{align*}&d_v(n+1-k,\phi\epsilon^{P}_d)^{-1}\\\cdot&L_v(n+1-t^P_d,\pi\times\phi\epsilon^{P}_d)\cdot \pW(\varphi)\end{align*}\end{minipage}} 
\\
 \hline
 $v|N$ 
 &  \begin{minipage}{12.7em}  \phantom{l}\\the ``big cell'' section\begin{equation*}\left.f^{\mr{vol}}_{v}(s,\phi\epsilon^{P}_d)\right|_{s=\frac{2n+1}{2}-t^P_d}\end{equation*} associated to  the characteristic function of \begin{equation*}-\left(\begin{smallmatrix}0&I_n\\I_n&0\end{smallmatrix}\right)+N\mr{Sym}(2n,\bZ_v)\end{equation*}\\ \end{minipage}
 & \begin{minipage}{20.5em} \begin{equation*}|N|_v^{n(2n+1)}\be_v(2\mr{Tr}\beta_0)\cdot\mathds{1}_{N^{-1}\mr{Sym}(2n,\bZ_v)^*}(\bbeta)\end{equation*}\end{minipage}
 &\begin{minipage}{12em}\begin{equation*}\phi_v(-1)^n\mr{vol}(\Gamma(N)_v)\cdot \pW(\varphi)\end{equation*}\end{minipage}\\
 \hline
 $v=p$ & \begin{minipage}{12.7em}the ``big cell'' section\begin{equation*} \left.f^{\alpha_{\utau^P}}_p(s,\phi\epsilon^{P}_d)\right|_{s=\frac{2n+1}{2}-t^P_d},\end{equation*} where $\alpha_{\utau^P}$ is the inverse Fourier transform of the Schwartz function in the next column\end{minipage} &\begin{minipage}{20.5em}{\begin{align*}\wh{\alpha}_{\utau^P}(\bbeta)
 =&\,\mathds{1}_{p^2\mr{Sym}(n,\mb{Z}_p)^*}(\beta_1)\mathds{1}_{\mr{Sym}(n,\mb{Z}_p)^*}(\beta_2)\\
 &\times\mathds{1}_{M_n(\mb{Z}_p)}(\beta_0)\prod_{i=1}^{d-1}\mathds{1}_{\mr{GL}(N_i,\mb{Z}_p)}((2\beta_0)_{N_i})\\
 &\times\prod_{i=1}^{d-1}\epsilon^P_i\epsilon^{P-1}_{i+1}(\mr{det}((2\beta_0)_{N_i}))\end{align*}}The major difference from the previous $\wh{\alpha}_{\kappa,\utau^P}$ is that here the support of $\wh{\alpha}_{\utau^P}$ has been enlarged and is no longer contained in $\GL(2n,\bZ_p)$.\\\end{minipage}
 &\begin{minipage}{12em} {\begin{align*}&\epsilon^P_d(-1)^n\frac{\prod_{i=1}^d\prod_{j=1}^{n_i}(1-p^{-j})}{\prod_{j=1}^n(1-p^{-2j})}\\&\cdot E^{\imp}_p(n+1-t^P_d,\pi\times\phi\epsilon^{P}_d)\\&\cdot e_P\pW(\varphi)\end{align*}(after $f_{\utau^P,p}$ being modified by appropriate $\bU^P_p$-operators in accordance with ordinary projection applied to $\left.E^*(\cdot,f_{\utau^P})\right|_{G\times G}$)}
  \end{minipage}\\
\hline
$v=\infty$ & \begin{minipage}{12.7em}{\begin{align*}\Bigg(&\prod_{i=1}^{d-1}\mr{det}\left(\frac{(\wh{\mu}^+_0)_{N_i}}{4\pi \sqrt{-1}}\right)^{t^P_i-t^P_{i+1}}\\
&\cdot f^{t^P_d}_\infty(s,\mr{sgn}^{t^P_d})\left.\Bigg)\right|_{s=\frac{2n+1}{2}-t^P_d}\end{align*}}\end{minipage}
&\begin{minipage}{20.5em}\phantom{l}\\vanishing unless $\bbeta\geq 0$  and equals {\begin{align*}&(-1)^{nt^P_d}\frac{2^{2nt^P_d-2n^2+n}}{\Gamma_{2n}\left(\frac{2n+1}{2}\right)}\pi^{n+2n^2}\\\times&\prod_{i=1}^{d-1}\mr{det}((2\beta_0)_{N_i})^{t^P_i-t^P_{i+1}}\\\times&\prod_{i=1}^{d-1}\mr{det}((y_0)_{N_i})^{t^P_i-t^P_{i+1}}\mr{det}(\mathbf{y})^{\frac{t^P_d}{2}}\be_\infty(\mr{Tr}\bbeta\bz)\\&+\text{terms irrelevant to } \varepsilon_{\qexp}(\cE_{\kappa,\utau^P})\end{align*}} where $\mathbf{y}=\mr{Im}(\bz)=\begin{psm}y_1&y_0\\\ltrans{y}_0&y_2\end{psm}$\end{minipage}
&\begin{minipage}{12em}
{\begin{align*}&\frac{2^{-2n^2+n+2nt^P_d}\pi^{2n^2+n}}{\Gamma_{2n}\left(\frac{2n+1}{2}\right)}\\  \cdot &\,\frac{\sqrt{-1}^{\frac{n^2}{2}-\frac{n}{2}}2^{-\frac{3n^2}{2}-\frac{3n}{2}-\sum_{j=1}^d n_jt^P_j}}{\dim\left(\GL(n),\imath(\ut^P)\right)}\\ \cdot &E_\infty(1+n-t^P_d,\pi\times\phi\epsilon^P_d)\cdot \pW(\varphi) 
\end{align*}}\end{minipage}
\end{tabular}
\caption{data for $d$-variable improved $p$-adic $L$-function}\label{tb:d}
\end{center}
}
\end{table}

\clearpage

\subsection{\texorpdfstring{The construction of $\mu_{\cE,\Pord}$ and $\cE^{\mr{imp}}_{\Pord}$}{The construction of mu E and E-imp}}

Assume $\phi^2\neq \mr{triv}$. Then it follows from our choices of the sections $f_{\kappa,\utau^P}\in I_{Q_H}(\frac{2n+1}{2}-k,\phi\chi)$ and $f_{\utau^P}\in I_{Q_H}(\frac{2n+1}{2}-t^P_d,\phi\epsilon^P_d)$ that there exist $p$-adic measures 
\begin{align*}
  \mu_{\cE,\qexp}&\in \mM eas\left(\bZ^\times_p\times T_P(\bZ_p),\cO_F \llbracket N^{-1}\Sym(n,\bZ)^{*\oplus 2}_{>0}\rrbracket\right),\\ 
  \mu_{\cE^{\mr{imp}},\qexp}&\in\mM eas \left(T_P(\bZ_p),\cO_F\llbracket N^{-1}\Sym(n,\bZ)^{*\oplus 2}_{\geq 0}\rrbracket\right)
\end{align*}
satisfying the interpolation properties
\begin{align*}
   \int_{\bZ_p^\times\times T_P(\bZ_p)}(\kappa,\utau^P)\,d\mu_{\cE,\qexp}&=\left\{\begin{array}{ll}\varepsilon_{\qexp}\left(\cE_{\kappa,\utau^P}\right) &\text{if }(\kappa,\utau^P) \text{ is admissible with }\phi\chi(-1)=(-1)^k, \\0 &\text{if }\phi(-1)\neq\kappa(-1),\end{array}\right.\\
   \int_{T_P(\bZ_p)}\utau^P\,d\mu_{\cE^{\mr{imp}},\qexp}&=\left\{\begin{array}{ll}\varepsilon_{\qexp}\left(\cE_{\utau^P}\right) &\text{if }\utau^P \text{ is admissible with }\phi\epsilon^P_d(-1)=(-1)^{t^P_d}, \\0 &\text{if }\phi(-1)\neq\uptau^P_d(-1),\end{array}\right.
\end{align*}
(see \cite[\S5.1, 5.2]{LiuSLF} for more details on how to obtain these $p$-adic measures from the formulae for $\varepsilon_{\qexp}\left(\cE_{\kappa,\utau^P}\right)$, $\varepsilon_{\qexp}\left(\cE_{\utau^P}\right)$ in \S\ref{sec:placep}).

\begin{rem}
The assumption $\phi^2\neq\mr{triv}$ is not essential. Without it,  due to the pole of the Kubota--Leopoldt $p$-adic $L$-function, we need to make some modification accordingly to allow a possible pole in the constructed measures.
\end{rem}

Let $V^{SP,r,\Delta}$ be the subspace of $\varprojlim\limits_m\varinjlim\limits_l V^{SP,r}_{m,l}\otimes_{\cO_F/p^m} V^{SP,r}_{m,l}$ consisting of elements annihilated by $\gamma\otimes 1-1\otimes \gamma$ for all $\gamma\in P(\bZ_p)$. By definition and \eqref{eq:cEV}, we know that $\cE_{\kappa,\utau^P}\in V^{SP,0,\Delta}$ and $\cE_{\utau^P}\in V^{SP,n_d,\Delta}$. Then due to the Zariski density of the admissible points $(\kappa,\utau^P)$ (resp. $\utau^P$) inside $\bZ^\times_p\times T_P(\bZ_p)$ (resp. $T_P(\bZ_p)$), the measure $\mu_{\cE,\qexp}$ (resp. $\mu_{\cE^{\mr{imp}},\qexp}$) lies inside the image of the following embedding induced by $p$-adic $q$-expansion
\begin{align*}
   \mM eas\left(\bZ^\times_p\times T_P(\bZ_p),V^{SP,0,\Delta}\right)&\lhra \mM eas\left(\bZ^\times_p\times T_P(\bZ_p),\cO_F\llbracket N^{-1}\Sym(n,\bZ)^{*\oplus 2}_{>0}\rrbracket\right)\\
   (\text{resp. } \mM eas\left(T_P(\bZ_p),V^{SP,n_d,\Delta}\right)&\lhra \mM eas\left(T_P(\bZ_p),\cO_F\llbracket N^{-1}\Sym(n,\bZ)^{*\oplus 2}_{\geq 0}\rrbracket\right)).
\end{align*}
Viewing  $\mu_{\cE,\qexp}$ (resp. $\mu_{\cE^{\mr{imp}},\qexp}$) as a $p$-adic measure valued in $V^{SP,0,\Delta}$ (resp. $V^{SP,n_d,\Delta}$) , Propositions \ref{prop:key} and \ref{prop:bconv} show that one can apply $e_P\times e_P$ to it and get
\begin{equation*}
   \mu_{\cE,\Pord}\in \mM eas\left(\bZ^\times_p\times T_P(\bZ_p),e_PV^{SP,0,\Delta}\right)\quad(\text{resp. }\mu_{\cE^{\mr{imp}},\Pord}\in \mM eas\left(T_P(\bZ_p),e_PV^{SP,n_d,\Delta}\right)).
\end{equation*}
For $\unu\in\Hom_{\cont}\left(T_P(\bZ/p),\mu_{p-1}\right)$ and an $\cO_F\llbracket T_P(\bZ_p)\rrbracket$-module, we use a subscript $_{\unu}$ to denote its $\unu$-isotypic part for the action of $T_P(\bZ/p)$. Then like \cite[(6.1.8)]{LiuSLF},  for all $0\leq r\leq n_d$ and $\utau^P$ such that $\utau^P|_{T_P(\bZ/p)}=\unu$, we have the commutative diagram 
$$\xymatrix{
   \mM eas\left(T_P(\bZ_p),e_PV^{SP,r,\Delta}\right)^\natural_{\unu}\ar[r]^{\Phi^\Delta_{\unu}} \ar[dr]_{\mu\mapsto \int_{T_P(\bZ_p)}\utau^P\,d\mu\quad\quad}&\cM^r_{\Pord,\unu}\otimes_{\Lambda_P}\cM^r_{\Pord,\unu}\ar[d]^{\mod \cP_{\utau^P}}\\
   &\varprojlim\limits_m\varinjlim\limits_l e_PV^{SP,r}_{m,l}[\utau^P]\otimes_{\cO_F}e_PV^{SP,r}_{m,l}[\utau^P],
}$$
where $\mM eas\left(T_P(\bZ_p),e_PV^{SP,r,\Delta}\right)^\natural$ is the subspace of  $\mM eas\left(T_P(\bZ_p),e_PV^{SP,r,\Delta}\right)$ consisting of measures $\mu$ satisfying
\begin{equation*}
   \int_{T_P(\bZ_p)}\utau^P\,d\mu\in e_PV^{SP,r,\Delta}[\utau^P],
\end{equation*} 
where $e_PV^{SP,r,\Delta}[\utau^P]$ denotes the subspace where $T_P(\bZ_p)$ acts via $\utau^P$.

Set
\begin{align*}
   \Phi^\Delta:=\bigoplus_{\unu}\Phi^\Delta_{\unu}:&\mM eas\left(T_P(\bZ_p),e_PV^{SP,r,\Delta}\right)^\natural=\bigoplus_{\unu}\mM eas\left(T_P(\bZ_p),e_PV^{SP,r,\Delta}\right)^\natural_{\unu}\\
   &\quad\lra\cM^r_{\Pord}\otimes_{\cO_F\llbracket T_P(\bZ_p)\rrbracket}\cM^r_{\Pord}=\bigoplus_{\unu}\cM^r_{\Pord,\unu}\otimes_{\Lambda_P}\cM^r_{\Pord,\unu},
\end{align*}
where $\unu$ runs over $\Hom_{\cont}(T_P(\bZ/p),\mu_{p-1})$. This $\Lambda_P$-module morphism $\Phi^\Delta$ also induces
\begin{equation*}
   \Phi^{\Delta}:\mM eas\left(\bZ^\times_p\times T_P(\bZ_p),e_PV^{SP,r,\Delta}\right)^\natural\lra \mM eas\left(\bZ^\times_p,\cM^r_{\Pord}\otimes_{\cO_F\llbracket T_P(\bZ_p)\rrbracket}\cM^r_{\Pord}\right).
\end{equation*}
As checked before Remark \ref{rem:biggersupport}, the measure $\mu_{\cE,\Pord}$ (resp. $\mu_{\cE^{\mr{imp}},\Pord}$) belongs to the domain of $\Phi^{\Delta}$  with $r=0$ (resp. $r=n_d$). Hence we can define (with a small abuse of notation)
\begin{align*}
   \mu_{\cE,\Pord}&=\Phi^\Delta(\mu_{\cE,\Pord})\in \mM eas\left(\bZ^\times_p,\cM^0_{\Pord}\otimes_{\cO_F\llbracket T_P(\bZ_p)\rrbracket}\cM^0_{\Pord}\right),\\
   \cE^{\mr{imp}}_{\Pord}&=\Phi^\Delta(\mu_{\cE^{\mr{imp}},\Pord})\in \cM^{n_d}_{\Pord}\otimes_{\cO_F\llbracket T_P(\bZ_p)\rrbracket}\cM^{n_d}_{\Pord}.
\end{align*}
Let $\cM^{n_d,\mr{cl}}_{\Pord}$ denote the sub-$\Lambda_P$-module of $\cM^{n_d}_{\Pord}$ consisting of $P$-ordinary families whose specializations are classical Siegel modular forms at $\ut^P$ with $t^P_1\gg t^P_2\gg\cdots\gg t^P_d\gg 0$. It follows from our construction that the $P$-ordinary family $\cE^{\mr{imp}}_{\Pord}$ is contained in $\cM^{n_d,\mr{cl}}_{\Pord}\otimes_{\cO_F\llbracket T_P(\bZ_p)\rrbracket}\cM^{n_d,\mr{cl}}_{\Pord}$

\subsection{\texorpdfstring{The $p$-adic $L$-functions and their interpolation properties}{The p-adic L-functions and their interpolation properties}}\label{sec:pthm}
Let $\bT^{0,N}_{\Pord}$ (resp. $\bT^{n_d,N}_{\Pord}$) be the $\cO_F\llbracket T_P(\bZ_p)\rrbracket$-algebra as defined at the end of \S\ref{sec:HF}. Let $\sC_P$ be a geometrically irreducible component of $\mr{Spec}\left(\bT^{0,N}_{\Pord}\right)$. Because of the natural quotient map $\bT^{n_d,N}_{\Pord}\ra\bT^{0,N}_{\Pord}$, $\sC_P$ can also be viewed as a geometrically irreducible component of $\mr{Spec}\left(\bT^{n_d,N}_{\Pord}\right)$. Denote by $F_{\sC_P}$ the function field of $\sC_P$ and by $\bI_{\sC_P}$ the integral closure of $\Lambda_P$ inside $F_{\sC_P}$. Let $\lambda_{\sC_P}:\bT^{n_d,N}_{\Pord}\ra \bI_{\sC_P}$ be the homomorphism of $\Lambda_P$-algebra associated to $\sC_P$.

As $F_{\sC_P}$-algebras, we have the decomposition
\begin{equation*}
   \bT^{n_d,N}_{\Pord}\otimes_{\Lambda_P} F_{\sC_P}=F_{\sC_P}\oplus R_{\sC_P}
\end{equation*}
such that the projection onto $F_{\sC_P}$ is given by $\lambda_{\sC_P}$. Define $\mathds{1}_{\sC_P}$ as the element in $\bT^{n_d,N}_{\Pord}\otimes_{\Lambda_P} F_{\sC_P}$ which corresponds to $(\mr{id},0)$ in $F_{\sC_P}\oplus R_{\sC_P}$.

\begin{prop}
\begin{equation*}
   \mathds{1}_{\sC_P}\left(\cM^{n_d,\mr{cl}}_{\Pord}\right)\subset \cM^0_{\Pord}\otimes_{\cO_F\llbracket T_P(\bZ_p)\rrbracket} F_{\sC_P}.
\end{equation*}   
\end{prop}
\begin{proof}
The cuspidality condition is about the vanishing of the restriction to the boundary, therefore an element in $\cM^{n_d}_{\Pord}$ is cuspidal as long as its specializations at a Zariski dense subset in $\Hom_{\cont}\left(T_P(\bZ_p),\ol{\bQ}_p^\times\right)$ are cuspidal. For all $t^P_1\gg t^P_2\gg\dots\gg t^P_d\gg 0$, the specialization of $\cM^{n_d,\mr{cl}}_{\Pord}$ at the algebraic point $\ut^P\in\Hom_{\cont}\left(T_P(\bZ_p),\ol{\bQ}_p^\times\right)$ consists of classical holomorphic Siegel modular forms of weight $\imath(\ut^P)$ and level $\Gamma(N)\cap \Gamma_{P}(p)$. We reduce to show that for all $t^P_1\geq t^P_2\geq\dots\geq t^P_d\gg 0$, if a Hecke eigenform $\varphi$ (for all unramified Hecke operators away from $Np\infty$) in $\mr M^{n_d}_{\imath(\ut^P)}\left(\Gamma(N)\cap\Gamma_P(p);\bC\right)$ shares the same eigenvalues as a cuspidal Hecke eigenform in $\mr M^0_{\imath(\ut^P)}\left(\Gamma(N)\cap\Gamma_P(p);\bC\right)$, then $\varphi$ is cuspidal. When $t^P_d> 2n+1$ this is true according to \cite[Theorem 2.5.6]{HaEC}. 
\end{proof}

As explained in \cite[\S 6.1.5]{LiuSLF}, for each pair $(\beta_1,\beta_2)\in N^{-1}\Sym(n,\bZ)^{*\oplus 2}_{>0}$ there is a $\Lambda_P$-linear map
\begin{equation*}
   \varepsilon_{\qexp,\beta_1,\beta_2}:\cM^{0}_{\Pord}\otimes_{\cO_F\llbracket T_P(\bZ_p)\rrbracket}\cM^{0}_{\Pord}\lra \Lambda_P,
\end{equation*}
which, when specialized at primes $\cP_{\utau^P}$, gives the map of taking the coefficients indexed by $\beta_1,\beta_2$ in the $p$-adic $q$-expansion $\varepsilon_{\qexp}$.

Define
\begin{align*}
   \mu_{\sC_P,\phi,\beta_1,\beta_2}&= \varepsilon_{\qexp,\beta_1,\beta_2}\circ\mathds{1}_{\sC_P}\left(\mu_{\cE,\Pord}\right)\in\mM eas\left(\bZ_p^\times, F_{\sC_P}\right),\\
   \cL^{\mr{imp}}_{\sC_P,\phi,\beta_1,\beta_2}&=\varepsilon_{\qexp,\beta_1,\beta_2}\circ\mathds{1}_{\sC_P}\left(\cE^{\mr{imp}}_{\Pord}\right)\in F_{\sC_P}.
\end{align*}

\begin{thm}\label{thm:twopadicLfun}
For a Dirichlet character $\phi$ with conductor dividing $N$ and $\phi^2\neq 1$, a geometrically irreducible component $\sC_P\subset \mr{Spec}\left(\bT^{0,N}_{\Pord}\right)$, and $(\beta_1,\beta_2)\in N^{-1}\Sym(n,\bZ)^{*\oplus 2}_{>0}$, the above constructed $\mu_{\sC_P,\phi,\beta_1,\beta_2}$ and $\cL^{\imp}_{\sC_P,\phi,\beta_1,\beta_2}$ satisfy the following interpolation properties. Let $x:\bI_{\sC_P}\ra F'$ be an $F'$-point of $\sC_P$ with $F'$ being a finite extension of $F$. Suppose that the weight projection map $\Lambda_P\ra \bT^{n_d,N}_{\Pord}$ is \'{e}tale at $x$ mapping $x$ to an arithmetic point $\utau^P\in\Hom_{\cont}\left(T_P(\bZ_p),F^{\prime\times}\right)$.

If $(\kappa,\utau^P)$ is admissible, and the finite part of $\kappa$ is $\chi$, then
\begin{align*}
   \left(\int_{\bZ^\times_p}\kappa\,d\mu_{\sC_P,\phi,\beta_1,\beta_2}\right)(x)=&\,C_{x,\beta_1,\beta_2,\phi,N}\cdot E_p(n+1-k,\pi_x\times\phi\chi)\,E_\infty(n+1-k,\pi_x\times\phi\chi)\\
   &\times L^{Np\infty}(n+1-k,\pi_x\times\phi\chi),
\end{align*}
if $\kappa(-1)=\phi(-1)$ and $x$ is classical, and otherwise vanishes.

Assume that the parity of $\sC_P$ is compatible with $\phi$ and that $\utau^P$ is admissible. If $x$ is classical, then
\begin{align*}
   \cL^{\mr{imp}}_{\sC_P,\phi,\beta_1,\beta_2}(x)=\,&C_{x,\beta_1,\beta_2,\phi,N}\cdot E^{P\textnormal{-imp}}_p(n+1-t^P_d,\pi_x\times\phi\epsilon^P_d)\,E_\infty(n+1-t^P_d,\pi_x\times\phi\epsilon^P_d)\\
   &\times  L^{Np\infty}(n+1-t^P_d,\pi_x\times\phi\epsilon^P_d),
\end{align*}
and if $x$ is not classical, then we have $0=0$.

The constant $C_{x,\beta_1,\beta_2,\phi,N}$ is defined as
\begin{align*}
   C_{x,\beta_1,\beta_2,\phi,N}=&\,\phi(-1)^n\mr{vol}\left(\wh{\Gamma}(N)\right)\frac{\prod_{l=1}^d\prod_{j=1}^{n_l}(1-p^{-j})}{\prod_{j=1}^n(1-p^{-2j})}\\
   &\,\sqrt{-1}^{\frac{n^2+n}{2}}2^{\frac{n^2-n}{2}}\frac{2^{-\sum_{j=1}^d n_jt^P_{x,j}}}{\dim\left(\GL(n),\imath(\ut^P_x)\right)}\,\sum_{\varphi\in \fs_x}\frac{\fc(\beta_1,\varphi)\fc(\beta_2,e_P\pW(\varphi))}{\left<\varphi,\ol{\varphi}\right>},
\end{align*}
where
\begin{itemize}
\item $\fs_x$ is a finite set consisting of an orthonormal basis of the eigenspace for the Hecke eigensystem parametrized by $x$ inside $e_P \mr M^0_{\imath(\ut^P)}\left(\Gamma(N)\cap\Gamma_{SP}(p^\infty),\uep;\bC\right)$. The set $\fs_x$ is empty and the evaluation is $0$ if $x$ is not classical, \textit{i.e.} if there exists no cuspidal irreducible automorphic representation $\pi_x\subset\cA_0(G(\bQ)\backslash G(\bA))$ with $\pi_{x,\infty}\cong \cD_{\imath(\ut^P)}$ such that the Hecke eigensystem parametrized by $x$ appears in $\pi_x$.
\item $\fc(\beta_i,\cdot)$, $i=1,2$, denotes the coefficient indexed by $\beta_i$ in the $q$-expansion.
\item The operator $\pW$ is defined as in \eqref{eq:pW}, and by the reasoning as in \cite[Proposition 5.7.2]{LiuSLF} we know that $e_P\pW(\varphi)\neq 0$ for $\varphi\in\fs_x$.
\end{itemize}
\end{thm}
\begin{proof}
By our construction, for a classical $x$ as in the theorem, the evaluations of $\mathds{1}_{\sC_P}\left(\int_{\bZ^\times_p}\kappa\,d\mu_{\cE,\Pord}\right)$ and $\mathds{1}_{\sC_P}\left(\cE^{\rm{imp}}_{\Pord}\right)$ at $x$ are classical cuspidal Siegel modular forms obtained by projecting the forms in \eqref{eq:norE} to the eigenspace associated to the Hecke eigensystem parametrized by $x$. Thus the interpolation formulae follow from the formulae for
\begin{align*}
   \numberthis\label{eq:tar}&\sL_{\ol{\varphi}}\left((e_P\times e_P)\left.E^*(\cdot,f_{\kappa,\utau^P})\right|_{G\times G}\right)& \text{and}& &\sL_{\ol{\varphi}}\left((e_P\times e_P)\left.E^*(\cdot,f_{\utau^P})\right|_{G\times G}\right),
\end{align*}
for $\varphi\in\fs_x$, which are proved later in Proposition \ref{thm:itp}.
\end{proof}
\begin{rem}
For each $j\in(\bZ/(p-1)$ such that $\phi\omega^j(-1)=1$, applying the $p$-adic Mellin transform with respect to $\omega^j$ to the measure $\mu_{\sC_P,\phi,\beta_1,\beta_2}$, one gets the $p$-adic $L$-function $\cL_{\sC_P,\phi\omega^j,\beta_1,\beta_2}\in\bI_{\sC_P}[[S]]\otimes_{\bI_{\sC_P}}F_{\sC_P}$ as described in the theorem in the introduction.
\end{rem}

Before finishing the calculations that prove the theorem, let us recall the state of the art concerning \'etaleness of eigenvarieties. First of all, the \'etale points are an open subset of $\mr{Spec}(\sC_P)$. In general for a given point we don't have good criterion for \'etalness (except for $n=1$). The typical strategy to prove that a point $x$ is \'etale over the weight space consists of two steps: first one shows that the system of eigenvalues corresponding to $x$ correspond to a classical holomorphic Siegel modular forms, and then one uses automorphic methods to control the dimension of the space of eigenforms corresponding to the eigensystem $x$.

In our situation, for what concerns the first step, as we consider only admisible points the condition $t_d^P \geq n+1$ ensure that the weights are cohomological and it is widely believed that all systems of eigenvalues of cohomological weights are classical, even though the best result on classicality at the moment \cite{BPS} does not cover the case we study in the next section, where $t_d^P =n+1$, as they require $t^P_d > \frac{n(n+1)}{2}$.

For the second point, one needs to prove that $\delta(x)\leq \delta(y)$, where $y$ is a point near $x$ such that  the weight projection map is {\'e}tale at $y$ and where $\delta(x)$ (resp. $\delta(y)$) is the dimension of the eigenspace corresponding to $x$ (resp. $y$) inside ordinary Siegel modular forms of a fixed tame level. One approach for proving such an equality is to show that $\delta(x)=1$. This can be done by adding $\bU_{\ell}$-operators for tame ramified places $\ell$ and assuming that they have distinct eigensystems at the point we consider, or for some cases of $n=2$ and Iwahori tame level, by using the exhaustive classification of local representations in \cite{RSLocal}. For example, the dimension one result can be proved for the automorphic representation of $\mathrm{GSp}(4,\bA)$ associated to the symmetric cube of $\pi_{f_{X_0(11)}}$, for $f_{X_0(11)}$ is the only normalized weight two elliptic cusp form of level $\Gamma_0(11)$. Another possible approach for proving the inequality is to add suitable conditions at tame ramified places to guarantee that all eigenforms for $x$ belongs to one irreducible automorphic representation, and then try to prove an analogue of \cite[Lemma 4.5]{Chfern} for symplectic groups.

\subsection{The archimedean zeta integral}\label{sec:Zinfty}
Let $E_\infty(s,\pi\times\xi)$ be the modified archimedean Euler factor for $p$-adic interpolation of critical values of $L(s,\pi\times\xi)$ to the left of the center according to the conjecture of Coates--Perrin-Riou. Unfolding the definition in \cite[\S5]{CoaMot}, if $\pi_\infty\cong \cD_{\ut}$ with $\ut=(t_1,\dots,t_n)$, then
\begin{equation}\label{eq:Euler-infty}
   E_\infty(s,\pi\times\xi)=\,\prod_{j=1}^ne^{-(s+t_j-j)\frac{\pi \sqrt{-1}}{2}} \Gamma_{\bC}(s+t_j-j)).
\end{equation}

\begin{prop}
With $f_{\kappa,\utau^P,\infty}\in I_{Q_H,\infty}(\frac{2n+1}{2}-k,\mr{sgn}^k)$ as in \eqref{eq:finfty} and $\pi_\infty\cong \cD_{\imath(\utau^P)}$,
\begin{align*}
   &\frac{\Gamma_{2n}\left(\frac{2n+1}{2}\right)}{2^{-2n^2+n+2nk}\pi^{2n^2+n}}\,\left.Z_\infty\left(f_{\kappa,\utau^P,\infty}(s),v^\vee_{\imath(\utau^P)},v_{\imath(\utau^P)}\right)\right|_{s=\frac{2n+1}{2}-k}\\
  =&\,\sqrt{-1}^{\frac{n^2}{2}+\frac{n}{2}}2^{\frac{n^2}{2}-\frac{n}{2}}\,\frac{2^{-\sum_{j=1}^d n_jt^P_j}}{\dim\left(\GL(n),\imath(\ut^P)\right)}\,E_\infty(1+n-k,\pi\times\phi\chi).
\end{align*}
\end{prop}
\begin{proof}
According to \cite{LapidRallis}, the local zeta integral for $\pi_\infty$ and $f_{\kappa,\utau^P,\infty}(s,\mr{sgn}^k)\in I_{Q_{H,\infty}}(s,\mr{sgn}^k)$ satisfies the functional equation 
\begin{equation}\label{eq:feq2}
   Z_\infty\left(f_{\kappa,\utau^P,\infty}(s),\cdot,\cdot\right)=\Gamma_\infty(s,\pi_\infty,\mr{sgn}^k)^{-1}Z_\infty\left(M(s,\mr{sgn}^k)f_{\kappa,\utau^P,\infty}(s),\cdot,\cdot\right),
\end{equation}
where $M(s,\mr{sgn}^k):I_{Q_H,\infty}(s,\mr{sgn}^k)\ra I_{Q_H,\infty}(-s,\mr{sgn}^k)$ is the intertwining operator, and
\begin{equation}\label{eq:feq3}
\begin{aligned}
   &\Gamma_\infty(s,\pi_\infty,\mr{sgn}^k)^{-1}\\
   =&\,\pi_\infty(-1)\,\gamma_\infty(s+\frac{1}{2},\pi_\infty\times\sgn^k)^{-1}\gamma_{\infty}\left(s-\frac{2n-1}{2},\sgn^k\right)\,\prod_{j=1}^n \gamma_\infty(2s-2j+2,\triv).
\end{aligned}
\end{equation}
It follows from \cite[(1.31)(4.34K)]{Sh6} and the definition of $f_{\kappa,\utau^P,\infty}$ that
\begin{align*}
   \numberthis \label{eq:feq1}&M(s,\mr{sgn}^k)f_{\kappa,\utau^P,\infty}(s,\mr{sgn}^k)\\
   =\,&(-1)^{nk} 2^{2n-2ns}\pi^{n(2n+1)}\cdot \frac{\Gamma_{2n}(s)}{\Gamma_{2n}\left(\frac{1}{2}(s+\frac{2n+1}{2})+\frac{k}{2}\right)\Gamma_{2n}\left(\frac{1}{2}(s+\frac{2n+1}{2})-\frac{k}{2}\right)}\, f_{\kappa,\utau^P\infty}(-s,\mr{sgn}^k).
\end{align*}
Combining \eqref{eq:feq1}\eqref{eq:feq2}\eqref{eq:feq3}, we get
\begin{equation}\label{eq:feq4}
\begin{aligned}
   \left.Z_\infty\left(f_{\kappa,\utau^P,\infty}(s),\cdot,\cdot\right)\right|_{s=\frac{2n+1}{2}-k}=&\,\sqrt{-1}^{k+n^2+n}2^{2nk-k-2n^2+n+1}\pi^{-(2n+1)k+2n^2+2n}\frac{\Gamma_{2n}(k)\Gamma(k-n)}{\Gamma_{2n}\left(\frac{2n+1}{2}\right)}\\
   &\times \pi_\infty(-1)\gamma_\infty(s+\frac{1}{2},\pi_\infty\times\sgn^k)\,\left.Z_\infty\left(f_{\kappa,\utau^P,\infty}(s),\cdot,\cdot\right)\right|_{s=k-\frac{2n+1}{2}}.
\end{aligned}
\end{equation}
Let $(t_1,\dots,t_n)=\imath(\ut^P)$. Easy computation shows that
\begin{align*}
   \gamma_\infty(k-n,\pi_\infty\times\sgn^k)=\frac{\sqrt{-1}^{-k+n^2}2^{k-n-1}\pi^{k-n}}{\Gamma(k-n)}\pi_\infty(-1)\prod_{j=1}^n\frac{\Gamma_{\bC}(n+1-k+t_j-j)}{\Gamma_{\bC}(k-n+t_j-j)},
\end{align*}
and \cite[Theorem 2.4.1]{LiuAZI} says that
\begin{align*}
   \left.Z_\infty\left(f_{\kappa,\utau^P,\infty}(s),v^\vee_{\imath(\ut^P)},v_{\imath(\ut^P)}\right)\right|_{s=k-\frac{2n+1}{2}}=\frac{\sqrt{-1}^{nk-\sum_{j=1}^nt_j}2^{-\sum_{j=1}^n t_j+\frac{n^2+3n}{2}}\pi^{2nk}}{\Gamma_{2n}(k)\dim\left(\GL(n),\imath(\ut^P)\right)}\,\prod_{j=1}^n\Gamma_{\bC}(k-n+t_j-j).
\end{align*}
Plugging these two formulas into \eqref{eq:feq4} proves the proposition.
\end{proof}

\subsection{\texorpdfstring{Computing the zeta integrals at $p$}{Computing the zeta integrals at p}}\label{sec:zetaintegralp}
The goal of this section is to prove the following proposition, which will give the interpolation properties for our $p$-adic $L$-functions.

\begin{prop}\label{thm:itp}
Let $\pi\subset \cA_0(G(\bQ)\backslash G(\bA))$ be an irreducible cuspidal automorphic representation with $\pi_\infty\cong \cD_{\imath(\ut^P)}$. Also, assume that $\pi$ contains a $P$-ordinary Siegel modular form $\varphi$ of weight $\imath(\ut^P)$ and $p$-nebentypus $\uep^P$, invariant under the translation of $(\Gamma(N)\cap\Gamma_{SP}(p^\infty))^\wedge\subset G(\bA_{\mr{f}})$. Then
\begin{equation}\label{eq:ktp}
\begin{aligned}
   &\chi(-1)^{n} \frac{\Gamma_{2n}\left(\frac{2n+1}{2}\right)}{2^{2nk-2n^2+n}\pi^{n+2n^2}}\,\sL_{\ol{\varphi}}\left((e_P\times e_P)\left.E^*(\cdot,f_{\kappa,\utau^P})\right|_{G\times G}\right)\\
   =&\,\phi(-1)^n\mr{vol}\left(\wh{\Gamma}(N)\right)\frac{\prod_{l=1}^d\prod_{j=1}^{n_l}(1-p^{-j})}{\prod_{j=1}^n(1-p^{-2j})}\,\sqrt{-1}^{\frac{n^2-n}{2}}2^{\frac{n^2-n}{2}}\,\frac{2^{-\sum_{j=1}^d n_jt^P_j}}{\dim\left(\GL(n),\imath(\ut^P)\right)}\\
   &\,\times E_p(n+1-k,\pi\times\phi\chi)\,E_\infty(n+1-k,\pi\times\phi\chi)\, L^{Np\infty}(n+1-k,\pi\times \phi\chi)\cdot e_P\pW(\varphi),
\end{aligned} 
\end{equation}
and 
\begin{equation}\label{eq:tp}
\begin{aligned}
   &\epsilon^P_d(-1)^{n} \frac{\Gamma_{2n}\left(\frac{2n+1}{2}\right)}{2^{2nt^P_d-2n^2+n}\pi^{n+2n^2}}\,\sL_{\ol{\varphi}}\left((e_P\times e_P)\left.E^*(\cdot,f_{\utau^P})\right|_{G\times G}\right)\\
   =&\,\phi(-1)^n\mr{vol}\left(\wh{\Gamma}(N)\right)\frac{\prod_{l=1}^d\prod_{j=1}^{n_l}(1-p^{-j})}{\prod_{j=1}^n(1-p^{-2j})}\,\sqrt{-1}^{\frac{n^2-n}{2}}2^{\frac{n^2-n}{2}}\frac{2^{-\sum_{j=1}^d n_jt^P_j}}{\dim\left(\GL(n),\imath(\ut^P)\right)}\\
   &\,\times E^{P\textnormal{-imp}}_p(n+1-t^P_d,\pi\times\phi\epsilon^{P}_d)\,E_\infty(n+1-t^P_d,\pi\times\phi\epsilon^P_d)\, L^{Np\infty}(n+1-t^P_d,\pi\times \phi\epsilon^P_d)\cdot e_P\pW(\varphi).
\end{aligned} 
\end{equation}
\end{prop}

With the results in the last section, the proof of the proposition is mainly about computing the zeta integrals at $p$. Because of the ordinary projection $e_P\times e_P$, we need to compute at $p$ the zeta integrals for $R(U^P_p)\times R(U^P_p)f_{\kappa,\utau^P}$ and $R(U^P_p)\times R(U^P_p)f_{\utau^P}$, where $R$ is a polynomial depending on $\pi_p$ and a sufficiently small open compact subgroup $K_p\subset G(\bZ_p)$ satisfying $e_P=R(U^P_p)$ on $\pi^{K_p}$. Recall that with a fixed $\ut^P$, we make $\bU^P_p$-operators act on smooth $G(\bQ_p)$-representations by \eqref{eq:AUp}.

The computation will essentially use Proposition \ref{prop:JL} and is similar as the computation in \cite[p.39-42]{LiuSLF}.

\begin{proof}
Since the projections to $\pi_p$ of $P$-ordinary forms inside $\pi$ span a one-dimensional subspace, we know that $\sL_{\ol{\varphi}}\left((e_P\times e_P)\left.E^*(\cdot,f_{\kappa,\utau^P})\right|_{G\times G}\right)$ (resp. $\sL_{\ol{\varphi}}\left((e_P\times e_P)\left.E^*(\cdot,f_{\utau^P})\right|_{G\times G}\right)$) equals $e_P\pW(\varphi)$ up to a scalar, given by 
\begin{align*}
  &\frac{ \left<\sL_{\ol{\varphi}}\left((e_P\times e_P)\left.E^*(\cdot,f_{\kappa,\utau^P})\right|_{G\times G}\right),\varphi^{\prime\mvw}\right>}{\left<e_P\pW(\varphi),\varphi^{\prime\mvw}\right>}, &(\text{resp.  }\frac{ \left<\sL_{\ol{\varphi}}\left((e_P\times e_P)\left.E^*(\cdot,f_{\utau^P})\right|_{G\times G}\right),\varphi^{\prime\mvw}\right>}{\left<e_P\pW(\varphi),\varphi^{\prime\mvw}\right>})
\end{align*}
for an arbitrary $\varphi'\in\pi$ such that the denominator is nonzero. We will always assume that $\varphi'$ is fixed by both $K_p$ and $\mvw SP_G(\bZ_p)\mvw$ and $e_P\varphi'\neq 0$. Then
\begin{equation*}
   \left<e_P\pW(\varphi),\varphi^{\prime\mvw}\right>=\left<\ol{\varphi},e_P\varphi'\right>\neq 0.
\end{equation*}
Let $\ub=(\underbrace{b_1,\dots,b_1}_{n_1},\dots,\underbrace{b_d,\dots,b_d}_{n_d})$ for $b_1>b_2>\cdots>b_d>0$ and denote $\mr{diag}(p^{b_1},\dots,p^{b_1},\dots,p^{b_d},\dots,p^{b_d})$ by $p^{\ub}$. A direct computation shows that
\begin{align*} 
   &\left<\sL_{\ol{\varphi}}\left((e_P\times U^P_{p,\ub})\left.E^*(\cdot,f_{\kappa,\utau^P})\right|_{G\times G}\right),\varphi^{\prime\mvw}\right>\\
   =&\,\phi(-1)^n\mr{vol}\left(\wh{\Gamma}(N)\right)\cdot\frac{Z_\infty\left(f_{\kappa,\utau^P,\infty},v_{\imath(\ut^P)},v^\vee_{\imath(\ut^P)}\right)}{\left<v_{\imath(\ut^P)},v^\vee_{\imath(\ut^P)}\right>}\cdot L^{Np\infty}(n+1-k,\pi\times\phi\chi)\\
   &\,\left.\times p^{\left<\ut+2\rho_{G,c},\,\ub\right>}\left<\ol{\varphi},R(U^P_p) \int_{G(\bQ_p)}f^{\alpha_{\kappa,\utau^P}}_p(s,\phi\chi)\left(\cS^{-1}\iota(g^{-1},1)\right)\pi_p(gp^{\ub})\varphi'\,dg\right>\right|_{s=\frac{2n+1}{2}-k}.
\end{align*}
The ratio
\begin{equation}\label{eq:ratio}
    \frac{\left<\ol{\varphi},R(U^P_p) \int_{G(\bQ_p)}f^{\alpha_{\kappa,\utau^P}}_p(s,\phi\chi)\left(\cS^{-1}\iota(g^{-1},1)\right)\pi_p(gp^{\ub})\varphi'\,dg\right>}{\left<\ol{\varphi},e_P\varphi'\right>}
\end{equation}
is independent of the choice of $\varphi$. Let $\theta_1,\dots,\theta_n$, $\alpha_1,\dots,\alpha_n$ be as in Proposition \ref{prop:JL}, and $\sigma_1,\dots,\sigma_d$ be as in the definition \eqref{eq:Ep}. We know that the $P$-ordinary space (with respect to $\ut^P$) inside $\mr{Ind}^{G(\bQ_p)}_{B_G(\bQ_p)}\utheta$ is one dimensional. Therefore, if we take from $\mr{Ind}^{G(\bQ_p)}_{B_G(\bQ_p)}\utheta$ a function $\pG:G(\bQ_p)\ra \bC$ invariant under the right translation by $K_p$ and $\mvw SP_G(\bZ_p)\mvw$, then
\begin{equation*}
   \eqref{eq:ratio}=\frac{\left(R(U^P_p) \int_{G(\bQ_p)}f^{\alpha_{\kappa,\utau^P}}_p(s,\phi\chi)\left(\cS^{-1}\iota(g^{-1},1)\right)\pi_p(gp^{\ub})\pG\,dg\right)(1)}{\left(R(U^P_p)\pG\right)(1)},
\end{equation*}
as long as the denominator is nonzero. Now we further assume that $\pG(1)=1$. Then by our description of the $\bU^P_p$-eigenvalues for the $P$-ordinary vector in terms of the $\alpha_i$'s in Proposition \ref{prop:JL}, we have $\left(R(U^P_p)\pG\right)(1)=1$. 

Now let $\pG\in\mr{Ind}^{G(\bQ_p)}_{B_G(\bQ_p)}\utheta$ be the smooth function on $G(\bQ_p)$ supported on $B_G(\bQ_p)\mvw SP_G(\bZ_p)\mvw$ and taking the value $1$ on $\mvw SP_G(\bZ_p)\mvw$. We also put $\fw(a)=|\det a|^{-\frac{n+1}{2}}_p\pG\left(\begin{pmatrix}a&0\\0&\ltrans{a}^{-1}\end{pmatrix}\right)$. Then $\fw\in \mr{Ind}^{\GL(n,\bQ_p)}_{B(\bQ_p)}\utheta$ and is invariant under the right translation of $u$ for  $\ltrans{u}\in SP(\bZ_p)$ and takes the value $1$ at identity. Let $\Phi_{\chi,\uep^P}$ be the Schwartz function on $M_n(\bQ_p)$ whose Fourier transform is  
\begin{align*}
   \wh{\Phi}_{\chi,\uep^P}(\beta_0)&=\prod_{i=1}^{d}\mathds{1}_{\GL(N_i,\bZ_p)}\left((\beta_0)_{N_i}\right)\times\,\prod_{i=1}^{d-1}\epsilon^P_i\epsilon^{P-1}_{i+1}(\mr{det}(-(\beta_0)_{N_i}))\cdot \epsilon^{P}_d\chi^{-1}(\mr{det}(-\beta_0)),&\beta_0\in M_n(\bQ_p).
\end{align*}
An easy computation shows that 
\begin{align*}
   f^{\alpha_{\kappa,\utau^P}}_p(s,\phi_p\chi_p)\left(\cS^{-1}\iota\left(\begin{psm}a&b\\c&d\end{psm},1\right)\right)=&|\det a|_p^{-(s+\frac{2n+1}{2})}\phi\chi(\det(-a))^{-1}p^{-n(n+1)}\\
   &\,\times\mathds{1}_{\GL(n,\bQ_p)}(a)\cdot\mathds{1}_{p^{-2}\Sym(n,\bZ_p)}(a^{-1}b)\mathds{1}_{\Sym(n,\bZ_p)}(ca^{-1})\cdot \Phi_{\chi,\uep^P}(a^{-1}),
\end{align*}
and
\begin{align*}
   &\left(U^P_{p,\uc} \int_{G(\bQ_p)}f^{\alpha_{\kappa,\utau^P}}_p(s,\phi_p\chi_p)\left(\cS^{-1}\iota(g^{-1},1)\right)\pi_p(gp^{\ub})\pG\,dg\right)(1)\\
   =&\,\frac{\prod_{j=1}^n(1-p^{-j})}{\prod_{j=1}^n(1-p^{-2j})}p^{-\left<\rho_{G,n},\,\ub\right>}\prod_{i=1}^d \mf{a}^{c_i}_{N_i}\int_{\GL(n,\bQ_p)}|\det a|_p^{s+\frac{n}{2}}\phi\chi(\det(-a))\Phi_{\chi,\uep^P}(a)\fw\left(ap^{\ub}\right)\,da,
\end{align*}
for $\uc=(\underbrace{c_1,\dots,c_1}_{n_1},\dots,\underbrace{c_d,\dots,c_d}_{n_d})$, $c_1>c_2>\cdots>c_d>0$.

Therefore, in order to show \eqref{eq:ktp}, it suffices to show that with our chosen $\fw\in\mr{Ind}^{\GL(n,\bQ_p)}_{B(\bQ_p)}\utheta$, 
\begin{equation}\label{eq:fint1}
\begin{aligned}
   \int_{\GL(n,\bQ_p)}|\det a|_p^{s+\frac{n}{2}}&\phi_p\chi_p(\det a)\Phi_{\chi,\uep^P}(a)\fw\left(ap^{\ub}\right)\,da\\
   &=\frac{\prod_{i=1}^d\prod_{j=1}^{n_i}(1-p^{-j})}{\prod_{j=1}^n(1-p^{-j})} \prod_{i=1}^d \prod_{j=N_{i-1}+1}^{N_i}\alpha_j^{b_i}\cdot E_p(s,\pi\times\phi\chi).
\end{aligned}
\end{equation}
By the same reasoning, in order to show \eqref{eq:tp}, it suffices to show 
\begin{equation}\label{eq:fint2}
\begin{aligned}
   \int_{\GL(n,\bQ_p)}|\det a|_p^{s+\frac{n}{2}}&\phi_p\epsilon^P_{d,p}(\det a)\Phi_{\uep^P}(a)\fw\left(ap^{\ub}\right)\,da\\
   &=\frac{\prod_{i=1}^d\prod_{j=1}^{n_l}(1-p^{-j})}{\prod_{j=1}^n(1-p^{-j})}  \prod_{i=1}^d \prod_{j=N_{i-1}+1}^{N_i}\alpha_j^{b_i} \cdot E^{P\text{-imp}}_p(s,\pi\times\phi\epsilon^P_d),
\end{aligned}
\end{equation}
where $\Phi_{\utau^P}$ is the Schwartz function on $M_n(\bQ_p)$ whose Fourier transform is
\begin{align*}
   \wh{\Phi}_{\uep^P}(\beta_0)&=\mathds{1}_{M_n(\bZ_p)}(\beta_0)\prod_{i=1}^{d-1}\mathds{1}_{\GL(N_i,\bZ_p)}\left((2\beta_0)_{N_i}\right)\times\,\prod_{i=1}^{d-1}\epsilon^P_i\epsilon^{P-1}_{i+1}(\mr{det}(-(2\beta_0)_{N_i})),&\beta_0\in M_n(\bQ_p).
\end{align*}

We will show \eqref{eq:fint1} and \eqref{eq:fint2} by induction. Write $n'=N_{d-1}$, $a=\begin{blockarray}{ccc}&n'&n_d\\\begin{block}{c(cc)}n'&a'&\eta\\n_d&0&\lambda\\ \end{block}\end{blockarray}\cdot\begin{pmatrix}I_{n'}&0\\\lambda^{-1}\ltrans{\mu}&I_{n_d}\end{pmatrix}$. Define $\fw'\in\mr{Ind}_{B_{n'}(\bQ_p)}^{\GL(n',\bQ_p)}(\theta_1,\dots,\theta_{n'})$ by $
   \fw'(a')=|\det a'|_p^{-\frac{n_d}{2}}\fw\left(\begin{pmatrix}a'&0\\0&I_{n_d}\end{pmatrix}\right)$, and the $\SL(n_d,\bZ_p)$-fixed $\fw_d\in\mr{Ind}^{\GL(n_d,\bQ_p)}_{B_{n_d}(\bQ_p)}(\theta_{n'+1},\dots,\theta_n)$ by $\fw_d(\lambda)=|\det a|_p^{\frac{n'}{2}}\fw\left(\begin{pmatrix}I_{n'}&0\\0&\lambda\end{pmatrix}\right)$. Also, let
\begin{align*}
   \wh{\Phi}'_{\chi,\uep^P}(\beta'_0)&=\prod_{i=1}^{d-1}\mathds{1}_{\GL(N_i,\bZ_p)}\left((\beta'_0)_{N_i}\right)\times\,\prod_{i=1}^{d-2}\epsilon^P_i\epsilon^{P-1}_{i+1}(\mr{det}(-(\beta'_0)_{N_i}))\cdot \epsilon^{P}_{d-1}\chi^{-1}(\mr{det}(-\beta'_0)),&\beta'_0\in M_{n'}(\bQ_p),\\
   \wh{\Phi}'_{\uep^P}(\beta'_0)&=\prod_{i=1}^{d-1}\mathds{1}_{\GL(N_i,\bZ_p)}\left((\beta'_0)_{N_i}\right)\times\,\prod_{i=1}^{d-2}\epsilon^P_i\epsilon^{P-1}_{i+1}(\mr{det}(-(\beta'_0)_{N_i}))\cdot \epsilon^{P}_{d-1}\epsilon^{P-1}_d(\mr{det}(-\beta'_0)),&\beta'_0\in M_{n'}(\bQ_p),
\end{align*}
and let $\Phi'_{\chi,\uep^P}$, $\Phi'_{\uep^P}$ be the inverse Fourier transform of $\wh{\Phi}'_{\chi,\uep^P}$, $\wh{\Phi}'_{\uep^P}$. Denote by $\cF_{\uep^P_d\chi^{-1}}$ the inverse Fourier transform of the Schwartz function $\lambda\mapsto\mathds{1}_{\GL(n_d,\bZ_p)}(\lambda)\cdot\uep^P_d\chi^{-1}(\det(-\lambda))$ on $M_{n_d}(\bQ_p)$. Then by a routine computation we get
\begin{align*}
   \numberthis\label{eq:fint3}\text{LHS of } \eqref{eq:feq1}=&\frac{\prod_{j=1}^{n'}(1-p^{-j})\prod_{j=1}^{n_d}(1-p^{-j})}{\prod_{j=1}^n(1-p^{-j})}\int_{\GL(n',\bQ_p)}|\det a'|_p^{s+\frac{n'}{2}}\phi_p\chi_p(\det a')\fw'\left(a'p^{\ub^\prime}\right)\Phi'_{\chi,\uep^P}(a')\,da'\\
      &\,\times\prod_{i=N_{d-1}+1}^n\alpha_i^{b_d}\int_{\GL(n_d,\bQ_p)}|\det\lambda|_p^{s+\frac{n_d}{2}}\phi_p\chi_p(\det \lambda) \fw_d(\lambda)\cF_{\epsilon^P_d\chi^{-1}}(\lambda)\,d\lambda,\\
    \numberthis\label{eq:fint4}  \text{LHS of } \eqref{eq:feq2}=&\frac{\prod_{j=1}^{n'}(1-p^{-j})\prod_{j=1}^{n_d}(1-p^{-j})}{\prod_{j=1}^n(1-p^{-j})}\int_{\GL(n',\bQ_p)}|\det a'|_p^{s+\frac{n'}{2}}\phi_p\chi_p(\det a')\fw'\left(a'p^{\ub^\prime}\right)\Phi'_{\uep^P}(a')\,da'\\
      &\,\times\prod_{i=N_{d-1}+1}^n\alpha_i^{b_d}\int_{\GL(n_d,\bQ_p)}|\det\lambda|_p^{s+\frac{n_d}{2}}\phi_p\chi_p(\det \lambda) \fw_d(\lambda) \mathds{1}_{M_{n_d}(\bZ_p)}(\lambda)\,d\lambda.
\end{align*}

For $\cF_{\epsilon^P_d\chi^{-1}}$, we have the following formula \cite[Proposition 6.1, Appendix of \S6]{BS}. If $\mr{cond}(\epsilon^P_d\chi^{-1})=p^c>1$, 
\begin{align*}
   \numberthis\label{eq:BScF1}\cF_{\epsilon^P_d\chi^{-1}}(\lambda)&=p^{-\frac{n_d(n_d+1)c}{2}}G(\epsilon^P_d\chi^{-1})^{n_d}\cdot \mathds{1}_{\GL(n_d,\bZ_p)}\left(p^{c}\lambda\right)\epsilon^{P-1}_d\chi(\det(p^c\lambda)),
\end{align*}
and if $\epsilon^P_d=\chi$,
\begin{align*}
  \numberthis\label{eq:BScF2}\cF_{\mr{triv}}(\lambda)=\sum_{j=0}^{n_d}(-1)^jp^{\frac{j(j-1-n_d)}{2}}\int_{\GL(n_d,\bZ_p)\begin{psm}pI_j&0\\0&I_{n_d-j}\end{psm}\GL(n_d,\bZ_p)}\mathds{1}_{M_{n_d}(\bZ_p)}(\lambda g)\,dg.
\end{align*}
We first look at the easier case \eqref{eq:fint4}, where for the integral in the second line of it, since $\fw_d$ is fixed by $\SL(n_d,\bZ_p)$ and $\GL(n_d,\bZ_p)$ acts on it by $\epsilon^{P-1}_{d,p}\circ\det$, we have
\begin{equation}\label{eq:d2}
   \int_{\GL(n_d,\bQ_p)}|\det\lambda|_p^{s+\frac{n_d}{2}}\phi_p\epsilon^{P}_{d,p}(\det \lambda) \fw_d(\lambda) \mathds{1}_{M_{n_d}(\bZ_p)}(\lambda)\,d\lambda=L_p\left(s+\frac{1}{2},\sigma_d\otimes\phi_p\right).
\end{equation}
Next we treat the integral in the second line of \eqref{eq:fint3} with $\epsilon^P\chi^{-1}\neq \mr{triv}$. Plugging in \eqref{eq:BScF1}, we get
\begin{equation}\label{eq:d12}
\begin{aligned}
   &\int_{\GL(n_d,\bQ_p)}|\det\lambda|_p^{s+\frac{n_d}{2}}\chi_p(\det \lambda) \fw_d(\lambda)\cF_{\epsilon^P_d\chi^{-1}}(\lambda)\,d\lambda=\prod_{j=n'+1}^n G(\epsilon^P_d\chi^{-1})\left(\phi(p)^{-1}\alpha^{-1}_jp^{s-\frac{1}{2}}\right)^{c}\\
   =&\,\prod_{j=n'+1}^n\gamma_p\left(\frac{1}{2}-s,\phi^{-1}_p\chi^{-1}_p\theta^{-1}_j\right)=\gamma_p\left(\frac{1}{2}-s,\wt{\sigma}_d\otimes\phi^{-1}_p\chi^{-1}_p\epsilon^P_{d,p}\right).
\end{aligned}
\end{equation}
Lastly, we consider the integral in the second line of \eqref{eq:fint3} with $\epsilon^P\chi^{-1}= \mr{triv}$. The formula for the Hecke action on spherical representations of $\GL(\bQ_p)$ gives
\begin{equation*}
   \int_{\GL(n_d,\bZ_p)\begin{psm}p^{-1}I_j&0\\0&I_{n_d-j}\end{psm}\GL(n_d,\bZ_p)}\fm_d(\lambda g)\,dg=p^{\frac{j(n_d-j)}{2}}\sum_{\substack{\iota:\{1,\dots,j\}\hra\{1,\dots,n_d\}\\\iota(1)<\dots<\iota(j)}}\alpha^{-1}_{n'+\iota(1)}\dots\alpha^{-1}_{n'+\iota(j)}\fm_d(\lambda).
\end{equation*}
Combining this with \eqref{eq:BScF2}, we get
\begin{equation}\label{eq:d11}
\begin{aligned}
   &\int_{\GL(n_d,\bQ_p)}|\det\lambda|_p^{s+\frac{n_d}{2}}\chi_p(\det \lambda) \fw_d(\lambda)\cF_{\mr{triv}}(\lambda)\,d\lambda\\
  =&\,\prod_{j=n'+1}^n(1-\phi_p(p)^{-1}\alpha^{-1}_jp^{s-\frac{1}{2}})\int_{\GL(n_d,\bQ_p)}|\det\lambda|_p^{s+\frac{n_d}{2}}\phi_p\epsilon^{P}_{d,p}(\det \lambda) \fw_d(\lambda) \mathds{1}_{M_{n_d}(\bZ_p)}(\lambda)\,d\lambda\\
  =&\prod_{j=n'+1}^n(1-\phi_p(p)^{-1}\alpha^{-1}_jp^{s-\frac{1}{2}})\cdot L_p\left(s+\frac{1}{2},\sigma\otimes\phi_p\right)=\gamma_p\left(\frac{1}{2}-s,\wt{\sigma}_d\otimes\phi^{-1}_p\right).
\end{aligned}
\end{equation}
The identities \eqref{eq:d2}, \eqref{eq:d12}, and \eqref{eq:d11} allows us to compute \eqref{eq:fint1} and \eqref{eq:fint2} by induction, and this finishes the proof.
\end{proof}

\section{\texorpdfstring{The derivative of the $p$-adic standard $L$-function}{The derivative of the p-adic standard L-function}}

In this section we explicit the trivial zeros of the $p$-adic $L$-function $\mu_{\sC_P,\phi,\beta_1,\beta_2}$ of Theorem \ref{thm:twopadicLfun} and interpret them from the point of view of the (conjectural) associated $p$-adic Galois representation. This will allow us to interpret a factor appearing in the derivative of the $p$-adic $L$-function as Greenberg's $\ell$-invariant in the case when the trivial zero is semi-stable (or of type $M$ as called in \cite{TTT}). At the end, we shall prove the main theorem of the paper which relates the derivative of $\cL_{\sC_P,\phi\omega^{n+1},\beta_1,\beta_2}$ at the semi-stable trivial zero to the $\ell$-invariant and the complex special value.

\subsection{\texorpdfstring{Greenberg--Benois conjecture and $\ell$-invariants}{Greenberg-Benois conjecture and l-invariants}}

That idea at the base of the conjecture by Greenberg and Benois is that when a $p$-adic $L$-function has trivial zeros one should be able to recover the value of the complex $L$-function from a suitable derivative of the $p$-adic $L$-function. The first studied case is the one of an elliptic curve with split multiplicative reduction at $p$ \cite{MTT}. They conjectured that the order of the $p$-adic $L$-function is $1$ plus the order of the complex $L$-function and that the leading coefficient of the $p$-adic $L$-function is the same as the one of the complex $L$-function, up to an error term equal to $\log_p(q_E)/\mr{ord}_p(q_E)$ (for $q_E$ the Tate uniformizer of $E$) that  they call the $\ell$-invariant. 

Later Greenberg interpreted this number in Galois theoretical terms and proposed a similar conjecture for a great number of ordinary motives  \cite{TTT}; succesively Benois generalised this conjecture to most semistable representations \cite{BenLinv}. Let us be more precise. Let $V$ the $p$-adic Galois representation associated with the motive, and suppose that it is absolutely irreducible and satisfies Pantchishkine condition, {\it i.e.} that there is a sub space $V'$ of $V$ stable under the action of $G_{\mb Q_p}$ and containing all positive Hodge--Tate weights. Assume moreover that the Frobenius acts semisimply on the semi-stable module (in the sense of Fontaine) associated with $V$. 

Greenberg defines two subspaces $ V^{11} \subset V^+ \subset V^{00}$ such that $V^+/V^{11}$ contains all the eigenvalues $p$ and $V^{00}/V^+$ contains all the eigenvalues $1$. Then we can decompose $V^{00}/V^{11}$ as $\mb Q_p^{t_0} \oplus M \oplus \mb Q_p(1)^{t_1}$, where $M$ is a non split extension of $\mb Q_p^t$ by $\mb Q_p(1)^t$. According to Greenberg's conjecture, the number of trivial zeros of the $p$-adic $L$-function is $e=t_0+t+t_1$. Assume furthermore that $t_1=0$ (certain motivic conjectures imply $t_1t_0=0$, so this hypothesis is not really restrictive as if $t_0=0$ can always consider the dual motive for which $t_0=0$). Then by picking a subspace $\wt{\mathbf{T}}\subset H^1(G_{\bQ},V)$ of dimension $e$, Greenberg defines an $e$-dimensional subspace $\wt{\mathbf{T}}_p$ inside $H^1(G_{\mb Q_p},V)/H^{1}_{\mr f}(G_{\mb Q_p},V)$ which injects into \begin{align*}
H^{1}(G_{\mb Q_p},V^1/V')\cong \bigoplus^{t+e_0}_{i=1}\mb Q_p \cdot \mr{ord}_p \oplus \mb Q_p \cdot \log_p.
\end{align*}  
If we denote by $p_u$ (resp. $p_c$) the projection of $\wt{\mathbf{T}}_p$ to  $\bigoplus^{t+e_0}_{i=1}\mb Q_p \cdot \mr{ord}_p$ (resp. $\bigoplus^{t+e_0}_{i=1}\mb Q_p \cdot \log_p$), then Greenberg shows that $p_c$ is an invertible map and defines $\ell(V):=\mr{det}(p_u \circ p_c^{-1})$. Note that in general $\wt{\mathbf{T}}_p$ depends on $V$ as $G_{\mb Q}$-representation, but if $e_0=0$ too then it depends only on the restriction to $G_{\mb Q_p}$ \cite[p.~169]{TTT}, \cite[p.~1239]{RosLinv}. We can finally state:
\begin{conj}\label{conj:greenberg}[Greenberg--Benois conjecture]
Let $r$ be  the order of the complex $L$-function $L(s,V)$ at $s=0$, and suppose that $\cL_p(S,V)$ has $e_0+t$ trivial zeros, then $S^{e_0+t+r}$ divides $\cL_p(S,V)$ exactly and
\begin{align*}
\cL_p(S,V)/S^{e_0+t+r} \equiv \ell(V^*(1)) {(E_p(0,V)L_p(0,V)^{-1})}^* \frac{L^{\rm{alg},(r)}(0,V)}{\log_p(1+p)^{e_0+t+r} (e_0+t+r)!}  \bmod S,
\end{align*}
where $E_p(s,V)$ is defined as in \cite[\S6]{CoaMot}, $V^*(1)$ is the dual representation twisted by $1$, $L_p(s,V)$ is the Euler factor of the motivic $L$-function for $V$ and ${(E_p(s,V)L_p(s,V)^{-1})}^*$ is obtained from $E_p(s,V)L_p(s,V)^{-1}$ by eliminating all the Euler factors vanishing at $s=0$. The function $L^{\rm{alg}}(s,V)$ is the $L$-function for $V$ divided by the period.
\end{conj}
Note that the conjecture implies the non-vanishing of the $\ell$-invariant $\ell(V)$.

This conjecture has been first shown in the case of an elliptic curve with bad multiplicative reduction by Greenberg--Stevens \cite{SSS}. Their method has been generalised many times to different contexts \cite{Mok,RosOC,RosH,RosSiegel} and is at the base of our current approach. It is worth to point out that many new strategies have recently arisen \cite{DDP,Dasg,Spi1,Deppe}.

\subsection{\texorpdfstring{The trivial zero of the $p$-adic standard $L$-function}{The trivial zero of the p-adic standard L-function}}\label{sec:trivzero}
Let $\cL_{\sC_P,\phi\omega^{j},\beta_1,\beta_2}=\cL_{\sC_P,\phi,\beta_1,\beta_2}(S,x)\in \bI_{\sC_P}[[S]]\otimes_{\bI_{\sC_P}}F_{\sC_P}$ to be the Mellin transform of the component corresponding to the character $\omega^j$ on $\left(\bZ/p\right)^\times$ of the measure $\mu_{\sC_P,\phi,\beta_1,\beta_2}$ constructed in Theorem \ref{thm:twopadicLfun}.

Suppose that $x_0:\bI_{\sC_P}\ra F$ is a point as in Theorem \ref{thm:twopadicLfun} and it is classical. Denote by $\pi_{x_0}$ a cuspidal irreducible automorphic representation of $\Sp(2n,\bA)$ generated by a $P$-ordinary Siegel modular $\bT^{0,N}_{\Pord}$-eigenform of weight $\iota(\ut^P)$ whose eigenvalues are parametrized by $x_0$. The isomorphism class of $\pi_{x_0,v}$ is determined by $x_0$ for $v\nmid N$ (the isomorphism class of $\pi_{x_0,p}$ can be read off from the eigenvalues for all $\bU^P_p$-operators by the discussion in \S\ref{sec:JM}). We are interested in possible trivial zeros of $\cL_{\sC_P,\phi\omega^{n+1},\beta_1,\beta_2}$ at $S=(1+p)^{n+1}-1$, where the corresponding critical $L$-value is the near-central value $L(0,\pi_x\times\phi)$. There are two types of trivial zeroes that can show up there.

\subsubsection{Crystalline trivial zero} 
We say that $\cL_{\sC_P,\phi\omega^{n+1},\beta_1,\beta_2}$ has a crystalline trivial zero at $((1+p)^{n+1}-1,x_0)$ if  $\phi_p(p)=1$ and the local $L$-factor $L_p(s,\pi_{x_0}\times \phi_p)$ contains the factor $(1-\phi_p(p)p^{-s})$. In particular, if $\pi_{x_0,p}$ is unramified, then a crystalline trivial zero shows up at $S=(1+p)^{n+1}-1$. 

One can also think of types of trivial zeroes in terms of the corresponding Galois representation restricted to $G_{\bQ_p}$ (if admitting the local-global compatibility for $\rho_{x_0}$ attached to $\pi_{x_0}$). Let $\rho_{x_0}:G_{\bQ}\ra \GL(2n+1,\ol{\bQ}_p)$ be the Galois representation attached to $\pi_{x_0}$ \cite{Arthur,CH}. Then conjecturally $\rho_{x_0}|_{G_{\bQ_p}}$ admits the following description. The Hodge--Tate weights are $0,\pm (t^P_1-1),\dots,\pm (t^P_1-N_1),\pm(t^P_2-N_1-1),\dots,\pm(t_2-N_2),\dots,\pm (t^P_d-N_{d-1}-1),\dots,\pm (t^P_d-n)$, and the eigenvalues of the Frobenius are $1,\alpha^{\pm 1}_1,\dots,\alpha^{\pm 1}_n$. Hence the Hodge polygon and the Newton polygon meet at the points with horizontal coordinates $0, N_1, N_2,\dots,N_d,2n+1-N_d,\dots,2n+1-N_2,2n+1-N_1,2n+1$. Therefore, $\rho_{x_0}$ admits a decreasing filtration $\mr{Fil}^j$, $-d\leq j\leq d$, such that $\mr{Fil}^j/\mr{Fil}^{j+1}$ has Hodge--Tate weights $t^P_{d-j+1}-N_{d-j}-1,\dots,t^P_{d-j+1}-N_{d-j+1}$ (resp. $t^P_{-j}-N_{-j+1}-1,\dots,t^P_{-j}-N_{-j}$) if $1\leq j\leq d$ (resp. $-d\leq j\leq -1$), and $\mr{Fil}^0/\mr{Fil}^{1}$ is one dimensional with trivial $G_{\bQ_p}$-action. If $\mr{Fil}^0/\mr{Fil}^2$ is a trivial extension of $\ol{\bQ}_p$ by $\mr{Fil}^1/\mr{Fil}^2$, then $\cL_{\sC_P,\phi\omega^{n+1},\beta_1,\beta_2}$ has a crystalline trivial zero at $((1+p)^{n+1}-1,x_0)$ if  $\phi_p(p)=1$.

\subsubsection{Semi-stable trivial zero}\label{sec:trivzerosst}
The other case is when the local $L$-factor $L_p(s,\pi_{x_0}\times\phi_p)$ does not contain the factor $(1-\phi_p(p)p^{-s})$, but the factor $(1-\phi_p(p)^{-1}\alpha_n^{-1} p^{s-1})$ in $E_p(s,\pi_{x_0}\times\phi)$ contributes a trivial zero at $((1+p)^{n+1}-1,x_0)$ when $\phi_p(p)=1$ and $\alpha_n=p^{-1}$. We call this type of trivial zero semi-stable. Since Newton polygon lies above the Hodge polygon and for $\rho_{x_0}|_{G_{\bQ_p}}$ they coincide along the slope $0$ segment, we have $v_p(\alpha_n)\leq -t^P_d+n\leq -1$. If $\alpha_n=p^{-1}$, then $t^P_d=n+1$ and the Newton and Hodge polygons coincide along the the segments of slope $-1,0,1$. Therefore, if $\cL_{\sC_P,\phi\omega^{n+1},\beta_1,\beta_2}(S,x_0)\in\cO_F[[S]]\otimes_{\cO_F} F$ has a semi-stable trivial zero at $S=(1+p)^{n+1}-1$, then there exists a partition $n=n_1+\dots +n_d$ with $n_d=1$ such that $\pi_{x_0}$ is also $P'$-ordinary for $P'\subset\GL(n)$ the standard parabolic subgroup attached to this partition. 

The above discussion shows that in order to use $p$-adic deformation of $\pi_{x_0}$ to study the trivial zero at $S=(1+p)^{n+1}-1$ of the $p$-adic $L$-function attached to $\pi_{x_0}$, one can always choose $P$ such that $n_d=1$ and consider the $P$-ordinary families passing through $x_0$. In the following, we assume that $n_d=1$. 

In terms of the local Galois representation, a semi-stable trivial zero appears if $\mr{Fil}^0/\mr{Fil}^2$ is a non-trivial extension of $\ol{\bQ}_p$ by $\ol{\bQ}_p(1)$.

\subsection{A formula for the derivative}
In this section we prove the main theorem of the paper applying the strategy of Greenberg--Stevens. Recall that $\mf{a}_i$, $1\leq i \leq n$, denotes the eigenvalue of the operator $U^P_{p,i}$ (see Proposition~\ref{prop:JL}). They are invertible elements inside $\bI_{\sC_P}$. In order to state \cite[Theorem 2]{BenLinv2} in our setting, we need to fix a coordinate for a rigid analytic open neighborhood of $x_0$ in $\sC_P$. Since we have assumed that the weight projection map is \'{e}tale at $x_0$, we can take the coordinates to be the natural coordinates $T^P_1,\dots,T^P_d$ of the weight space. Then it follows from \cite[Theorem 2]{BenLinv2} and \cite[Proposition 2.2.24]{BenLinv} that 
\begin{theo}\label{teo:linva}
Let $x_0\in\sC_P(F)$ be a classical point where the weight projection map is \'{e}tale and has image $\utau^P_{x_0}\in\Hom_{\cont}(T^P(\bZ_p),F^\times)$. Suppose that $S=(1+p)^{n+1}-1$ is a semi-stable trivial zero for $\cL_{\sC_P,\phi\omega^{n+1},\beta_1,\beta_2}(S,x_0)$ and the local-global compatibility is satisfied by the $p$-adic Galois representation $\rho_{x_0}$. Then 
\begin{align*}
\ell(\rho_{x_0})=\ell(\rho_{x_0}^*(1)) =- \left. \frac{\partial \log_p \left(\mf{a}_{n}(T^P_1,\dots T^P_d)/\mf{a}_{n-1}(T^P_1,\dots,T^P_d)\right)}{\partial T^P_d}  \right|_{(T^P_1,\dots,T^P_d)=\utau^P_{x_0}(1+p)}.
\end{align*}
\end{theo}
(For the proof see  \cite[Theorem 1.3]{RosLinv}, where the theorem is stated for parallel weight, but the proof is the same.)

\begin{rem}
When $n=2$ one can calculate the $\ell$-invariant also for crystalline trivial zero using the method of \cite{HIwa}.\end{rem}

\begin{lemma}\label{rem:alltrivialzero}
The $p$-adic $L$-function $\cL_{\sC_P,\phi\omega^{n+1},\beta_1,\beta_2}$ vanishes identically on the hyper-plane $S=(1+p)^{n+1}-1$ whenever $\phi_p(p)=1$. 
\end{lemma}
\begin{proof}
We shall show that if $\phi_p(p)=1$ then $\mu_{\cE,\Pord}$ vanishes at all interpolation points with $S=(1+p)^{n+1}-1$. This will imply that $\mu_{\cE,\Pord}$ vanishes on the hyper-plane $S=(1+p)^{n+1}-1$, so $\cL_{\sC_P,\phi\omega^{n+1},\beta_1,\beta_2}$ will also vanish on that hyper-plane.

Indeed, from Proposition \ref{prop:FourierCoeffEis} we see that all Fourier coefficients $\fc_{\kappa,\utau^P}(\bbeta)$ present the factor $ L^{Np\infty}(0,\phi\lambda_{\bbeta})$; we recall that $\lambda_{\bbeta}(p)$ is the Legendre symbol $\left(\frac{(-1)^{\mr{rank}(\bbeta)/2}\det^*(\bbeta)}{p}\right)$, where $\det^*(\bbeta)$ denotes the product of all the nonzero eigenvalues of $\bbeta$. 

Note that in the construction of $\mu_{\cE,\Pord}$, our choice of local sections at $p$ makes $W_{\bbeta,p}$ vanish unless $\bbeta=\begin{pmatrix}\beta_1&\beta_0\\\ltrans{\beta}_0&\beta_2\end{pmatrix}$ with $\beta_0\in\GL(n,\bZ_p)$ and $\beta_1\equiv 0\mod p$ (see Table \ref{tb:d+1}). 

Such $\bbeta$'s are of maximal rank $2n$, and hence 
\[{\det}^*(\bbeta)=\det(\bbeta)\equiv \begin{pmatrix} 0&\beta_0\\\ltrans{\beta}_0&\beta_2\end{pmatrix}  \equiv (-1)^n\det(\beta_0)^2 \bmod p.\]  Thus, $(-1)^{\mr{rank}(\bbeta)/2}\det(\bbeta)$
 belongs to $\bZ^\times_p$ and 
 \[(-1)^{\mr{rank}(\bbeta)/2}\det(\bbeta) \equiv(-1)^{\mr{rank}(\bbeta)/2} (-1)^n\det(\beta_0)^2  \equiv \det(\beta_0)^2 \bmod p.\]

It follows that $\lambda_{\bbeta}(p)=1$. Therefore, as we also have $\phi_p(p)=1$, we get
\[L^{Np\infty}(0,\phi\lambda_{\bbeta}) = (1-\lambda_{\bbeta}(p)\phi_p(p)p^0)L^{N\infty}(0,\phi\lambda_{\bbeta})=0,\]
which implies the vanishing of all the $\fc_{\kappa,\utau^P}(\bbeta)$'s and the vanishing of the measure $\mu_{\cE,\Pord}$ along $S=(1+p)^{n+1}-1$.
\end{proof}
\begin{rem} It is possible to define an algebraic $\ell$-invariant $\ell(\sC_P)$, which is a meromorphic function on $\sC_P$, using the definition of \cite{BenLinv} in the context of \cite{PottDerived}.
\end{rem}
Define $j(\sC_P)\in\bZ/(p-1)$ by $\omega^{j(\sC_P)}=\uptau^P_d|_{(\bZ/p)^\times}$ for a $\utau^P$ inside the image of the projection of $\sC_P$ to the weight space. The relation of the improved $p$-adic $L$-function $\cL^{\imp}_{\sC_P,\phi,\beta_1,\beta_2}$ and the restriction of $\cL_{\sC_P,\phi\omega^{j(\sC_P)},\beta_1,\beta_2}$ to $\kappa=\uptau^P_{d}$ is given as:
\begin{prop}
Suppose that $n_d=1$. As elements in $F_{\sC_P}$ we have 
\begin{equation*}
  \left. \cL_{\sC_P,\phi\omega^{j(\sC_P)},\beta_1,\beta_2}\right\vert_{\kappa=\uptau^P_{d}}=\left(1-\phi_p(p)^{-1} \mf{a}_{n-1}/\mf{a}_n \right)\cL^{\imp}_{\sC_P,\phi,\beta_1,\beta_2}.
\end{equation*}
\end{prop}
\begin{proof}
This follows straightforwardly from the two interpolation formulae of Theorem \ref{thm:twopadicLfun} by noticing that when the point $x$ is classical and $n_d=1$, $\cA^P(\pi_x\times\xi)=(1-\phi_p(p)^{-1}\alpha_n^{-1}p^{n-t^P_d})$, so
\begin{equation*}
   E_p(n+1-t^P_{d,x},\pi_x\times\phi\epsilon^P_d)= \left(1-\phi_p(p)^{-1} \mf{a}_{n-1}(x)/\mf{a}_n(x) \right)E^{P\text{-imp}}_p(n+1-t^P_{d,x},\pi_x\times\phi\epsilon^P_d)
\end{equation*}
(cf. \eqref{eq:Ecompare}).
\end{proof}

Now we are ready to prove the main theorem. 
\begin{theo}\label{thm:mainzero}
Let $x_0$ be an $F$-point of $\sC_P$ where the weight projection map $\Lambda_P\ra \bT^{1,N}_{\Pord}$ is \'{e}tale and maps $x_0$ to $\utau^P_0$. Suppose that $x_0$ is classical and that the $p$-adic $L$-function $\cL_{\sC_P,\phi\omega^{n+1},\beta_1,\beta_2}\in  \bI_{\sC_P}[[S]] \otimes_{\bI_{\sC_P} } F_{\sC_P}$ has a semi-stable trivial zero at $((1+p)^{n+1}-1,x_0)$ (so $j(\sC_P)=n+1$) and the local-global compatibility is satisfied by the $p$-adic Galois representation $\rho_{x_0}$. (This in particular implies that $\epsilon^P_d$ is trivial as $\phi$ has level prime to $p$.) Then we have 
\begin{align*}
&\left.\frac{\textup{d}\cL_{\sC_P,\phi\omega^{n+1},\beta_1,\beta_2}(S,x_0)}{\textup{d} S}\right|_{S=(1+p)^{n+1}-1}\\
   =& \, -\ell(\rho_{x_0})\cdot C_{x_0,\beta_1,\beta_2,\phi,N}\cdot E^{P\textnormal{-imp}}_p(0,\pi_{x_0}\times\phi)\, E_\infty(0,\pi_{x_0}\times\phi)\, L^{Np\infty}(0,\pi_{x_0}\times\phi),
\end{align*}
\end{theo}
\begin{proof}
Again, we use the natural coordinate $T^P_1,\dots, T^P_d$ of the weight space to parametrize a rigid analytic open neighborhood of $x_0$ in $\sC_P$. Note that by Lemma \ref{rem:alltrivialzero}, we know that $\cL_{\sC_P,\phi\omega^{n+1},\beta_1,\beta_2}((1+p)^{n+1}-1,x)$ is identically vanishing, so  
\begin{equation*}
\left.\frac{\partial\cL_{\sC_P,\phi\omega^{n+1},\beta_1,\beta_2}(S,T^P_1,\dots,T^P_d)}{\partial T^P_d}\right|_{S=(1+p)^{n+1}-1}=0.
\end{equation*}
It follows that
\begin{align*}
   &\left.\frac{\textup{d}\cL_{\sC_P,\phi,\beta_1,\beta_2}(S,x_0)}{\textup{d} S}\right|_{S=(1+p)^{n+1}-1}\\
   =&\left.\left(\frac{\partial\cL_{\sC_P,\phi\omega^{n+1},\beta_1,\beta_2}(S,T^P_1,\dots,T^P_d)}{\partial S}+\frac{\partial\cL_{\sC_P,\phi\omega^{n+1},\beta_1,\beta_2}(S,T^P_1,\dots,T^P_d)}{\partial T^P_d}\right)\right|_{S=(1+p)^{n+1}-1,(T^P_1,\dots,T^P_d)=\utau^P_0(1+p)}\\
   =&\left.\frac{\partial}{\partial T^P_d}\left(\left(1-\phi_p(p)^{-1} \mf{a}_{n-1}(T^P_1,\dots,T^P_d)/\mf{a}_n(T^P_1,\dots,T^P_d) \right)\cL^{\mr{imp}}_{\sC_P,\phi,\beta_1,\beta_2}(T^P_1,\dots,T^P_d)\right)\right|_{(T^P_1,\dots,T^P_d)=\utau^P_0(1+p)}\\
   =&\left. \frac{\partial \left(\mf{a}_{n}(T^P_1,\dots T^P_d)/\mf{a}_{n-1}(T^P_1,\dots,T^P_d)\right)}{\partial T^P_d}  \right|_{(T^P_1,\dots,T^P_d)=\utau^P_{0}(1+p)}\cdot \,\cL^{\mr{imp}}_{\sC_P,\phi,\beta_1,\beta_2}(x_0).
\end{align*}
Then the theorem follows from Theorem~\ref{teo:linva} and the interpolation property of $\cL^{\mr{imp}}_{\sC_P,\phi,\beta_1,\beta_2}$ in Theorem~\ref{thm:twopadicLfun}.
\end{proof}

\bibliographystyle{alpha}
\bibliography{Bibliografy} 

\begin{thebibliography}{BSDGP96}

\bibitem[AIP15]{AIP}
Fabrizio Andreatta, Adrian Iovita, and Vincent Pilloni.
\newblock {$p$}-adic families of {S}iegel modular cuspforms.
\newblock {\em Ann. of Math. (2)}, 181(2):623--697, 2015.

\bibitem[Art13]{Arthur}
James Arthur.
\newblock {\em {T}he Endoscopic Classification of Representations Orthogonal and Symplectic Groups}, volume~61.
\newblock American Mathematical Soc., 2013.

\bibitem[Ben10]{BenLinv2}
Denis Benois.
\newblock Infinitesimal deformations and the {$\ell$}-invariant.
\newblock {\em Doc. Math.}, (Extra volume: Andrei A. Suslin sixtieth birthday):5--31, 2010.

\bibitem[Ben11]{BenLinv}
Denis Benois.
\newblock A generalization of {G}reenberg's {$\mathcal{ L}$}-invariant.
\newblock {\em Amer. J. Math.}, 133(6):1573--1632, 2011.

\bibitem[BPS16]{BPS}
St\'ephane Bijakowski, Vincent Pilloni, and Beno{\^\i}t Stroh.
\newblock Classicit\'e de formes modulaires surconvergentes.
\newblock {\em Ann. of Math. (2)}, 183(3):975--1014, 2016.

\bibitem[BR89]{BR}
D.~Blasius and J.~Rogawski.
\newblock Galois representations for {H}ilbert modular forms.
\newblock {\em Bull. Amer. Math. Soc. (N.S.)}, 21(1):65--69, 1989.

\bibitem[BR15]{BrasRos}
Riccardo Brasca and Giovanni Rosso.
\newblock Eigenvarieties for non-cuspidal modular forms over certain {PEL} type {S}himura varieties.
\newblock preprint, \href{https://arxiv.org/abs/1605.05065}{arXiv:1605.05065}, 2015.

\bibitem[BR19]{BraRosmu}
Riccardo Brasca and Giovanni Rosso.
\newblock {H}ida theory over some unitary {S}himura varieties without ordinary locus.
\newblock {\em American Journal of Mathematics}, 2019.
\newblock to appear.

\bibitem[BS00]{BS}
S.~B{\"o}cherer and C.-G. Schmidt.
\newblock {$p$}-adic measures attached to {S}iegel modular forms.
\newblock {\em Ann. Inst. Fourier (Grenoble)}, 50(5):1375--1443, 2000.

\bibitem[BSDGP96]{Saint}
Katia Barr{\'e}-Sirieix, Guy Diaz, Fran{\c{c}}ois Gramain, and Georges Philibert.
\newblock Une preuve de la conjecture de {M}ahler-{M}anin.
\newblock {\em Invent. Math.}, 124(1-3):1--9, 1996.

\bibitem[Cas]{CasselJ}
William~A. Casselman.
\newblock Introduction to the theory of admissible representations of p-adic reductive groups.
\newblock unpublished notes distributed by P. Sally, draft dated May 1, 1995.

\bibitem[CH13]{CH}
Ga\"etan Chenevier and Michael Harris.
\newblock Construction of automorphic {G}alois representations, {II}.
\newblock {\em Camb. J. Math.}, 1(1):53--73, 2013.

\bibitem[Che11]{Chfern}
Ga\"{e}tan Chenevier.
\newblock On the infinite fern of {G}alois representations of unitary type.
\newblock {\em Ann. Sci. \'{E}c. Norm. Sup\'{e}r. (4)}, 44(6):963--1019, 2011.

\bibitem[CHLN11]{BP}
Laurent Clozel, Michael Harris, Jean-Pierre Labesse, and Bao-Ch\^{a}u Ng\^{o}.
\newblock {\em Stabilization of the trace formula, {S}himura varieties, and arithmetic applications}.
\newblock Boston: International Press, 2011.

\bibitem[Coa91]{CoaMot}
John Coates.
\newblock Motivic {$p$}-adic {$L$}-functions.
\newblock In {\em {$L$}-functions and arithmetic ({D}urham, 1989)}, volume 153 of {\em London Math. Soc. Lecture Note Ser.}, pages 141--172. Cambridge Univ. Press, Cambridge, 1991.

\bibitem[Das16]{Dasg}
Samit Dasgupta.
\newblock Factorization of {$p$}-adic {R}ankin {$L$}-series.
\newblock {\em Invent. Math.}, 205(1):221--268, 2016.

\bibitem[DDP11]{DDP}
Samit Dasgupta, Henri Darmon, and Robert Pollack.
\newblock Hilbert modular forms and the {G}ross-{S}tark conjecture.
\newblock {\em Ann. of Math. (2)}, 174(1):439--484, 2011.

\bibitem[Dep16]{Deppe}
Holger Deppe.
\newblock {$p$}-adic {L}-functions of automorphic forms and exceptional zeros.
\newblock {\em Doc. Math.}, 21:689--734, 2016.

\bibitem[DJR18]{DJR}
Mladen Dimitrov, Fabian Januszewski, and A~Raghuram.
\newblock ${L}$-functions of $\mathrm{{GL}}(2n)$: p-adic properties and nonvanishing of twists.
\newblock {\em preprint available at https://arxiv.org/abs/1802.10064}, 2018.

\bibitem[EHLS16]{EHLS}
E.~Eischen, M.~Harris, J.~Li, and C.~Skinner.
\newblock $p$-adic {$L$}-functions for unitary groups, part {II}: zeta-integral calculations.
\newblock 2016.

\bibitem[EM17]{EischenMantovan}
Ellen Eischen and Elena Mantovan.
\newblock $p$-adic families of automorphic forms in the $\mu$-ordinary setting.
\newblock preprint. \href{https://arxiv.org/abs/1710.01864}{arXiv:1710.01864}, 2017.

\bibitem[EW16]{EisWan}
Ellen Eischen and Xin Wan.
\newblock {$p$}-adic {E}isenstein series and {$L$}-functions of certain cusp forms on definite unitary groups.
\newblock {\em J. Inst. Math. Jussieu}, 15(3):471--510, 2016.

\bibitem[FC90]{FC}
Gerd Faltings and Ching-Li Chai.
\newblock {\em Degeneration of abelian varieties}, volume~22 of {\em Ergebnisse der Mathematik und ihrer Grenzgebiete (3) [Results in Mathematics and Related Areas (3)]}.
\newblock Springer-Verlag, Berlin, 1990.
\newblock With an appendix by David Mumford.

\bibitem[Gar84]{GaPull}
Paul~B. Garrett.
\newblock Pullbacks of {E}isenstein series; applications.
\newblock In {\em Automorphic forms of several variables ({K}atata, 1983)}, volume~46 of {\em Progr. Math.}, pages 114--137. Birkh\"auser Boston, Boston, MA, 1984.

\bibitem[GPSR87]{GPSR}
Stephen Gelbart, Ilya Piatetski-Shapiro, and Stephen Rallis.
\newblock {\em Explicit constructions of automorphic {$L$}-functions}, volume 1254 of {\em Lecture Notes in Mathematics}.
\newblock Springer-Verlag, Berlin, 1987.

\bibitem[Gre94]{TTT}
Ralph Greenberg.
\newblock Trivial zeros of {$p$}-adic {$L$}-functions.
\newblock In {\em {$p$}-adic monodromy and the {B}irch and {S}winnerton-{D}yer conjecture ({B}oston, {MA}, 1991)}, volume 165 of {\em Contemp. Math.}, pages 149--174. Amer. Math. Soc., Providence, RI, 1994.

\bibitem[GS93]{SSS}
Ralph Greenberg and Glenn Stevens.
\newblock {$p$}-adic {$L$}-functions and {$p$}-adic periods of modular forms.
\newblock {\em Invent. Math.}, 111(2):407--447, 1993.

\bibitem[Har81]{Ha81}
Michael Harris.
\newblock Special values of zeta functions attached to {S}iegel modular forms.
\newblock {\em Ann. Sci. \'Ecole Norm. Sup. (4)}, 14(1):77--120, 1981.

\bibitem[Har84]{HaEC}
M.~Harris.
\newblock {E}isenstein series on {S}himura varieties.
\newblock {\em Ann. of Math.}, 119(1):59--94, 1984.

\bibitem[Hid90]{H6}
Haruzo Hida.
\newblock {$p$}-adic {$L$}-functions for base change lifts of {${\rm GL}_2$} to {${\rm GL}_3$}.
\newblock In {\em Automorphic forms, {S}himura varieties, and {$L$}-functions, {V}ol.\ {II} ({A}nn {A}rbor, {MI}, 1988)}, volume~11 of {\em Perspect. Math.}, pages 93--142. Academic Press, Boston, MA, 1990.

\bibitem[Hid93]{H}
Haruzo Hida.
\newblock {\em Elementary theory of {$L$}-functions and {E}isenstein series}, volume~26 of {\em London Mathematical Society Student Texts}.
\newblock Cambridge University Press, Cambridge, 1993.

\bibitem[Hid02]{HPEL}
Haruzo Hida.
\newblock Control theorems of coherent sheaves on {S}himura varieties of {PEL} type.
\newblock {\em J. Inst. Math. Jussieu}, 1(1):1--76, 2002.

\bibitem[Hid06]{HIwa}
Haruzo Hida.
\newblock {\em Hilbert modular forms and {I}wasawa theory}.
\newblock Oxford Mathematical Monographs. The Clarendon Press Oxford University Press, Oxford, 2006.

\bibitem[Hsi14]{Hsieh}
Ming-Lun Hsieh.
\newblock Eisenstein congruence on unitary groups and {I}wasawa main conjectures for {CM} fields.
\newblock {\em J. Amer. Math. Soc.}, 27(3):753--862, 2014.

\bibitem[Jan15]{Januszewski1}
Fabian Januszewski.
\newblock On {$p$}-adic {$L$}-functions for {$\mathrm{GL}(n)\times\mathrm{GL}(n-1)$} over totally real fields.
\newblock {\em Int. Math. Res. Not. IMRN}, (17):7884--7949, 2015.

\bibitem[Jan16]{Januszewski2}
Fabian Januszewski.
\newblock {$p$}-adic {$L$}-functions for {R}ankin-{S}elberg convolutions over number fields.
\newblock {\em Ann. Math. Qu{\'e}.}, 40(2):453--489, 2016.

\bibitem[KR90]{KRPoles}
Stephen~S. Kudla and Stephen Rallis.
\newblock Poles of {E}isenstein series and {$L$}-functions.
\newblock In {\em Festschrift in honor of {I}. {I}. {P}iatetski-{S}hapiro on the occasion of his sixtieth birthday, {P}art {II} ({R}amat {A}viv, 1989)}, volume~3 of {\em Israel Math. Conf. Proc.}, pages 81--110. Weizmann, Jerusalem, 1990.

\bibitem[Lanar]{LanOrd}
Kai-wen Lan.
\newblock {\em Compactifications of PEL-Type Shimura Varieties and Kuga Families with Ordinary Loci}.
\newblock World Scientific Publishing Co., Singapore, to appear.

\bibitem[Liu16]{LiuSLF}
Zheng Liu.
\newblock {$p$}-adic {$L$}-functions for ordinary families on symplectic groups.
\newblock {\em to appear in J. Inst. Math. Jussieu}, 2016.
\newblock available at \url{https://doi.org/10.1017/S1474748018000415}.

\bibitem[Liu19a]{LiuAZI}
Zheng Liu.
\newblock The doubling {A}rchimedean zeta integrals for {$p$}-adic interpolation.
\newblock available at \url{https://arxiv.org/pdf/1904.07121.pdf}, 2019.

\bibitem[Liu19b]{LiuNHF}
Zheng Liu.
\newblock Nearly overconvergent {S}iegel modular forms.
\newblock {\em Ann. Inst. Fourier (Grenoble)}, 69(6):2439--2506, 2019.

\bibitem[LR05]{LapidRallis}
Erez~M. Lapid and Stephen Rallis.
\newblock {O}n the local factors of representations of classical groups.
\newblock In {\em Automorphic representations, {$L$}-functions and applications: progress and prospects}, volume~11 of {\em Ohio State Univ. Math. Res. Inst. Publ.}, pages 309--359. de Gruyter, Berlin, 2005.

\bibitem[MgVW87]{MVW}
Colette M\oe~glin, Marie-France Vign\'eras, and Jean-Loup Waldspurger.
\newblock {\em Correspondances de {H}owe sur un corps {$p$}-adique}, volume 1291 of {\em Lecture Notes in Mathematics}.
\newblock Springer-Verlag, Berlin, 1987.

\bibitem[Mok09]{Mok}
Chung~Pang Mok.
\newblock The exceptional zero conjecture for {H}ilbert modular forms.
\newblock {\em Compos. Math.}, 145(1):1--55, 2009.

\bibitem[MTT86]{MTT}
B.~Mazur, J.~Tate, and J.~Teitelbaum.
\newblock On {$p$}-adic analogues of the conjectures of {B}irch and {S}winnerton-{D}yer.
\newblock {\em Invent. Math.}, 84(1):1--48, 1986.

\bibitem[Pil12]{PilHida}
Vincent Pilloni.
\newblock Sur la th\'eorie de {H}ida pour le groupe {${\rm GSp}_{2g}$}.
\newblock {\em Bull. Soc. Math. France}, 140(3):335--400, 2012.

\bibitem[Pot13]{PottDerived}
Jonathan Pottharst.
\newblock Analytic families of finite-slope {S}elmer groups.
\newblock {\em Algebra Number Theory}, 7(7):1571--1612, 2013.

\bibitem[PSR87]{PSR}
I.~Piatetski-Shapiro and S.~Rallis.
\newblock ${L}$-functions for the classical groups.
\newblock volume 1254 of {\em Lecture Notes in Math}, pages 1--52. Springer, 1987.

\bibitem[Ros15]{RosLinv}
Giovanni Rosso.
\newblock {$\mathcal{L}$}-invariant for {S}iegel--{H}ilbert forms.
\newblock {\em Doc. Math.}, 20:1227--1253, 2015.

\bibitem[Ros16a]{RosH}
Giovanni Rosso.
\newblock Derivative at {$s=1$} of the {$p$}-adic {$L$}-function of the symmetric square of a {H}ilbert modular form.
\newblock {\em Israel J. Math.}, 215(1):255--315, 2016.

\bibitem[Ros16b]{RosOC}
Giovanni Rosso.
\newblock A formula for the derivative of the {$p$}-adic {$L$}-function of the symmetric square of a finite slope modular form.
\newblock {\em Amer. J. Math.}, 138(3):821--–878, 2016.

\bibitem[Ros18]{RosSiegel}
Giovanni Rosso.
\newblock Derivative of the standard {$p$}-adic {$L$}-function associated with a {S}iegel form.
\newblock {\em Trans. Amer. Math. Soc.}, 370(9):6469--6491, 2018.

\bibitem[RS07]{RSLocal}
Brooks Roberts and Ralf Schmidt.
\newblock {\em Local newforms for {GS}p(4)}, volume 1918 of {\em Lecture Notes in Mathematics}.
\newblock Springer, Berlin, 2007.

\bibitem[Shi82]{Sh6}
Goro Shimura.
\newblock Confluent hypergeometric functions on tube domains.
\newblock {\em Math. Ann.}, 260(3):269--302, 1982.

\bibitem[Shi97]{ShEuler}
Goro Shimura.
\newblock {\em Euler products and {E}isenstein series}, volume~93 of {\em CBMS Regional Conference Series in Mathematics}.
\newblock Published for the Conference Board of the Mathematical Sciences, Washington, DC; by the American Mathematical Society, Providence, RI, 1997.

\bibitem[Shi00]{Sh00}
Goro Shimura.
\newblock {\em {A}rithmeticity in the theory of automorphic forms}, volume~82 of {\em Mathematical Surveys and Monographs}.
\newblock American Mathematical Society, Providence, RI, 2000.

\bibitem[Shi11]{Shin}
Sug~Woo Shin.
\newblock Galois representations arising from some compact {S}himura varieties.
\newblock {\em Ann. of Math. (2)}, 173(3):1645--1741, 2011.

\bibitem[Spi14]{Spi1}
Michael Spie\ss.
\newblock On special zeros of {$p$}-adic {$L$}-functions of {H}ilbert modular forms.
\newblock {\em Invent. Math.}, 196(1):69--138, 2014.

\bibitem[SU14]{SU}
Christopher Skinner and Eric Urban.
\newblock The {I}wasawa {M}ain {C}onjectures for {$GL_2$}.
\newblock {\em Invent. Math.}, 195(1):1--277, 2014.

\bibitem[Wei83]{WeiVek}
Rainer Weissauer.
\newblock Vektorwertige {S}iegelsche {M}odulformen kleinen {G}ewichtes.
\newblock {\em J. Reine Angew. Math.}, 343:184--202, 1983.

\end{thebibliography}
\end{document}